\documentclass[11pt]{amsart}
    
\usepackage{amsmath,amssymb,amsthm,amsfonts}
\usepackage[margin=1in,marginparwidth=0.8in, marginparsep=0.1in]{geometry}
\usepackage{graphicx}
\usepackage[all]{xy}
\usepackage[colorlinks, linkcolor=blue, anchorcolor=blue,
            citecolor=blue]{hyperref}
\usepackage{cleveref}
\usepackage{enumerate}
\usepackage{tikz}
\usetikzlibrary{arrows.meta}  
\usepackage{tikz-cd}
\usepackage{xcolor,float}
\usepackage{mathrsfs}
\usepackage{caption}
\usepackage{comment}
\usepackage{mathtools}

\usepackage[
backend=biber,
style=ext-numeric,
articlein=false,
url=false,  
doi=false,
eprint=true,
isbn=false
]{biblatex}
\usepackage{fancyhdr}
\usepackage{mathrsfs}
\theoremstyle{plain}
\newtheorem{theorem}{Theorem}[section]
\newtheorem{proposition}[theorem]{Proposition}
\newtheorem{lemma}[theorem]{Lemma}
\newtheorem{corollary}[theorem]{Corollary}
\crefname{theorem}{Theorem}{Theorems}
\crefname{proposition}{Proposition}{Propositions}

\theoremstyle{definition}
\newtheorem{definition}[theorem]{Definition}
\newtheorem{example}[theorem]{Example}
\newtheorem{notation}[theorem]{Notation}
\newtheorem{remark}[theorem]{Remark}
\newtheorem{conjecture}[theorem]{Conjecture}

\newtheorem{warning}[theorem]{Warning}

\numberwithin{equation}{section}

\newcommand{\R}{\mathbb{R}}

\newcommand{\SC}{\mathcal{C}}

\newcommand{\Aut}{\operatorname{Aut}}

\newcommand{\Cont}{\operatorname{Cont}}

\newcommand{\Hom}{\operatorname{Hom}}
\newcommand{\Ind}{\operatorname{Ind}}

\newcommand{\Loc}{\operatorname{Loc}}

\DeclareMathOperator{\Mod}{{\operatorname{Mod}}}
\newcommand{\msh}{{\operatorname{\mu sh}}}

\newcommand{\ol}{\overline}

\newcommand{\PrLst}{ {\operatorname{Pr}^{\mathrm{L}}_{st}} }

\newcommand{\Sh}{\operatorname{Sh}}

\newcommand{\supp}{\operatorname{supp}}

\newcommand{\cSS}{\operatorname{SS}}
\newcommand{\cSSif}{\cSS^\infty}

\newcommand{\id}{ {\operatorname{id}} }

\newcommand{\bN}{\mathbb{N}}

\newcommand{\bR}{\mathbb{R}}
\newcommand{\bZ}{\mathbb{Z}}

\newcommand{\bS}{\mathbb{S}}

\newcommand{\cC}{\mathcal{C}}
\newcommand{\cD}{\mathcal{D}}

\newcommand{\cF}{\mathcal{F}}

\newcommand{\cP}{\mathcal{P}}

\newcommand{\cS}{\mathcal{S}}

\newcommand{\cU}{\mathcal{U}}

\DeclareMathOperator{\colim}{colim}
\newcommand\clmi[1]{  \underset{ {#1} }{\operatorname{colim}} }
\newcommand\lmi[1]{  \underset{ {#1} }{\lim} }

\newcommand{\dT}{\dot{T}} 
\newcommand{\HS}{\mathrm{HS}}

\begin{document}

\title{$\boldsymbol{C^0}$-Rigidity of Legendrians and Coisotropics via Sheaf Quantization}
\date{}
\author{Tomohiro Asano \and Yuichi Ike \and Christopher Kuo \and Wenyuan Li}

\address{Department of Mathematics, Kyoto University, Kitashirakawa-Oiwake-Cho, Sakyo-ku, 606-8502, Kyoto, Japan}
\email{tasano[at]math.kyoto-u.ac.jp}

\address{Graduate School of Mathematical Sciences, The University of Tokyo, 3-8-1 Komaba Meguro-ku Tokyo 153-8914, Japan}
\email{ike[at]ms.u-tokyo.ac.jp}

\address{Max Planck Institute for Mathematics, Vivatsgasse 7, 53111 Bonn, Germany}
\email{chrislpkuo[at]berkeley.edu }

\address{Department of Mathematics, University of Southern California, 3551 Trousdale Parkway, Los Angeles, CA 90089, USA}
\email{wenyuan.li[at]usc.edu}

\begin{abstract}
    We prove that in the standard cosphere bundle, for any contact homeomorphism in the closure of the compactly supported contactomorphism group, when the image of a coisotropic submanifold (not necessarily properly embedded) is smooth, it is still coisotropic. Moreover, when contactomorphisms in the sequence are in the identity component and the image of a Legendrian is smooth, the Maslov data is preserved, and the category of sheaves with singular support on the Legendrian and the microstalk corepresentative are also preserved (and thus so is the wrapped Floer cochains of the linking disks). The main ingredients are the result of Guillermou--Viterbo, a new sheaf quantization result for $C^0$-small contactomorphisms (not necessarily in the identity component) different from Guillermou--Kashiwara--Schapira, and continuity of the interleaving distance of sheaves with respect to the Hofer--Shelukhin distance and the $C^0$-distance. The appendix contains different arguments for local $C^0$-limits and certain Hausdorff limits of Legendrians without appealing to the interleaving distance.
\end{abstract}

\maketitle

\tableofcontents 

\section{Introduction}

\subsection{Context and background}
Our objective in this paper is to show the rigidity of contact topology and dynamics in cosphere bundles of a manifold under the $C^0$-topology using microlocal theory of sheaves. Using the sheaf-theoretic interleaving distance, we prove a number of new rigidity results on Legendrian submanifolds.

One of the central topics in symplectic and contact topology is to understand the dichotomy between flexibility and rigidity. Gromov \cite{GromovPDR} and Eliashberg \cite{Eliashberg87} showed that for a symplectic manifold, the symplectomorphism group is closed in the diffeomorphism group under the $C^0$-topology (see also \cite{HoferZehnder}), which initiates the study of $C^0$-symplectic topology. Since then, people have proved various results on the rigidity of Lagrangians and coisotropics under symplectic homeomorphisms \cite{LaudenbachSikorav94,HumiliereLeclercqSeyfaddinni15}.

Following the idea of Gromov and Eliashberg, M\"uller--Spaeth showed that the contactomorphism group is also closed in the diffeomorphism group under the $C^0$-topology \cite{MullerSpaeth14}, so we can define contact homeomorphisms as $C^0$-limits of contactomorphisms accordingly. While contactomorphisms always lift to symplectomorphisms of the symplectizations, it is not true that the $C^0$-convergence of contactomorphisms will imply $C^0$-convergence of the corresponding symplectomorphisms, making the studies of $C^0$-contact topology more subtle \cite{Usher21Conformal}. For the rigidity of Legendrians, Dimitroglou Rizell--Sullivan \cite{DRG24C0Knot,DRG24C0Legendrian}, after series of previous works by others \cite{RosenZhang20Dichotomy,Usher21Conformal,Nakamura20C0,Stokic22}, 
showed that the image of properly embedded Legendrians under contact homeomorphisms (that arise as limits of contactomorphisms supported in a given compact subset), if smooth, are still Legendrians. However, the coisotropic rigidity is only known assuming uniform lower bounds on the conformal factor by Usher \cite{Usher21Conformal}, after \cite{RosenZhang20Dichotomy}. On the contrary, on the flexbility side, it is also known that any smooth manifold with correct dimension can be $C^0$ approximated by embedded Legendrians \cite{CE12,MurphyLoose}.

The studies of $C^0$-behavior in symplectic and contact geometry are closely related to many other topics in symplectic and contact geometry. As suggested by Buhovsky--Opshtein \cite{BuhovskyOpshtein}, one natural perspective to study the nearby Lagrangian conjecture is to study the problem through the action of symplectic homeomorphisms in a Weinstein neighborhood (see also the recent work \cite{ChasseLeclercq}). This naturally generalizes to the setting of Legendrians (we know closed Legendrians in 1-jet bundles are not unique even within the same formal Legendrian embedding class, and thus the contact homeomorphism assumption becomes necessary in the Legendrian case).

Our approach is to understand the rigidity in $C^0$-contact topology using microlocal theory of sheaves, developed by Kashiwara--Schapira \cite{KS90}. Sheaf-theoretic techniques have been used in symplectic geometry since the work of Nadler--Zaslow \cite{NZ,Nad} and Tamarkin \cite{Tamarkin}, and was used by Guillermou to give a new proof of the Eliashberg--Gromov theorem \cite{guillermou2013gromov,Guillermou23}. More recently, Guillermou--Viterbo and the first two authors demonstrate the strength of microlocal theory of sheaves in $C^0$-symplectic topology \cite{GV24,AI24,AI24rectangular}. 

The results mentioned above cannot be directly applied to $C^0$-contact topology. In this paper, we will show some persistence distance estimations with respect to the Hofer--Shelukhin distance and the $C^0$-distance. For the estimation with respect to the $C^0$-distance, we will need to prove a new sheaf quantization construction for any $C^0$-small contactomorphism (potentially not compactly supported nor in the identity component) that is different from Guillermou--Kashiwara--Schapira \cite{GKS}. These results together with the previous ones \cite{GV24,AI24} will allow us to deduce new $C^0$-rigidity results of Legendrians and more generally coisotropics.

\subsection{Main results and applications}
Let $(Y, \xi)$ be a co-oriented contact manifold, where $\xi$ is a co-oriented maximally non-integrable hyperplane distribution given by the kernel of a 1-form $\alpha$. A contactomorphism is a diffeomorphism $\varphi \colon Y \to Y$ such that $\varphi_*\xi = \xi$, in other words, 
$\varphi^*\alpha = e^h\alpha$
for some smooth function $h \colon Y \to \bR$. We call $h$ the conformal factor of $\varphi$.
We denote the group of contactomorphisms by $\operatorname{Cont}(Y, \xi)$ and the path connected component of the identity map be $\operatorname{Cont}_0(Y, \xi)$. Fix a Riemannian metric on $Y$. We consider the uniform $C^0$-norm on the contactomorphism group $d_{C^0}(\varphi, \psi) = \sup_{y\in Y}d(\varphi(y), \psi(y)).$

Our first result shows the rigidity of any (not necessarily properly embedded) Legendrian and coisotropic submanifolds under contact homeomorphisms, in other words, $C^0$-limits of contactomorphisms in the group of homeomorphisms. The rigidity result generalizes the result of Dimitroglou Rizell--Sullivan \cite{DRG24C0Legendrian} for properly embedded Legendrians after the works of \cite{RosenZhang20Dichotomy,Usher21Conformal,Nakamura20C0,Stokic22}, and Usher \cite{Usher21Conformal} for coisotropics assuming uniform lower bound of the conformal factors after \cite{RosenZhang20Dichotomy}.\footnote{For non-properly embedded Legendrians, our result is not compatible with the recent preprint \cite{Stokic24}.} 
As the discussion on $C^0$-geometry relies on a certain non-degeneracy result, precisely \Cref{thm:non-deg_kernel}, we will assume that the metric on $S^* M$ is induced from a complete Riemannian metric with bounded geometry (in particular, with positive injectivity radius).

\begin{theorem}\label{thm:main-rigidity}
    Let $(S^*M, \xi_{\mathrm{std}})$ be the cosphere bundle with the standard contact structure and $C \subseteq S^*M$ be a locally closed (embedded) coisotropic. Consider contactomorphisms $\varphi_n \in \operatorname{Cont}(S^*M, \xi_{\mathrm{std}})$ each of which has bounded conformal factor $h_n$. Suppose $\varphi_n \to \varphi_\infty$ in the $C^0$-topology and $\varphi_\infty$ is a homeomorphism.
    If $\varphi_\infty(C)$ is smooth, then $\varphi_\infty(C)$ is also coisotropic.\footnote{We say that a submanifold is coisotropic in a contact manifold $(Y, \xi)$ if $TC \cap \xi$ is coisotropic in $\xi$, following Huang \cite{Huang15}, which is compatible with \cite{RosenZhang20Dichotomy,Usher21Conformal}. This is different from the notion of regular coisotropic in \cite{SerrailleStokic25}.}
\end{theorem}
\begin{remark}\label{rem:intro-stokic}
    We require only that the conformal factor $h_n$ of each contactomorphism $\varphi_n$ is bounded by some (non-uniform) constant $C_n$. For example, this always holds when each $\varphi_n$ is supported in a compact subset $K_n$. Thus, for proper Legendrians, our theorem is also slightly stronger than the result of Dimitroglou Rizell--Sullivan \cite{DRG24C0Legendrian}.
    When $\varphi_n$ are all supported in one given compact subset and $\Lambda$ is a properly embedded Legendrian, they showed that if $\varphi_\infty(\Lambda)$ is smooth then it is Legendrian \cite{DRG24C0Legendrian}. To the best of our knowledge, the rigidity result was also not known for non-properly embedded Legendrians in the literature.
\end{remark}
\begin{remark}
    Other than the argument we use in the main theorem, in \Cref{sec:appendix}, we will provide a straightforward argument independent of the main body of the paper (relying on only standard sheaf theory techniques), which proves local $C^0$-rigidity of Legendrians, that when $\varphi_n$ are all equal to the identity on a fixed open subset $\Omega$ and $\Lambda \cap \Omega = \varnothing$, then the result holds. The appendix will also contain results on certain local rigidity of Hausdorff limits of Legendrians.
\end{remark}

Our next result shows the rigidity of Maslov data of Legendrian submanifolds under contact homeomorphisms. This is the natural Legendrian analogue of the result that nearby Lagrangians have vanishing Maslov classes and some higher obstructions \cite{Kragh13Parametrized,AK16,Guillermou23,Jin20Jhomomorphism}. 

\begin{theorem}\label{thm:main-Maslov}
    Let $(S^*M, \xi_{\mathrm{std}})$ be the cosphere bundle with the standard contact structure and $\Lambda \subseteq S^*M$ be a locally closed (embedded) Legendrian. Consider contactomorphisms $\varphi_n \in \operatorname{Cont}_0(S^*M, \xi_{\mathrm{std}})$ each of which has bounded conformal factor $h_n$. Suppose $\varphi_n \to \varphi_\infty$ in the $C^0$-topology and $\varphi_\infty$ is a homeomorphism.
    When $\varphi_\infty(\Lambda)$ is smooth, the composition of Lagrangian Gauss map and the delooping of $J$-homomorphism remains the same:
        \[\begin{tikzcd}[row sep=scriptsize]
    \Lambda \ar[dr,dashed] \ar[drr] \ar[dd, "\varphi_\infty" left] & & \\
    & U/O \ar[r,dashed] & B\mathrm{Pic}(\mathbb{S}) \\
    \varphi_\infty(\Lambda) \ar[ur,dashed] \ar[urr] & &
    \end{tikzcd}\]
    In particular, $\Lambda$ and $\varphi_\infty(\Lambda)$ have the same Maslov class and relative second Stiefel--Whitney classes.
\end{theorem}

\begin{remark}
    When $\Lambda$ is a Legendrian knot in a contact 3-manifold, Dimitroglou Rizell--Sullivan showed that $\varphi_\infty(\Lambda)$ has the same Maslov class (and in fact is contactomorphic to $\Lambda$). However, to the best of our knowledge, the statements about Maslov classes in higher dimensions are not known.
    While the Lagrangian analogue of the Maslov class results are known due to Abouzaid--Kragh \cite{Kragh13Parametrized,AK16}, Guillermou \cite{Gu12,Guillermou23}, Jin \cite{Jin20Jhomomorphism} and Membrez--Opshtein \cite{MembrezOpshtein21}, they do not imply the Legendrian version of the results as we do not know whether the Legendrian has no short Reeb chords.
\end{remark}

One of the reasons that we can only show the above results for cosphere bundles is because we do not know whether it is possible to cut off a contactomorphism in a tubular neighborhood of a closed Legendrian without changing the $C^0$-distance from the identity.  
However, in general, we still conjecture the following:

\begin{conjecture}
    Let $\Lambda \subseteq (Y, \xi)$ be a Legendrian embedding. Consider contactomorphisms $\varphi_n \in \operatorname{Cont}_0(Y, \xi)$ such that $\varphi_n \to \varphi_\infty$ in the $C^0$-topology and $\varphi_\infty$ is a homeomorphism. When $\varphi_\infty(\Lambda)$ is smooth, its Lagrangian Gauss map remains the same:
    \[\begin{tikzcd}[row sep=4pt]
    \Lambda \ar[dr] \ar[dd, "\varphi_\infty" left] & \\
    & U/O \\
    \varphi_\infty(\Lambda) \ar[ur] & 
    \end{tikzcd}\]
\end{conjecture}

Our second result shows the rigidity of the category of sheaves with singular support on Legendrian submanifolds in the cosphere bundle of a manifold $M$ equipped with the natural contact structure $(S^*M, \xi_{\mathrm{std}})$ under contact homeomorphisms when the conformal factors are uniformly bounded. Using the result of Ganatra--Pardon--Shende \cite{GPS3}, this implies the rigidity of the partially wrapped Fukaya category. When further combining with the Legendrian surgery formula \cite{BEE,EkholmLekili,Ekholm19Surgery,AsplundEkholm}, this implies the rigidity of the Legendrian contact homology. This is a natural Legendrian analogue of the results that nearby Lagrangians define the same object as the zero section in the sheaf category or Fukaya category \cite{AbouzaidNearby,Nad,AK16,Guillermou23,Jin20Jhomomorphism,AsplundDeshmukhPieloch}.

\begin{theorem}\label{thm:main-sheaf-invariance}
    Let $(S^*M, \xi_{\mathrm{std}})$ be the cosphere bundle with the standard contact structure and $\Lambda \subseteq S^*M$ be a proper Legendrian embedding. Consider contactomorphisms $\varphi_n \in \operatorname{Cont}_0(Y, \xi)$ each of which has bounded conformal factor $h_n$. Suppose $\varphi_n \to \varphi_\infty$ in the $C^0$-topology and $\varphi_\infty$ is a homeomorphism. Then there exists a functor
    \[
        K^{\varphi_\infty} \colon \Sh_{\Lambda}(M) \to \Sh_{\varphi_\infty(\Lambda)}(M).
    \]
    Furthermore, $K^{\varphi_\infty^{-1}}$ is the inverse of $K^{\varphi_\infty}$ and it preserves microstalks. \footnote{The $C^0$-distance is equivalent to its two-sided counterpart $\bar{d}_{C^0}(\varphi,\psi) = \sup_{y \in Y}d\left(\varphi(y),\psi(y)\right) + \sup_{y \in Y}d\left(\varphi^{-1}(y),\psi^{-1}(y)\right)$, and thus if $\varphi_\infty$ is a contact homeomorphism, then $\varphi_\infty^{-1}$ is a also a contact homeomorphism.}
\end{theorem}
\begin{remark}
    By Ganatra--Pardon--Shende \cite{GPS3}, we know that there is an equivalence between sheaves on manifolds with singular supports and partially wrapped Fukaya categories of cotangent bundles with stops (over a discrete ring)
    \[
        \Sh_\Lambda(M) \simeq \Mod\mathcal{W}(T^*M, \Lambda)^{op}
    \]
    which sends the corepresentatives of microstalk functors to the linking disks. Therefore, as our equivalence preserves microstalks, we can conclude that there is a quasi-isomorphism between the self wrapped Floer cochains of linking disks $D_\Lambda$. Then, by the Legendrian surgery formula \cite{EkholmLekili,AsplundEkholm}, we know they are isomorphic to Legendrian contact homologies
    \[
        CW^*(D_\Lambda, D_\Lambda) \simeq \mathcal{A}_{C_{-*}(\Omega_*\Lambda)}(\Lambda).
    \]
    Hence there is a quasi-equivalence between the Legendrian contact homologies with coefficients in based loop spaces. Moreover, since microlocal rank corresponds to the rank of the representation of Legendrian contact homologies, we know that there is an equivalence between augmentations of Legendrian contact homologies (see also \cite{NRSSZ20,CNS19}). 
\end{remark}

\begin{remark}
    For contact 3-manifolds, it is shown by Dimitroglou Rizell--Sullivan \cite{DRG24C0Knot} that for a closed Legendrian $\Lambda$, when $\varphi_\infty(\Lambda)$ is also smooth, there exists an ambient contactomorphism that sends $\Lambda$ to $\varphi_\infty(\Lambda)$. In particular, this implies that they have quasi-isomorphic Legendrian contact homology when it is defined. However, no such result is known in higher dimensions.
\end{remark}

Consequently, we can strengthen the result of Dimitroglou Rizell--Sullivan \cite{DRS20Persistence} that the Legendrian contact homology of $\Lambda$ is nontrivial if and only if the Legendrian contact homology of $\varphi_\infty(\Lambda)$ is nontrivial (since $K^{\varphi_\infty}$ sends non-local systems to non-local systems). For a loose Legendrian $\Lambda \subseteq S^*M$ in dimension at least 5 \cite{MurphyLoose}, the category of sheaves is trivial \cite{STZ17}, and so is the Legendrian contact homology \cite{EES05Nonisotopic}. In particular, as we do not know any example of non-loose Legendrian with trivial Legendrian contact homology or sheaf category in $J^1\mathbb{R}^n$, our theorem excludes the possibility of any known example of non-loose Legendrian to be the $C^0$-limit of a loose Legendrian. We therefore conjecture the following:

\begin{conjecture}
    Let $\Lambda \subseteq S^*M$ be a proper loose Legendrian embedding. Consider contactomorphisms $\varphi_n \in \operatorname{Cont}_0(S^*M, \xi_{\mathrm{std}})$ such that $\varphi_n \to \varphi_\infty$ in the $C^0$-topology and $\varphi_\infty$ is a homeomorphism. When $\varphi_\infty(\Lambda)$ is also smooth, then $\varphi_\infty(\Lambda)$ is still a loose Legendrian.
\end{conjecture}

Other than the $C^0$-topology on the group of contactomorphisms $\operatorname{Cont}(Y, \xi)$, there has also been a number of studies on the topology on the identity component of the comtactomorphism group ${\operatorname{Cont}}_0(Y, \xi)$ induced by the Hofer--Shelukhin norm \cite{Shelukhin17} with respect to a fixed contact form $\alpha$. For $\varphi \in {\operatorname{Cont}}_0(Y, \xi)$, consider a path $\varphi_t$ from $\id$ to $\varphi$. Then $\varphi_t$ is induced by a contact Hamiltonian $H \colon Y \times \bR \to \bR$. We then define the Hofer--Shelukhin norm to be
\[
    d_{\mathrm{HS},\alpha}(\varphi, \psi) = \inf_{\varphi = \varphi^H_1 \circ \psi} \|H\|_{\mathrm{HS},\alpha} = \inf_{\varphi = \varphi^H_1 \circ \psi} \int_0^1 \sup_{y \in Y}|H(y, t)|\, dt.
\]
Since its universal cover, under the $C^1$-topology, $\widetilde{\operatorname{Cont}}_0(Y, \xi)$ can be realized as the group of homotopy classes of contact isotopies, we also abuse notations and define ${d}_{\mathrm{HS},\alpha}(\varphi, \psi)$ by considering Hamiltonian isotopies in a fixed homotopy class of paths.


Similarly, for a Legendrian $\Lambda \subseteq Y$, on the space $\operatorname{Leg}_0(\Lambda)$ of Legendrian submanifolds isotopic to $\Lambda$, we can define the Chekanov--Hofer--Shelukhin distance to be
\[
d_{\mathrm{CHS},\alpha}(\Lambda_0, \Lambda_1) = \inf_{\Lambda_1 = \varphi^H_1(\Lambda_0)} \|H\|_{\mathrm{CHS},\Lambda,\alpha} = \inf_{\Lambda_1 = \varphi^H_1(\Lambda_0)} \int_0^1 \sup_{y \in \varphi_H^t(\Lambda)}|H(y, t)|\, dt.
\]
Then we can recover the following simple cases of non-degeneracy results of the (Chekanov--)Hofer--Shelukhin distance in cosphere bundles $(S^*M, \xi_{\mathrm{std}})$. The Hofer--Shelukhin distance shares a closer relation to the interleaving distance, as explained in \Cref{lem: HS-norm}. As a result, we do not need to assume a complete Riemannian metric for this part of the discussion (we only need to assume that the Reeb flow is well-defined).

\begin{theorem}\label{thm:main-Hofer-nondeg}
    Let $(S^*M, \xi_{\mathrm{std}})$ be the cosphere bundle equipped with the standard contact structure. Then for any contact form $\alpha$,
    \begin{enumerate}
        \item (Shelukhin \cite{Shelukhin17}) the Hofer--Shelukhin distance $d_{\mathrm{HS},\alpha}$ defines a non-degenerate metric on the space $\operatorname{Cont}_0(S^*M, \xi_{\mathrm{std}})$;
        \item (Dimitroglou Rizell--Sullivan \cite{DRG24Rabinowitz}) when $\Lambda \subseteq S^*M$ is a closed Legendrian such that $\Sh_\Lambda(M)$ is non-trivial, the Chekanov--Hofer--Shelukhin distance $d_{\mathrm{CHS},\Lambda,\alpha}$ defines a non-degenerate metric on the space $\operatorname{Leg}_0(\Lambda)$.
    \end{enumerate}
\end{theorem}
\begin{remark}
    Rosen--Zhang \cite{RosenZhang20Dichotomy} showed that the Chekanov--Hofer--Shelukhin distance is either zero or non-degenerate. Therefore, we only need to show that the distance is not identically zero. Dimitroglou Rizell--Sullivan \cite{DRG24Rabinowitz} showed that the distance is non-degenerate on all closed contact manifolds. On the contrary, Cant \cite{Cant23Shelukhin} and Nakamura \cite{Nakamura23Legendrian} showed that there exist open contact manifolds where the metric is zero. Similar to Dimitroglou Rizell--Sullivan \cite{DRG24Rabinowitz}, we expect that there is also a sheaf-theoretic proof for the non-degeneracy of general closed Legendrians $\Lambda \subseteq S^*M$ by considering the doubling construction as in \cite{Guillermou23,AI23}. However, there are nontrivial technical difficulties one needs to overcome and we choose not to go into this direction in this paper.
\end{remark}

We now explain the main techniques in the proof of the main theorems. To study the contact topology under limits with respect to the $C^0$-distance or Hofer--Shelukhin distance, we need to construct limits of sheaves under such limits. On the category of sheaves, one can define an interleaving pseudo-distance \cite{KS18persistent}. Guillermou--Viterbo \cite{GV24} and the first two authors \cite{AI24} recently showed that the category of sheaves is complete with respect to the pseudo-metric. 

For contact Hamiltonian isotopies on $S^*M$, there exist canonical sheaves in the product, known as the sheaf quantizations of the contact isotopies by Guillermou--Kashiwara--Schapira \cite{GKS}. We can thus define a variant of the interleaving distance $d_\tau$ on the product, depending on a positive Hamiltonian $\tau$ on $S^*M$, which induces a distance on the contact isotopies that is right invariant under contactomorphisms in the identity component (the conformal factor of contactomorphisms accounts for the failure of the distance being bi-invariant).

Our first technical result is that the interleaving distance is continuous with respect to the Hofer--Shelukhin distance. This generalizes the result of the first two authors for symplectic Hamiltonians $T^*M$ or contact Hamiltonians in $T^*M \times \bR$ \cite{AI20}.

\begin{theorem}\label{thm:main-Hofer-distance}
    Let $(S^*M, \xi_{\mathrm{std}})$ be the cosphere bundle equipped with the standard contact structure and $\varphi \in \widetilde{\operatorname{Cont}}_0(S^*M, \xi_{\mathrm{std}})$ be a homotopy class of contact isotopy.
    For $K^\varphi \in \Sh(M \times M)$ the time-$1$ GKS sheaf quantization of the contact isotopy, we have 
    \[
        d_\tau(1_\Delta, K^\varphi) \leq 2\,{d}_{\mathrm{HS},\alpha}(\id, \varphi).
    \]  
\end{theorem}

Our next (and perhaps most) technical theorem is that the interleaving distance is continuous with respect to the $C^0$-distance. 
This turns out to be a very subtle theorem. 

Since it is unknown whether a $C^0$-small contactomorphism in the identity component is connected to the identity through a $C^0$-small isotopy, it is hard to work with the sheaf quantization of contact isotopies by Guillermou--Kashiwara--Schapira and still get good control on the distance \cite{GKS}. Therefore, we will prove a separate sheaf quantization theorem for $C^0$-small contactomorphisms based on Guillermou's theorem on nearby Lagrangians \cite{Guillermou23} and then deduce the distance estimation. Our argument is also very different from the symplectic analogue in \cite{BHS2021}.

\begin{theorem}\label{thm:main-C0-distance}   
    Let $M$ be a complete Riemannian manifold, $(S^*M, \xi_{\mathrm{std}})$ be the cosphere bundle equipped with the standard contact structure. Let $\alpha$ be a contact form whose Reeb flow is defined by a positive Hamiltonian with positive lower bound and finite $C^1$-norm. Then there exist $\epsilon > 0$ and $C_\alpha > 0$ such that for a contactomorphism $\varphi \in \operatorname{Cont}(S^*M, \xi_{\mathrm{std}})$ with $d_{C^0}(\id, \varphi) < \epsilon$, there exists a canonical sheaf quantization $K^\varphi \in \Sh(M \times M)$ of $\varphi$ such that
    \[
        d_\tau(1_\Delta, K^\varphi) \leq 2C_\alpha \,d_{C^0}(\id, \varphi).
    \]
\end{theorem}
\begin{remark}
    Related results to \Cref{thm:main-Hofer-distance} on the continuity of spectral invariants with respect to the Hofer--Shelukhin norm have been proved \cite{DUZ23,Cant23Shelukhin}. One related result to \Cref{thm:main-C0-distance} is the local continuity of spectral invariants with respect to the $C^0$-distance in $W \times S^1$ \cite[Proposition 4.1]{SerrailleStojisavljevic}. We remark that we can similarly construct a discrete conjugation-invariant distance using sheaf-theoretic spectral invariants as Sandon \cite{Sandon} (note that all conjugation-invariant distance on $\operatorname{Cont}_0(S^*M, \xi_{\mathrm{std}})$ must be discrete \cite{BIP}). We expect that this distance is also lower semi-continuous with respect to the $C^0$-distance \cite[Proposition 1.10]{SerrailleStojisavljevic}.
\end{remark}

Considering sheaf quantizations of contactomorphisms in $S^*M$, the interleaving distance is right invariant but only left invariant up to scaling by the conformal factor of the contactomorphism. Therefore, we can construct sheaf quantizations for any contact homeomorphism in the identity component using \Cref{thm:main-C0-distance}, but they are not compatible with compositions (and thus not invertible) in general.

Finally, we mention that there have been other results in the sheaf theory literature that are related to $C^0$-contact topology as well. The full faithfulness theorem for nearby cycle functors of Nadler--Shende \cite[Theorem 5.1]{NadlerShende} also constructs sheaf quantizations for certain non-smooth objects in contact topology. However, one needs to start with a path of contactmorphisms to apply their theorem. In particular, when we only have a discrete sequence of contactomorphisms, one cannot directly apply their result, which is exactly the situation we need to deal with in the paper.

\subsection*{Acknowledgement}
We would like to thank Mohammed Abouzaid, Roger Casals, Sheng-Fu Chiu, Georgios Dimitroglou Rizell, St\'ephane Guillermou, Vincent Humili\`ere, R\'emi Leclercq, Emmy Murphy, Vivek Shende, Sobhan Seyfaddini, Michael Sullivan, Bingyu Zhang and Jun Zhang for helpful discussions. We would also like to thank Maksim Stoki\'c for helpful correspondence. 
T.~A. is partially supported by JSPS KAKENHI Grant Number JP24K16920. 
Y.~I. is partially supported by JSPS KAKENHI Grant Numbers JP22H05107 and JP25K17254.
T.~A. and Y.~I. are partially supported by JST, CREST Grant Number JPMJCR24Q1, Japan. W.~L. is partially supported by the AMS-Simons Travel Grant.
C.~K. is supported by Max Planck Institute for Mathematics in Bonn.

\section{Completeness of Sheaves}

\subsection{Sheaves and singular supports}
Fix a (small idempotent complete) rigid symmetric monoidal category $\SC_0$ and consider the compactly generated category\footnote{When $\SC$ is not compactly generated (even if it is dualizable presentable), the non-characteristic deformation lemma \cite[Proposition 2.7.2]{KS90} \cite{Robalo-Schapira} fails and the conditions of \cite[Proposition 5.1.1]{KS90} are no longer equivalent, as explained in \cite[Remark 4.24]{Efimov-K-theory}. However, \cite{Efimov-K-theory,Zhang25remark} suggest a definition of singular supports for sheaves with values in non-compactly generated coefficients using $\Omega$-lenses (defined below).} $\SC \coloneqq \Ind(\SC_0)$. For a smooth manifold $M$, we will use the notation $\Sh(M)$ to mean sheaves on $M$ with values in $\SC$. 

Microlocal sheaf theory developed in \cite{KS90} comes with the following ingredients: On a $C^\infty$ manifold $M$, one can assign, for any sheaf $F \in \Sh(M)$, its singular support or microsupport $\cSS(F) \subseteq T^* M$, which is a closed conic coisotropic subset, and, under some mild regularity condition, is Lagrangian if and only if $F$ is constructible.\footnote{We follow the convention in \cite{GPS3, Kuo23, KuoLi22} and do not require constructible sheaves to have perfect stalks, which is different from \cite{KS90} (where such sheaves are called weakly constructible sheaves).} 
For $F \in \Sh(M)$, we write $\cSS^\infty(F) \subset S^*M$ for the closed subset corresponding to $\cSS(F)$. Moreover, for a closed subset $\Lambda$, we write $\Sh_\Lambda(M)$ for the full subcategory of $\Sh(M)$ spanned by sheaves $F$ such that $\cSS^\infty(F) \subset \Lambda$. 

We will use the characterization of singular supports using $\Omega$-lenses \cite[Definition 3.1]{GV24} \cite[Section 2.7.2]{JT17}.\footnote{From the viewpoint of \cite{GPS3}, when $\Omega$ is a contractible neighborhood around a point, $\Omega$-lenses correspond to (sheaf-theoretic) linking disks around that point.}
We set $\dT^*M \coloneqq T^*M \setminus 0_M$, where $0_M$ denotes the zero section of $T^*M$. For a conic open set $\Omega \subseteq \dot T^*M$, an $\Omega$-lens is a locally closed subset $\Sigma \subset M$ with the following properties: $\overline{\Sigma}$ is compact and there exists an open neighborhood $U$ of $\overline{\Sigma}$ and a family of $C^1$-functions $f \colon U \times [0, 1] \to \bR$ such that 
\begin{enumerate}
    \item $df_t(x) \in \Omega$ for any $(x,t) \in U \times [0,1]$, where $f_t \coloneqq f|_{U \times \{t\}}$;
    \item $f_t^{-1}((-\infty,0)) \subseteq f_{t'}^{-1}((-\infty,0))$ for $t \le t'$;
    \item the hypersurfaces $f_t^{-1}(0)$ coincide on $U \setminus \overline{\Sigma}$;
    \item $\Sigma = f_1^{-1}((-\infty,0)) \setminus f_0^{-1}((-\infty,0))$.
\end{enumerate}

\begin{lemma}[{\cite[Lemma 3.2]{GV24}}]\label{lem: omega-lenses-test}
Let $F \in \Sh(M)$ and $\Omega \subseteq \dT^* M$ be a conic open set. Then $\cSS(F) \cap \Omega = \varnothing$ if and only if $\Hom(1_\Sigma, F) = 0$ for all $\Omega$-lenses $\Sigma$.
\end{lemma} 

For sheaves $F_{12} \in \Sh(M_1 \times M_2)$ and $F_{23} \in \Sh(M_2 \times M_3)$, we can consider their convolution 
$F_{23} \circ F_{12} \in \Sh(M_1 \times M_3)$ defined by
\[
F_{23} \circ F_{12} \coloneqq \pi_{13!}(\pi_{23}^*F_{23} \otimes \pi_{12}^*F_{12}),
\]
where $\pi_{ij} \colon M_1 \times M_2 \times M_3 \to M_i \times M_j$ is the projection map.\footnote{We follow the convention in \cite{Kuo23,Kuo-Li-duality} to write $F_{23} \circ F_{12}$, which is different from \cite{KS90,GKS} (where they use $F_{12} \circ F_{23}$). This also results in a difference between our definition of graphs with the one in \cite{GKS}.}
We also define the convolution of sets as follows. 
We let $p_{ij} \colon T^*(M_1 \times M_2 \times M_3) \to T^*(M_i \times M_j)$ denote the projection.  
For $\Lambda_{12} \subseteq T^*(M_1 \times M_2)$ and $\Lambda_{23} \subseteq T^*(M_2 \times M_3)$, we set  
\[
    \Lambda_{23} \circ \Lambda_{12} 
    \coloneqq 
    p_{13}(p_{23}^{-1}(\Lambda_{23})^{a_2} \cap p_{12}^{-1} \Lambda_{12}) \subseteq T^*(M_1 \times M_3)
\]
where $(-)^{a_2} \colon T^*(M_2 \times M_3) \rightarrow T^*(M_2 \times M_3)$ is the antipodal map on the $M_2$-component.
The convolution satisfies the following singular support estimation:

\begin{lemma}[{\cite[Equation (1.12)]{GKS}}] \label{lem: microsupport-convolution-composition}
    Let $F_{12} \in \Sh(M_1 \times M_2)$ and $F_{23} \in \Sh(M_2 \times M_3)$.
    Assume that $\pi_{13}$ is proper on $\pi_{23}^{-1} \supp(F_{23}) \cap \pi_{12}^{-1} \supp(F_{12})$ and 
    \[
        p_{23}^{-1}(\cSS(F_{23})^{a_2}) \cap p_{12}^{-1} \cSS(F_{12}) \cap (0_{M_1} \times T^*M_2 \times 0_{M_3}) \subseteq 0_{M_1 \times M_2 \times M_3}.
    \]
    Then, we have $\cSS(F_{23} \circ F_{12}) \subseteq \cSS(F_{23}) \circ \cSS(F_{12})$. 
\end{lemma}

Furthermore, the GKS sheaf quantization \cite{GKS} allows contact isotopies to act on the category of sheaves via convolutions:

\begin{notation}\label{not:graph}
    We denote a contactomorphism by $\varphi \colon S^*M \to S^*M$, and by abuse of notations, its associated homogeneous symplectomorphism by $\varphi \colon \dot T^*M \to \dot T^*M$, whose graph in $S^*(M \times M)$ or $\dot T^*(M \times M)$ is
    \[\Gamma^\varphi = \{((x, -\xi), \varphi(x, \xi)) \mid (x, \xi) \in \dot T^*M \}.\]
    We write $\Phi \colon S^*M \times U \to S^*M$ for a $U$-family of contact isotopies and $\varphi_t = \Phi|_{S^*M \times \{t\}}$ the time-$t$ map and $H \colon S^*M \times U \to \bR$ the associated Hamiltonian where $H_t \circ \varphi_t(x, \xi) = \alpha(\partial_t\varphi_t(x, \xi))$. We also abuse notations and write $\Phi$ and $H$ for the associated homogeneous Hamiltonian isotopy $\Phi \colon \dot T^*M \times U \to \dot T^*M$ and $H \colon \dot T^*M \times U \to \bR$. We denote the graph of the homogeneous Hamiltonian diffeomorphism at time-$t$ by $\Gamma^\Phi_t$ and sometimes $\Gamma^H_t$ in $S^*(M \times M)$ or $\dot T^*(M \times M)$ to be
    \[\Gamma^\Phi_t = \{((x, -\xi), \varphi_t(x, \xi)) \mid (x, \xi) \in \dot T^*M \}.\]
\end{notation}

\begin{theorem}[{\cite[Proposition 3.2, Theorem 3.7, \& Remark 3.9]{GKS}}] \label{thm: GKS}
    For any contact isotopy $\Phi=(\varphi_t)_{t \in U} \colon S^* M \times U \rightarrow S^* M$ where $U$ is a contractible manifold containing $0$, there exists a unique sheaf kernel $K^\Phi \in \Sh(M \times M \times U)$ such that for the inclusion $i_0 \colon M \times M \times \{0\} \hookrightarrow M \times M \times U$, $i_0^*K^{\Phi} = 1_\Delta$ and
    \begin{equation} \label{for: ms-isotopy}
    \cSS^\infty(K^{\Phi}) \subseteq \left\{( x, - \xi, \varphi_t(x,\xi), t, -H_t \circ \varphi_t(x,\xi)) \mid (x, \xi) \in S^* M, t \in U \right\},
    \end{equation} 
    where $H \colon S^*M \times U \to \bR$ is the associated contact Hamiltonian. 
\end{theorem} 

\begin{remark}\label{rem:GKS}
    (1)~Let $U = I$ be a closed interval containing $0$, we can construct a sheaf quantization for any 1-parametric contact isotopy such that $i_0^*K^{\Phi} = 1_\Delta$ and 
    \[
    \cSS^\infty(K^{\Phi}) \subseteq \left\{( x, - \xi, \varphi_t(x,\xi), t, -H_t \circ \varphi(x,\xi)) \mid (x, \xi) \in S^* M, t \in I \right\}.
    \] 
    (2)~Let $U = I \times J$ be a product of closed intervals containing $0$, we can construct a sheaf quantization for any 2-parametric contact isotopy. When there is a contactomorphism $\varphi = \varphi_{1,s}$ for any $s \in J$, we know that
    \[
    \cSS^\infty(i_{1\times J}^*K^{\Phi}) \subseteq \left\{( x, - \xi, \varphi(x,\xi), s, 0) \mid (x, \xi) \in S^* M, s \in J \right\}.
    \]
    Then by \cite[Corollary 1.6]{GKS}, there exists a well-defined sheaf kernel $K^\varphi = i_{1,s}^*K^{\Phi} \in \Sh(M \times M)$ such that 
    \[
    \cSS^\infty(K^\varphi) \subseteq \left\{( x, -\xi, \varphi(x,\xi)) \mid (x, \xi) \in S^* M \right\}.
    \]
    Therefore, we have a canonical sheaf quantization for any homotopy class of contact isotopies.
\end{remark}

\begin{notation}
For a contact isotopy $\Phi \colon S^*M \times I \to S^*M$ with $I$ being a closed interval, we simply write $K^{\Phi}_t \coloneqq i_t^*K^{\Phi}$ for $t \in I$. 
When considering the contact isotopy $\Phi$ induced by the contact Hamiltonian $H$, we also use the notations $K^H$ and $K^H_t$ to denote $K^\Phi$ and respectively $K^{\Phi}_t$.
\end{notation}

The assignment $\Phi \mapsto K^{\Phi}$ matches compositions of isotopies with convolutions of sheaves because of Equation~\eqref{for: ms-isotopy}. That is, if $\Phi$ and $\Psi$ are two contact isotopies, then $K^{\Phi} \circ|_U K^\Psi = K^{\Phi \circ \Psi}$. In particular, $K^{\Phi}$ is invertible with respect to $\circ|_U$, and it defines, a $U$-family of auto-equivalence on $\Sh(M)$. Explicitly, composing $F \in \Sh(M)$ with $K^{\Phi}$ gives a $U$-family sheaf $K^{\Phi} \circ F$ which satisfies the microsupport estimation
\begin{equation} \label{for: ms-isotopy-object}
\cSSif(K^{\Phi} \circ F) \subseteq \left\{\left( \varphi_t(x,\xi), t, -H_t \circ\varphi_t(x,\xi) \right) \mid (x, \xi) \in \cSSif(F), t \in U \right\}.
\end{equation}
Similarly, in the case of a sheaf kernel $L \in \Sh(M \times M)$, we have 
\begin{equation} \label{for: ms-isotopy-endofunctor}
\cSSif(K^{\Phi} \circ L) \subseteq \left\{\left( (x,\xi), \varphi_t(y,\eta), t, -H_t \circ \varphi_t(x,\xi) \right) \mid (x, \xi, y, \eta) \in \cSSif(L), t \in U \right\}.
\end{equation}

One structure which comes out from this machinery is the continuation map \cite[Section 3]{Kuo23}. Let $\sigma$ be the fiber coordinate for $T^*I$. The subcategory of $\Sh(M \times I)$ which consists of objects $F$ such that $\cSS(F) \subseteq \{\sigma \leq 0\}$. Being a subcategory closed under limits, it admits a surjective left adjoint and it can be given explicitly by
\[
1_{ \{ (t,s) \mid s > t\}}[1] \circ (-) \colon \Sh(M \times I) \rightarrow \Sh_{\sigma \leq 0}(M \times I),
\]
where $1_{ \{ (t,s) \mid s > t\}}$ convolves on the $I$-direction. Thus, for $F \in \Sh(M \times I)$, $F$ is in $\Sh_{\sigma \leq 0}(M \times I)$ if and only if the canonical map $F = 1_{ \{ s = t\} } \circ F \xrightarrow{\sim} 1_{ \{ (t,s) \mid s > t\}}[1] \circ F$ is an isomorphism. In this case, for $s \in I$ and the inclusion $i_s \colon M \times \{s\} \hookrightarrow M \times I$, we have $i_s^*F = 1_{(-\infty, s)} \circ F[1]$.

\begin{definition} \label{def: conti-map}
For $s \leq t$ in $I$ and $F \in \Sh_{\sigma \leq 0}(M \times I)$, the \emph{continuation map}
\[
c(H, s, t) \colon i_s^*F \rightarrow i_t^*F
\]
is the map induced by $1_{(-\infty, s)} \circ F[1] \xrightarrow{c_{s,t} \circ \id_F [1]} 1_{(-\infty, t)} \circ F[1]$, where 
$c_{s,t} \colon 1_{(-\infty, s)} \rightarrow 1_{(-\infty, t)}$ is the universal continuation map coming from the inclusion $(-\infty,s) \subseteq (-\infty, t)$.
\end{definition}

\begin{proposition}[{\cite[Proposition 3.9]{Kuo23}}]\label{prop:continuation} 

Let $F \in \Sh(M)$ be a sheaf and $\Phi \colon S^*M \times I \to S^*M$ be a contact isotopy that is positive on $\cSS^\infty(F)$. Then there exists a canonical continuation morphism that only depends on the homotopy class of the isotopy and the time-1 map
    \[
    c(\Phi, F) \colon F \to K^\Phi_1 \circ F.
    \]
\end{proposition}
\begin{remark}\label{rem:continuation}
    For any positive contact isotopy $\Phi \colon S^*M \times I \to S^*M$, \cite[Proposition 3.33]{Kuo23} implies that there exists a continuation morphism
    \[
    c(\Phi, F) \colon F \to K^{\Phi}_1 \circ F.
    \]
    For a $J$-family of positive contact isotopies $\Phi \colon S^*M \times I \times J \to S^*M$, when $\varphi = \varphi_{1,s}$ for any $s \in J$, by \cite[Proposition 3.9]{Kuo23}, we know that $c(\Phi_s, F)$ for different $s \in J$ is canonically identified and thus does not depend on the parameter, so we have a canonical continuation morphism.
\end{remark}

Consider also the situation where we have two isotopy $\Phi, \Psi \colon S^* M \times I \rightarrow S^*M$. Denote by $H_\Phi$ and $H_\Psi$ their corresponding Hamiltonians. Then $\Phi^{-1}$ is generated by the Hamiltonian $\overline{H}_\Phi = -H_\Phi(\varphi_t(p), t)$. If $H_\Phi \geq H_\Psi$, then the composition $\Psi^{-1} \circ \Phi$ will be a positive isotopy since it is generated by $(\overline{H}_\Psi \sharp H_\Phi)(p,t) \coloneqq -H_\Psi(\psi_t(p),t) + H_{\Phi}(\psi_t(p),t)$ and there is a continuation map 
\begin{equation}\label{eq: order-gives-continuation-map-pre}
    1_\Delta \to K^{\Psi^{-1} \circ \Phi}_1 =  K^{\Psi^{-1}}_1 \circ K^{\Phi}_1.
\end{equation}

\begin{definition}
    We refer to the morphism $K^\Psi_1 \to K^{\Phi}_1$ that corresponds to \eqref{eq: order-gives-continuation-map-pre} under the equivalence $K^\Psi_1 \circ (-) \colon \Sh(M \times M) \xrightarrow{\sim} \Sh(M \times M)$ as the continuation map induced from $\Phi \geq \Psi$.
\end{definition}

Let $\tau \colon S^*M \to \bR$ be a smooth function and define $\Sh_{\tau \geq 0}(M)$ to be the full subcategory consisting of objects $F$ with $\cSSif(F) \subseteq \{ \tau \ge 0 \}$. By \eqref{for: ms-isotopy-object}, $\cSS(K^\tau \circ F) \subseteq \{\sigma \leq 0\}$ and thus has a notion of continuation map.

\begin{notation} \label{not: GKS-kernels}
For $F \in \Sh_{\tau \geq 0}(M)$, we denote by $c_t(F,\tau) \colon F \rightarrow K^{\tau}_t \circ F$ the continuation map associated to $F$ and $\tau$. When the exact sheaf $F$ and Hamiltonian $\tau$ are clear from the context or not crucial, we simply denote it by $c_t$ and $c \coloneqq c_1$ for time-1 continuation maps. Similarly, for a morphism $u \colon G \rightarrow F$, we use the notation $u_t$ to mean the image of $u$ under the equivalence
\[
    K^{\tau}_t \circ (-) \colon \Hom(G,F) \xrightarrow{\sim} \Hom(K^{\tau}_t \circ G, K^{\tau}_t \circ F).
\]
 \end{notation}

\begin{remark} \label{rmk: compatibility-morphisms}
The wrappings the morphisms are compatible with continuation maps. More precisely, with previous notation, we note that since the continuation maps between $1_{(-\infty, t)} \circ \cF[1]$ are induced from the universal continuation maps, as recalled in \Cref{def: conti-map}, we have $c_t(F) \circ u = u_t \circ c_t(G)$. That is, the diagram
\[
\begin{tikzpicture}
\node at (0,2) {$G$};
\node at (4,2) {$F$};
\node at (0,0) {$G_t$};
\node at (4,0) {$F_t$};

\draw [->, thick] (0.4,2) -- (3.7,2) node [midway, above] {$u$};
\draw [->, thick] (0.4,0) -- (3.7,0) node [midway, above] {$u_t$};

\draw [->, thick] (0,1.7) -- (0,0.3) node [midway, left] {$c_t(G)$}; 
\draw [->, thick] (4,1.7) -- (4,0.3) node [midway, right] {$c_t(F)$};
\node[scale=1.5] at (2,1) {$\circlearrowleft$};
\end{tikzpicture}
\] 
commutes canonically since the convolution operation $\circ$ is bilinear.
\end{remark}

One of the main benefits of having continuation maps is that one can import the perturbation trick from Floer theory to microlocal sheaf theory. First, without any constructibility assumption, continuation maps respect colimit and limit with respect to the time direction.

\begin{lemma}[{\cite[Corollaries 3.4 \& 3.8]{Kuo23}}]\label{lem: continuation-(co)lim}
    Let $F \in \Sh_{\tau \geq 0}(M)$. Then the canonical maps
    \[
       K^{\tau}_t  \circ F  \rightarrow \lmi{s \rightarrow t^+} \left( K^{\tau}_s  \circ F \right) \quad \text{and} \quad \clmi{s \rightarrow t^-} \, \left( K^{\tau}_s \circ F \right) \rightarrow K^{\tau}_t  \circ F 
    \]
    are equivalences.
\end{lemma}
\begin{proof}
    The cited corollaries asserts that, as long as continuation map exists, in the general sense in \Cref{def: conti-map}, the map from the colimit is always an equivalence. To conclude the map to the limit is an equivalence, one has to check that the sheaf $K^{\tau} \circ F$ is $I$-non-characteristic, i.e., 
    \begin{align}\label{eq: non-characteristic}
        \cSS(K^{\tau} \circ F) \cap \left(0_M \times T_t^* I\right) \subseteq 0_{M \times I}, \quad \text{for} \ t \in I,
    \end{align}
    but this follows directly from \Cref{for: ms-isotopy-object}.
\end{proof}

Secondly, there is a deformation lemma which parallels the transverse intersection phenomenon from Floer theory \cite{Zhou-isotopy, Kuo23}. The usual setting is when $F \in \Sh(M)$ is a compactly supported sheaf with a Legendrian microsupport (at the infinity) $\Lambda$ such that $\Lambda$ is positively displaceable in the sense that there exists $\epsilon > 0$ such that $\varphi_t(\Lambda) \cap \Lambda = \varnothing$ for $0 < t < \epsilon$. For the purpose of this paper, we generalizes to the case of non-compactly supported sheaves. We will fix a complete Riemannian metric on the manifold $M$. This determines a complete Riemannian metric on $S^*M$ and a standard contact form on $S^*M$ which is bounded with respect to the metric.

\begin{definition}\label{def: bound-reeb}
Let $M$ be a complete Riemannian manifold and $\Phi \colon S^* M \times I \rightarrow S^* M$ be a contact isotopy. Then $\Phi$ is called a \emph{bounded contact isotopy} if the contact Hamiltonian $H_\Phi$ with respect to the standard contact form induced by the Riemannian metric is uniformly $C^0$-bounded below and above and $C^1$-bounded above by some positive constants.
\end{definition}

\begin{remark}
On the cosphere bundle $S^*M$, there always exist a bounded positive contact isotopy. For example, one can consider a complete Riemannian manifold with positive injectivity radius and take the geodesic flow associated to the metric. This is defined by the constant Hamiltonian, and the lengths of closed orbits have a positive lower bound by the injectivity radius.
\end{remark}

\begin{lemma}[{\cite[Lemma 2.10]{Zhou-isotopy}, \cite[Proposition 3.18]{Kuo23}}]\label{lem: pert} 
Let $F$ and $G \in \Sh_\Lambda(M)$.
Let $\tau \colon S^* M \rightarrow \bR$ be a bounded positive contact Hamiltonian such that
$\varphi_t^\tau(\Lambda) \cap \Lambda = \varnothing$ for $0 < t < \epsilon$.
Then for any $t > 0$, the continuation map $F \rightarrow K^\tau_t \circ F$ induces an isomorphism
\[
    \Hom(G,F) \xrightarrow{\sim} \Hom(G, K^\tau_t \circ F).
\]
\end{lemma}

\begin{proof} 
We follow the proof of \cite[Section 4.1]{LiEstimate}. We know by \Cref{lem: continuation-(co)lim} that
\[
    \Hom(G, F) = \lim_{t \to 0^+} \Hom(G, K^\tau_t \circ F),
\]
and we will argue that the right hand side is a constant diagram. The key observation is that the assumption in the lemma ensures the existence of a Hamiltonian isotopy, which realizes the diagram by convolutions by GKS kernels and shows that it is constant. This relaxes the compact support assumption in the cited references \cite{Zhou-isotopy,Kuo23}.

Since the $C^1$-norm of the contact Hamiltonian $\tau$ is uniformly bounded below, we know that the contact vector field of the isotopy $\Phi^\tau$ is bounded below by a positive constant. Therefore, for any $0 < t < t' < \epsilon$, there exists a neighborhood of $\Lambda$ of positive radius that is disjoint from $\varphi_s^\tau(\Lambda)$ for any $t < s < t'$. Then consider a cut-off function $\rho$ such that $\rho|_{\Lambda} = 0$, $\rho|_{\varphi_s^\tau(\Lambda)} = 1$ for any $t < s < t'$, and $\rho$ has bounded $C^1$-norm. Since the $C^1$-norm of $\tau$ is bounded, we know that the $C^1$-norm of the contact Hamiltonian $\tilde{\tau} = \rho \tau$ is also bounded, and hence the associated contact flow is complete. Moreover, we know that the contact flow generated by $\rho\tau$ sends $\Lambda \cup \varphi_t^\tau(\Lambda)$ to $\Lambda \cup \varphi_{t'}^\tau(\Lambda)$. Therefore, the isotopy $\Phi^{\tilde{\tau}}$ generated by $\tilde{\tau} = \rho \tau$ implies that
\[
    \Hom(G, K^\tau_t \circ F) \simeq \Hom(K^{\tilde{\tau}}_{t' -t} \circ G, K^{\tilde{\tau}}_{t'-t} \circ (K^\tau_t \circ F_t)) \simeq \Hom(G, K^\tau_{t'} \circ F).
\]
This shows that we have a constant diagram and finishes the proof.
\end{proof}

\subsection{Microlocalization and microsheaves}
We recall the notion of microlocalization following \cite[Chapters 4 \& 6]{KS90} and \cite{Guillermou23, nadler2016wrapped, NadlerShende}.

Let $X \subseteq S^*M$ be a closed subset. We know by \cite[Proposition 3.4]{GV24} that the inclusion functor $\iota_{X*} \colon \Sh_X(M) \hookrightarrow \Sh(M)$ is limit and colimit preserving, and consequently has a left and right adjoint. Since $\Sh(M)$ is a dualizable category, we can show that the left adjoint is represented by a sheaf kernel \cite[Proposition 3.1 and 3.2]{KSZ23}, which we call the microlocal kernel (initiated in the works of \cite[Section 7.1]{Tamarkin15}, \cite{Chiu}, \cite[Section 3.5]{Guillermou23}, and \cite{Zhang24capacity}):

\begin{proposition}[{\cite[Lemma 4.6]{Kuo-Li-duality}, \cite[Proposition 3.5]{KSZ23}}]
Let $X \subseteq S^*M$ be a closed subset. Then the left adjoint $\iota_X^* \colon \Sh(M) \to \Sh_X(M)$ of the tautological inclusion $\Sh_X(M) \hookrightarrow \Sh(M)$ is represented by the convolution with the unique sheaf kernel $P_X = \iota_{T^*M \times X}^*1_\Delta \in \Sh_{T^*M \times X}(M \times M)$.
\end{proposition}

\begin{notation} \label{not: proj}
We denote the projector from $\Sh(M)$ to $\Sh_{\tau\geq 0}(M)$ by $P_\tau$. This is, $P_\tau$ the unique sheaf kernel realizing the left adjoint of the inclusion $\Sh_{\tau\geq 0}(M) \subseteq \Sh(M)$ by convolution. 
\end{notation}

\begin{definition}
    Let $\Lambda \subseteq S^*M$ be a closed subset. Define the presheaf of stable $\infty$-categories $\msh^{\mathrm{pre}}_\Lambda$ on $S^*M$ by
    \[
        \msh^{\mathrm{pre}}_\Lambda \colon \Omega \mapsto \msh^{\mathrm{pre}}_\Lambda(\Omega) = \Sh_{\Lambda \cup \Omega^c}(M) / \Sh_{\Omega^c}(M),
    \]
    where the quotient is taken in $\PrLst$. 
    Define the sheafification of $\msh_\Lambda^{\mathrm{pre}}$ in the category of all stable $\infty$-categories to be $\msh_\Lambda$.\footnote{The choice of the ambient categories (of categories) ensure \cite[Theorem 6.1.2]{KS90} applies, which implies that the sheaf-Hom in $\msh_\Lambda$ is computed by the celebrated $\mu hom$ defined in \cite[Definition 4.4.1]{KS90}.}
\end{definition}

The sheaf of categories $\msh_\Lambda$ is a sheaf in $S^*M$ that is supported in $\Lambda$. By abuse of notations, we will use $\msh_\Lambda$ to denote the restriction of the sheaf on $\Lambda$. When $\Lambda$ is a smooth Legendrian, the sheaf of categories $\msh_\Lambda$ is a locally constant sheaf on $\Lambda$, whose obstruction theory is studied by Guillermou \cite{Gu12,Guillermou23} and Jin \cite{Jin20Jhomomorphism}:

\begin{theorem}[{\cite[Section 10.3 \& 10.6]{Guillermou23}, \cite[Theorem 1.1]{Jin20Jhomomorphism}}]\label{thm:musheaf}
    Let $\Lambda \subseteq S^*M$ be a closed Legendrian. Then the sheaf of categories $\msh_\Lambda$ is a locally constant sheaf of categories whose stalk is $\cC$ that is classified by the composition of the Lagrangian Gauss map and the delooping of the $J$-homomorphism 
    \[
        \Lambda \to U/O \to B\mathrm{Pic}(\mathbb{S}) \to B\mathrm{Pic}(\cC).
    \]
    For $\cC = \Mod(\bZ/2\bZ)_{/[1]}$ the (pre)triangulated orbit category of $\Mod(\bZ/2\bZ)$, the classifying map is always trivial, and for $\cC = \Mod(\bZ)$, the classifying map is trivial if and only if the Maslov class $\mu(\Lambda) = 0$ and the relative second Stiefel--Whitney class $rw_2(\Lambda) = 0$, and $\mu(\Lambda) = 0$ if and only if there exist objects with stalks being bounded complexes.
    \footnote{In fact, in order to exhibit the existence of $F$, the main result of \cite[Section 11.6]{NadlerShende}, which states that the obstruction is classified by a map to $B \mathrm{Pic}(\SC)$, is enough. This is because, when $\SC = \Mod(\bZ/2 \bZ)_{/[1]}$, the latter is a point $\{*\}$ so the obstruction vanishes trivially. As explained in \cite[Remark 11.21]{NadlerShende}, one of Jin's main contributions in \cite{Jin20Jhomomorphism} is to identify the obstruction in the universal case $\SC = \Mod( \bS)$ with the well-known invariant, the $J$-homomorphism.} 
\end{theorem}

\begin{notation}\label{not:legendrian_movie}
Given a contact Hamiltonian $\tau \colon S^*M \to \bR$ that is positive on a subset $\Lambda \subseteq S^*M$, we denote by $\Lambda^\tau_t \subseteq S^*M$ the time-$t$ push-off of $\Lambda$ under the contact Hamiltonian flow of $\tau$. We also denote by $\Lambda^{\tau} \subseteq S^*(M \times I)$ the Legendrian movie of $\Lambda$ under the contact Hamiltonian flow of $\tau$. 
\end{notation}

Following the suggestion of Viterbo, one can understand microsheaves using sheaves via the following doubling construction. The following theorem is a generalization of the result of Guillermou \cite[Proposition 11.3.5 \& Theorem 12.1.1]{Guillermou23} to arbitrary positive Hamiltonian flows.

\begin{theorem}[{\cite[Theorem 7.18]{NadlerShende}, \cite[Theorem 4.1]{KuoLi22}}]\label{thm:doubling}
    Let $M$ be a complete Riemannian manifold and $\Lambda \subseteq S^*M$ be a properly embedded Legendrian with a tubular neighborhood of positive radius. Let $\tau \colon S^*M \to \bR$ be a contact Hamiltonian that is positive and bounded on $\Lambda$. Then for sufficiently small $\epsilon > 0$, there exists a fully faithful embedding
    \[
    w_\Lambda \colon \msh_\Lambda(\Lambda) \hookrightarrow \Sh_{\Lambda^{\id} \cup \Lambda^\tau}(M \times [0, \epsilon]) \hookrightarrow \Sh_{\Lambda \cup \Lambda_\epsilon^\tau}(M).
    \]
\end{theorem} 
\begin{remark}
    While \cite[Theorem 4.1]{KuoLi22} is proved for compact Legendrians $\Lambda \subseteq S^*M$, it can be generalized to properly embedded Legendrians with tubular neighborhoods of positive radius by exactly the same argument in \cite[Theorem 3.7]{Li23Cobordism2} whenever the contact Hamiltonian $\tau$ is $C^1$-bounded, so that the push-off $\Lambda_\epsilon^\tau$ is still contained in the tubular neighborhood of $\Lambda$.
\end{remark}

\subsection{Interleaving distance and completeness of sheaves}  
Given a smooth function $\tau \colon S^*M \to \bR$, Guillermou--Viterbo \cite{GV24} and Asano--Ike \cite{AI24} considered the interleaving distance $d_\tau$ on certain categories of sheaves. We adapt the setting of Petit--Schapira \cite{petit2023thickening} and Guillermou--Viterbo \cite{GV24}. 

\begin{definition} \label{def: interleaving-distance} 
Let $\tau$ be a smooth Hamiltonian on $S^*M$. For $F, G \in \Sh_{\tau\geq 0}(M)$ and $a, b \geq 0$, we say that $(F,G)$ is \emph{$(a, b)$-interleaved} if there are morphisms
    \begin{equation}\label{eq: inv-up-to-c}
    u \colon F \to K^{\tau}_a \circ G, \ v \colon G \to K^{\tau}_b \circ F
    \end{equation}
    such that the compositions 
    \[
        K^{\tau}_{-a} \circ F \to G \to K^{\tau}_{b} \circ F, \quad K^{\tau}_{-b} \circ G \to F \to K^{\tau}_{a} \circ G
    \]
    are isomorphic to the natural continuation maps $c_{-a,b} \colon K^{\tau}_{-a} \to K^{\tau}_{b}$ and $c_{-b,a} \colon K^{\tau}_{-b} \to K^{\tau}_{a}$. The \emph{interleaving distance} between $F$ and $G$ are defined by
    \[
    d_\tau(F, G) = \inf\{a + b \mid (F, G) \text{ is }(a, b)\text{-interleaved}\}.
    \]
\end{definition}
\begin{remark}
    Here, we view the $K^{\tau}_c$'s as auto-equivalences on $\Sh_{\tau\geq 0}(M)$, and the existence of continuation maps are implied by and \Cref{def: conti-map} and \Cref{prop:continuation}. 
    A pair $(F,G)$ is $(a, b)$-interleaved if $F$ and $G$ are isomorphic to each other up to continuation maps with a size $a + b$.  
\end{remark}

Following the heuristic of viewing sheaves as singular Lagrangians, the notion of interleaving distance shares a close relation with its underlying geometry. In fact, microsupports respect the metric topology induced from $d_\tau$.

\begin{proposition}[{\cite[Proposition 6.26]{GV24}}]\label{thm:ss-limit}
    Let $F_n \in \Sh_{\tau\geq 0}(M)$ be a sequence of sheaves such that $\lim_{n\to \infty} d_\tau(F_n,F_\infty) = 0$, i.e., the sequence $F_n$ converges to $F_\infty$ under the interleaving distance. Then
    \[
        \cSS(F_\infty) \subseteq \bigcap\nolimits_{n\geq 1}\overline{\bigcup\nolimits_{k \geq n}\cSS(F_n)}.
    \]
\end{proposition}

\begin{remark} 
    A strict inequality in \Cref{thm:ss-limit} is sometimes unavoidable. 
    One might wonder: is it possible to perturb $F_n$ by some small $K^\tau_{\epsilon_n}$'s and obtain $F_n'$ so that it converges to the same limit $F_\infty$ while making the inequality an equality? 
    The answer is that this is in general not possible. 
    Per \Cref{thm:complete}, if $\lim_{n\rightarrow \infty} d(F_n,0) = 0$ and $F_n$ are constructible, $F_\infty = 0$ must not have any microsupport. However, for the family of sheaves $F_n = 1_{[0,{1}/{n})}$ on $\bR$, even up to perturbation to some $F_n' = 1_{[\epsilon_n, {1}/{n} + \epsilon_n)}$ for some sequence $\epsilon_n \to 0$, the intersection
    \[\bigcap\nolimits_{n\geq 1}\overline{\bigcup\nolimits_{k \geq n}\cSS(F_n')}\] 
    is always non-empty and hence the equality cannot hold.
\end{remark} 

The following result is an immediate corollary of \Cref{thm:ss-limit}. However, one can also directly prove it using $\Omega$-lenses.

\begin{proposition}\label{cor: ss-is-same} 
    Let $F, G \in \Sh_{\tau\geq 0}(M)$ such that $d_\tau(F, G) = 0$. Then $\cSS(F) = \cSS(G)$.  
\end{proposition}
\begin{proof}
    We remark that the statement follows from \Cref{thm:ss-limit} by setting $F_n = G$ for all $n$. 
    We present here a straightforward argument using only the fact that $\Omega$-lenses detect microsupport.

    By symmetry, it is sufficient to show $\cSS(G) \subseteq \cSS(F)$. Given any conic open subset $\Omega \subseteq T^*M$ such that $\overline{\Omega} \cap \cSS(F) = \varnothing$, consider any $\Omega$-lens $\Sigma$. When $\epsilon > 0$ is sufficiently small, we have
    \[
        \Hom(1_\Sigma, K^{\tau}_\epsilon  \circ F) = 0.
    \]
    Since $d_\tau(F, G) = 0$, for any $\epsilon > 0$ sufficiently small, there exist morphisms such that the composition
    \[
        \Hom(1_\Sigma, G) \to \Hom(1_\Sigma, K^{\tau}_\epsilon \circ F) \to \Hom(1_\Sigma, K^{\tau}_{2\epsilon} \circ G)
    \]
    is the natural continuation morphism. Thus, for any $\epsilon > 0$ sufficiently small, the natural continuation morphism factors through zero. However, we know that by \Cref{lem: continuation-(co)lim}
    \[
        \Hom(1_\Sigma, G) \xrightarrow{\sim} \lim\nolimits_{\epsilon \to 0}\Hom(1_\Sigma, K^{\tau}_{2\epsilon} \circ G).
    \]
    This implies that $\Hom(1_\Sigma, G) = 0$ so $\Omega \cap \cSS(G) = \varnothing$ by \Cref{lem: omega-lenses-test}. Thus, $\cSS(G) \subseteq \cSS(F)$.
\end{proof}

The converse of the previous \Cref{cor: ss-is-same} is however untrue. The reason is that the distance on the category comes from both the  underlying geometry as well as the algebraic coefficient:

\begin{example} \label{eg: loc-is-discrete}
    Local systems $\Loc(M) \subseteq (\Sh_{\tau \geq 0}(M), d_\tau)$ consists of discrete points. 
    Let $F \in \Sh_{\tau \geq 0}(M)$ and $L \in \Loc(M)$. Since $\Hom(K^{\tau}_t \circ F, L) = \Hom(F, L)$ for any $t \in \R$, $(F,L)$ being $(a,b)$-interleaving is equivalent to having $u \colon F \rightarrow L$ and $v \colon L \rightarrow F$ such that $u \circ v$ and $v \circ u$ are both identity. That is, $d_\tau(F,L) < \infty$ if and only if $F = L$.
\end{example}

As a consequence of the example above, we see that $(\Sh_{\tau \geq 0}(M), d_\tau)$ in general can have infinitely many disconnected components. For example,  when $M = S^1$, $\Sh_{\tau \geq 0}(M)$ contains at least $\Aut(1_{\SC})$-many components; in the case when $\SC = \Mod(\Bbbk)$, $\Aut(1_{\SC}) = \Bbbk^\times$. 

Now, we give a mild generalization to the fact that the interleaving distance induces a complete metric space on a suitable subcategory of $\Sh_{\tau\geq 0}(M)$. This generalizes the foundation theorem established by Guillermou--Viterbo and the first two authors is that the interleaving distance is a complete pseudo-metric on the corresponding category of sheaves:

\begin{theorem}[{\cite[Proposition 6.22]{GV24}, \cite[Corollary 4.5]{AI24}}]\label{thm:complete}
    The interleaving distance $d_\tau$ is a complete pseudo-metric on $\Sh_{\tau\geq 0}(M)$. 
\end{theorem} 

The non-Hausdorff feature of this metric comes from the fact that, for sheaves with arbitrary size, there can be a small portion which is undetectable:

\begin{remark}[{\cite[Example 6.5]{GV24}}] \label{rmk: degenecy}
    The interleaving distance $d_\tau$ is in general a degenerate pseudo-metric on $\Sh_{\tau\geq 0}(M)$. For example, let $M = \bR_t$, $F = \bigoplus_{n=1}^\infty 1_{[1/n, +\infty)}$ and $G = (\bigoplus_{n=1}^\infty 1_{[1/n, +\infty)} )\oplus 1_{[0,+\infty)}$. Then $d_\tau(F, G) = 0$ but $F \neq G$.
\end{remark}
 
However, on the category of constructible sheaves up to infinity with perfect stalks, Petit--Schapira--Waas showed that the distance is non-degenerate \cite[Theorem 3.4]{petit2025property}, and on the category of limits of constructible sheaves with perfect stalks over a field $\Sh_{lc}^b(M; \Bbbk)$, Guillermou--Viterbo showed that the distance is non-degenerate \cite[Proposition~B.8]{GV24}.

We can give a proof of non-degeneracy of the interleaving distance for all constructible sheaves whose singular supports are positively displaceable by the positive isotopy $\tau$.
 
\begin{theorem}\label{thm:non-deg-0}
    Let $\tau \colon S^*M \to \bR$ define a bounded contact isotopy on $S^*M$. Then the interleaving distance $d_\tau$ is non-degenerate on the category $\Sh_{c,\tau\geq 0}(M)$ of constructible sheaves $F \in \Sh_{\tau\geq 0}(M)$ such that $\varphi^\tau_t(\cSS^\infty(F)) \cap \cSS^\infty(F) = \varnothing$ for some $0 < t < \epsilon$.
\end{theorem}  
\begin{proof}
    Suppose $d_\tau(F, G) = 0$. \Cref{cor: ss-is-same} implies that $\cSS^\infty(F) = \cSS^\infty(G)$. Since $\cSS^\infty(F) = \cSS^\infty(G)$ is positively displaceable from itself when $0 < t < \epsilon_0$ for some $\epsilon_0 > 0$. Pick $\epsilon > 0$ such that $2 \epsilon < \epsilon_0$ and we can apply \Cref{lem: pert}  and obtain that the continuation maps induces equivalences 
    \begin{equation} \label{eq: ac-proof-purt}
        \Hom(F, G) \xrightarrow{\sim} \Hom(F, K^{\tau}_\epsilon \circ G) \ \text{and} \ \Hom(G, F) \xrightarrow{\sim} \Hom(G, K^{\tau}_\epsilon \circ F).
    \end{equation}
    Shrink $\epsilon$ if needed, since $d_\tau(F, G) = 0$, there exist $u \colon F \rightarrow K^{\tau}_\epsilon \circ G$ and $v \colon G \rightarrow K^{\tau}_\epsilon \circ F$ such that the compositions 
    \[
        F \to K^{\tau}_\epsilon \circ G \to K^{\tau}_{2\epsilon} \circ F, \quad  G \to K^{\tau}_\epsilon \circ F \to K^{\tau}_{2\epsilon} \circ G
    \]
    are equivalent to the continuation maps $c_{2 \epsilon}(F)$ and $c_{2 \epsilon}(G)$. The equivalences in \Cref{eq: ac-proof-purt} implies that there exist unique $a \colon F \rightarrow G$ and $b \colon G \rightarrow F$ factorizes them to $u = c_\epsilon(G) \circ a$ and $v = c_\epsilon(F) \circ b$. We thus have the following commuting diagram:
    \[
    \begin{tikzpicture}
    \node at (-3,2) {$F$};
    \node at (0,2) {$K^{\tau}_\epsilon \circ G$};
    \node at (3,2) {$K^{\tau}_{2 \epsilon} \circ F$};
    \node at (0,0) {$G$};
    \node at (3,0) {$K^{\tau}_\epsilon \circ F$};
    \node at (3,-2) {$F$};
    
    \draw [->, thick] (-2.7,2) -- (-0.8,2) node [midway, above] {$u$};
    \draw [->, thick] (0.9,2) -- (2.2,2) node [midway, above] {$v_\epsilon$};
    \draw [->, thick] (0.4,0) -- (2.2,0) node [midway, above] {$v$};
    
    \draw [<-, thick] (0,1.7) -- (0,0.3) node [midway, right] {$c_\epsilon(G)$}; 
    \draw [<-, thick] (3,1.7) -- (3,0.3) node [midway, right] {$c_\epsilon(F)$};
    \draw [<-, thick] (3,-0.3) -- (3,-1.7) node [midway, right] {$c_\epsilon(F)$};
    
    \draw [->, thick] (-2.7,1.8) -- (-0.3,0.1) node [midway, above] {$a$};
    \draw [->, thick] (0.3,-0.2) -- (2.7,-1.9) node [midway, above] {$b$};
    
    \node[scale=1.5] at (2.1,1) {$\circlearrowleft$};
    \node[scale=1.5] at (-0.7,1.2) {$\circlearrowleft$};
    \node[scale=1.5] at (2.3,-0.8) {$\circlearrowleft$};
    \end{tikzpicture}
    \]
    Note that the upper right square commutes by \Cref{rmk: compatibility-morphisms}. Since post-composing with $c_\epsilon(F) \circ(-)$ induces equivalences on the relevant Hom-spaces, the equality
    \[
        c_\epsilon(F) \circ c_\epsilon(F) \circ b \circ a = u \circ v_\epsilon = c_{2\epsilon}(F) = c_\epsilon(F) \circ c_\epsilon(F)
    \]
    implies that $b \circ a = \id_F$ and one can deduce $a \circ b = \id_G$ similarly.
\end{proof}

Recall that a sheaf $L \in \Sh(M \times M)$ is often referred as a kernel since the category $\Sh(M \times M)$ classifies colimit-preserving endomorphisms of $\Sh(M)$, 
and the convolution product corresponds to the composition of functors. 
For the main applications, we will consider interleaving distance between sheaf kernels from the GKS sheaf quantization \Cref{thm: GKS}. For such operators, we consider a variant of $\Sh_{\tau \geq 0}(M)$. 

Consider the cosphere bundle of the product $S^*(M \times M)$. We parametrize points in $S^*(M \times M)$ by $(x, \xi, y, \eta)$ where $|\xi|^2 + |\eta|^2 = 1$. Let 
\[
    \tau_2 \colon S^*(M \times M) \to \bR, \quad (x, \xi, y, \eta) \mapsto \tau(y, \eta/|\eta|)|\eta|
\]
and consider the subcategory of sheaves  
\[
    \Sh_{\tau_2 \geq 0}(M \times M) \coloneqq \{ L \in \Sh(M \times M) \mid \cSSif(L) \subseteq \tau_2^{-1}([0,+\infty)) \}.
\]
For $L \in \Sh_{\tau_2 \geq 0}(M \times M)$, applying \Cref{lem: microsupport-convolution-composition} to the situation where $M_1 = M \times M$ and $M_2 = M_3 = M \times M \times I$, we see that $K^\tau \circ L$ admits continuation maps. Namely, there is a canonical map
\[
    K^{\tau}_s \circ L \rightarrow K^{\tau}_t \circ L
\]
for $s \leq t$, and we define the interleaving distance for objects in $\Sh_{\tau_2 \geq 0}(M \times M)$ in a the same way as \Cref{def: interleaving-distance}, and completeness of the interleaving distance follows in the same way as before. 

\begin{definition} \label{def: distance-kernel}
    For $F, G \in \Sh_{\tau_2 \geq 0}(M \times M)$ and $a, b \geq 0$, we say that $(F,G)$ is \emph{$(a, b)$-interleaved} if there are morphisms $u \colon F \to K^{\tau}_a \circ G$ and $v \colon G \to K^{\tau}_b \circ F$ such that the compositions
    \[
        K^{\tau}_{-a} \circ F \to G \to K^{\tau}_b \circ F, \quad K^{\tau}_{-b} \circ G \to F \to K^{\tau}_a \circ G
    \]
    are isomorphic to the natural continuation morphisms $c_{-a,b} \colon K^{\tau}_{-a} \to K^{\tau}_{b}$ and $c_{-b,a} \colon K^{\tau}_{-b} \to K^{\tau}_{a}$. The \emph{interleaving distance} between $F$ and $G$ are defined by
    \[
    d_\tau(F, G) = \inf\{a + b \mid (F, G) \text{ is }(a, b)\text{-interleaved}\}.
    \]
\end{definition}

\begin{theorem}\label{thm:complete-kernel}
    The interleaving distance $d_\tau$ is a complete pseudo-metric on $\Sh_{\tau_2\geq 0}(M \times M)$.
\end{theorem} 

While the function $\tau_2 \colon S^*(M \times M) \to \bR$ is not smooth, one can check that it is true that for any $L \in \Sh_{\tau_2\geq 0}(M \times M)$, we have $L \simeq \lim_{s \to 0^+}K^{\tau}_s \circ L$. Therefore, \Cref{cor: ss-is-same} still hold.

\begin{proposition}\label{prop: ss-is-same-kernel}
    Let $K, L \in \Sh_{\tau_2\geq 0}(M \times M)$ such that $d_\tau(K, L) = 0$. Then $\cSS(K) = \cSS(L)$.
\end{proposition}
\begin{proof}
    We know that by Equation \eqref{for: ms-isotopy-endofunctor}, $\cSS(K^\tau \circ L)$ consists of points $(x_1, \xi_1, x_3, \xi_3, t, -\tau(x_3, \xi_3))$ where there exist $(x_2, -\xi_2, x_3, \xi_3) \in \cSS(K^\tau_t)$ and $(x_1, \xi_1, x_2, \xi_2) \in \cSS(L)$.
    When $\epsilon > 0$ is sufficiently small, $\cSS(K^\tau_\epsilon)$ is sufficiently close to $T^*_\Delta(M \times M)$. Thus, we also know that for any open subset $\Omega \subseteq T^*(M \times M)$ such that $\ol{\Omega} \cap \cSS(L) = \varnothing$, when $\epsilon > 0$ is small, $\ol{\Omega} \cap \cSS(K_\epsilon^\tau \circ L) = \varnothing$. Thus, for any $\Omega$-lens $\Sigma$, when $\epsilon > 0$ is sufficiently small,
    \[\Hom(1_\Sigma, K^\tau_\epsilon \circ L) = 0.\]
    Since $d_\tau(K, L) = 0$, for $\epsilon > 0$ sufficiently small, there exist morphisms such that the composition
    \[\Hom(1_\Sigma, L) \to \Hom(1_\Sigma, K^\tau_\epsilon \circ K) \to \Hom(1_\Sigma, K^\tau_{2\epsilon} \circ L)\]
    is the natural continuation morphism. Hence when $\epsilon > 0$ is sufficiently small, the natural continuation morphism factors through zero.
    
    On the other hand, for the points $(x_1, \xi_1, x_3, \xi_3, t, -\tau(x_3, \xi_3))$ where there exist $(x_2, -\xi_2, x_3, \xi_3) \in \cSS(K^\tau_t)$ and $(x_1, \xi_1, x_2, \xi_2) \in \cSS(L)$, if $\xi_1 = \xi_3 = 0$ then $\tau(x_3, \xi_3) = 0$, and thus $K^\tau \circ L$ is still non-characteristic along $I$. Therefore,
    \[
    \Hom(1_\Sigma, L) \simeq \lim_{\epsilon \to 0^+}\Hom(1_\Sigma, K^\tau_{2\epsilon} \circ L).
    \]
    This implies that $\Hom(1_\Sigma, L) = 0$ so $\Omega \cap \cSS(L) = \varnothing$ by \Cref{lem: omega-lenses-test}. Thus, $\cSS(L) \subseteq \cSS(K)$ and this completes the proof.
\end{proof}

However, the non-degeneracy \Cref{thm:non-deg-0} becomes more tricky as it requires \Cref{lem: pert}. We need to consider a smoothing of $\tau_2$ away from $S^*M \times 0_M$ where the contact flow is well defined.

\begin{lemma}\label{lem:disjoint-graph}
    Let $M$ be a complete Riemannian manifold and suppose $\tau \colon S^*M \to \bR$ defines a bounded Reeb flow on $S^*M$. Then there exists a cut-off function $\rho\colon S^*(M \times M) \to \bR$ supported on a tubular neighborhood of $S^*M \times 0_M$ such that the Hamiltonian $(1-\rho) \tau_2 \colon S^*(M \times M) \to \bR$ defines a complete contact flow.
\end{lemma}
\begin{proof}
    Consider a complete Riemannian metric on $M$, which induces a complete Riemannian metric on $M \times M$ and $S^*(M \times M)$. Then there exists a tubular neighborhood of $S^*M \times 0_M$ of positive radius. We can choose a cut-off function $\rho\colon S^*(M \times M) \to \bR$ supported in the tubular neighborhood such that $|d\rho|$ is uniformly bounded. Then since $|(1-\rho)\tau_2| \leq |\tau_2|$ and $|d((1-\rho)\tau_2)| \leq |\tau_2||d\rho| + |d\tau_2|$ are both uniformly bounded, the norm of the contact vector field $X_{(1-\rho)\tau_2}$ is uniformly bounded and therefore defines a complete contact flow.
\end{proof}

\begin{theorem}\label{thm:non-deg_kernel}
    Let $\tau \colon S^*M \to \bR$ define a bounded positive contact Hamiltonian and $\Omega' \subseteq \Omega$ is a pair of tubular neighborhoods of $S^*M \times 0_M$ with positive radii $r' < r$. Let $\rho \colon S^*(M \times M) \to \bR$ be a bounded cut-off function supported in $\Omega'$. Then the interleaving distance $d_\tau$ is non-degenerate on the category $\Sh_{c,\Omega^c}(M \times M)$ of constructible sheaves $K \in \Sh_{\Omega^c}(M \times M)$ such that $\varphi^{(1-\rho)\tau_2}_t(\cSS^\infty(K)) \cap \cSS^\infty(K) = \varnothing$ for some $0 < t < \epsilon$.
\end{theorem}
\begin{proof}
    Suppose $d_\tau(K, L) = 0$. \Cref{cor: ss-is-same} implies that $\cSS^\infty(K) = \cSS^\infty(L)$. By \Cref{lem:disjoint-graph}, $(1 - \rho)\tau_2$ defines a bounded contact vector field with a complete contact flow. Since $\cSS^\infty(K) \subseteq \Omega^c$, we can pick $\epsilon_0 > 0$ such that $\varphi^{(1-\rho)\tau_2}_t(\cSS^\infty(K)) \subseteq \Omega^c$ and
    \[\varphi^{(1-\rho)\tau_2}_t(\cSS^\infty(K)) \cap \cSS^\infty(K) = \varnothing, \quad 0 < t < \epsilon_0.\]
    Pick $\epsilon > 0$ such that $2 \epsilon < \epsilon_0$. Let $\widetilde{K}^{(1-\rho)\tau_2}$ be the sheaf quantization of the contact Hamiltonian $(1-\rho)\tau_2$; note that $\widetilde{K}^{(1-\rho)\tau_2}$ lives on $M^2 \times M^2 \times I$ so it treats $L$ as an object. We can apply \Cref{lem: pert} and obtain that the continuation maps induces equivalences 
    \begin{equation}
        \Hom(K, L) \xrightarrow{\sim} \Hom(K, \widetilde{K}^{(1-\rho)\tau_2}_\epsilon \circ L) \ \text{and} \ \Hom(L, K) \xrightarrow{\sim} \Hom(L, \widetilde{K}^{(1-\rho)\tau_2}_\epsilon \circ K).
    \end{equation}
    Since $\varphi^{(1-\rho)\tau_2}_t(\cSS^\infty(K)) \subseteq \Omega^c$, we know that $\varphi^{(1-\rho)\tau_2}_t(\cSS^\infty(K)) = \Gamma^{\tau_2}_t \circ \cSS^\infty(K)$, by \Cref{lem: microsupport-convolution-composition} $\cSS^\infty(\widetilde{K}^{(1-\rho)\tau_2} \circ K) = \cSS(K^\tau \circ K)$. Then we can conclude that $\widetilde{K}^{(1-\rho)\tau_2}_\epsilon \circ K = K^\tau_\epsilon \circ K$, i.e., post-composition by $K^\tau_\epsilon$ is equivalent to being acts as an object by $\widetilde{K}^{(1-\rho)\tau_2}_\epsilon$. Thus the following continuation maps induces equivalences 
    \begin{equation} \label{eq: ac-proof-purt-kernel}
        \Hom(K, L) \xrightarrow{\sim} \Hom(K, K^{\tau}_\epsilon \circ L) \ \text{and} \ \Hom(L, K) \xrightarrow{\sim} \Hom(L, K^{\tau}_\epsilon \circ K).
    \end{equation}
    Shrink $\epsilon$ if needed, since $d_\tau(K, L) = 0$, there exist $u \colon K \rightarrow K^{\tau}_\epsilon \circ L$ and $v \colon L \rightarrow K^{\tau}_\epsilon \circ K$ such that the compositions 
    \[
        K \to K^{\tau}_\epsilon \circ L \to K^{\tau}_{2\epsilon} \circ K, \quad  L \to K^{\tau}_\epsilon \circ K \to K^{\tau}_{2\epsilon} \circ L
    \]
    are equivalent to the continuation maps $c_{2 \epsilon}(K)$ and $c_{2 \epsilon}(L)$. The equivalences in \Cref{eq: ac-proof-purt-kernel} imply that there exist unique $a \colon K \rightarrow L$ and $b \colon L \rightarrow K$ factorizes them to $u = c_\epsilon(L) \circ a$ and $v = c_\epsilon(K) \circ b$. We thus have $a \circ b = \id_L$ and $b \circ a = \id_K$. This completes the proof.
\end{proof}

\begin{remark} \label{rmk: genera-functors}
    In fact, \Cref{prop: ss-is-same-kernel} and \Cref{thm:non-deg_kernel} hold more generally on $N \times M$ where $N$ is another smooth manifold. Since kernels induce functors, this means that there is a non-degenerate metric on a suitable class of functors from $\Sh(N)$ to $\Sh(M)$, which in the case of $N =\{*\}$ recovers the case of objects \Cref{thm:non-deg-0}. Furthermore, there is a $\tau_1$ version corresponding to pre-composition instead of the post-composition case presented here. We do not go into details of these generalizations since the focus of this paper requires only endofunctors.
\end{remark}

We can also define another variant of the interleaving distance, which will be used later on. Given a smooth function $\tau \colon S^*M \to \bR$, we define a log interleaving distance $d_{\ln\tau}$ on the derived category $\Sh_{\tau \geq 0}(M)$ as follows:

\begin{definition}
    Let $\tau$ be a smooth Hamiltonian on $S^*M$. For $F, G \in \Sh_{\tau\geq 0}(M)$ and $a, b \geq 0$, we say that $(F,G)$ is \emph{log $(a, b)$-interleaved} if there are morphisms $u \colon F \to K^{\tau}_{e^a-1} \circ G$ and $v \colon G \to K^{\tau}_{e^b-1} \circ F$ such that the compositions
    \[
        K^{\tau}_{1-e^a} \circ F \to G \to K^{\tau}_{e^b-1} \circ F, \quad K^{\tau}_{1-e^b} \circ G \to F \to K^{\tau}_{e^a-1} \circ G
    \]
    are isomorphic to the continuation morphisms. The \emph{log interleaving distance} of $F$ and $G$ are defined by
    \[
    d_{\ln, \tau}(F, G) = \inf\{a + b \mid (F, G) \text{ is log }(a, b)\text{-interleaved}\}.
    \]
\end{definition}

Since the function $f(x) = e^x - 1$ is continuous and $f(0) = 0$, it follows immediately from \Cref{thm:complete} that the log interleaving distance is also a complete pseudo-metric.

\section{Hofer--Shelukhin Distances and Sheaves}

In this section, we discuss the relation between the interleaving distance of sheaves and the (Chekanov--)Hofer--Shelukhin distance of Legendrians. We will prove \Cref{thm:main-Hofer-distance} (\Cref{thm:hofer-norm}) and use that to deduce \Cref{thm:main-Hofer-nondeg} (\Cref{thm:Hofer-nondeg,thm:chekanov-hofer-shelukhin}).

\subsection{Hofer--Shelukhin norm and interleaving distance} \label{sec: HS-norm}
Our goal in the section is to deduce \Cref{thm:hofer-norm} (\Cref{thm:main-Hofer-distance}), which generalizes the result of the first two authors \cite{AI20} to any cosphere bundles and any Reeb flows.

For the contact manifold $(S^*M, \xi_{\mathrm{std}})$, we will consider a smooth function $\tau \colon S^*M \to \bR$ and consider the open contact submanifold $Y = \tau^{-1}(\bR_{>0}) \subseteq S^*M$. The Hamiltonian $\tau \colon Y \to \bR_{>0}$ defines a positive contact flow, which is the Reeb flow with respect to some contact form $\alpha$ on $Y$. 
We will consider contact Hamiltonians $H \colon Y \times I \to \bR$ that extend smoothly to $S^*M \times I$ by zero, or equivalently, $H \colon S^*M \times I \to \bR$ such that $\supp(H) \subseteq \tau^{-1}(\bR_{>0})$.

We recall that in \cite{Shelukhin17}, the Hofer--Shelukhin norm of a contact Hamiltonian $H \colon Y \times I \to \bR$ with respect to the contact form $\alpha$ is defined as \footnote{While $\|H\|_{\mathrm{HS},\alpha}$ does not depend on $\alpha$,  we choose this notation to keep consistency with the later $\|H\|_{\mathrm{HS},\tau}$.}
\[
    \|H\|_{\mathrm{HS},\alpha} \coloneqq \int_0^1 \sup_{(x, \xi) \in Y} |H_s(x, \xi)| \, ds. 
\] 
The Hofer--Shelukhin distance of $\varphi, \psi \in \mathrm{Cont}_0(S^*M, \xi_{\mathrm{std}})$ is defined as 
\[
    d_{\mathrm{HS},\alpha}(\varphi, \psi) \coloneqq \inf_{H: \varphi = \varphi^H_1 \circ \psi}\|H\|_{\mathrm{HS},\alpha} = \inf_{H: \varphi = \varphi^H_1 \circ \psi}\int_0^1 \sup_{(x, \xi) \in Y} |H_s(x, \xi)| \, ds.
\]
On the universal cover $\widetilde{\mathrm{Cont}}_0(Y, \xi_{\mathrm{std}})$, consisting of homotopy classes of contact isotopies, abusing notations, the Hofer--Shelukhin distance is defined as
\[
    {d}_{\mathrm{HS},\alpha}(\varphi, \psi) \coloneqq \inf_{H:\varphi = \varphi^H \circ \psi}\|H\|_{\mathrm{HS},\alpha} = \inf_{H: \varphi = \varphi^H \circ \psi}\int_0^1 \sup_{(x, \xi) \in Y} |H_s(x, \xi)| \, ds.
\]
Note that the Hofer--Shelukhin distance depends on the choice of the contact form. 
Consider the positive Hamiltonian $\tau \colon Y \to \bR_{>0}$.
Under the contact form $\alpha_{\mathrm{std}}$, $\tau$ defines the Reeb flow for a different contact form $\alpha = \tau\alpha_{\mathrm{std}}$. 
Then, for any Hamiltonian $H$ under the given contact form $\alpha_{\mathrm{std}}$, we will abuse the notation and write 
\[
    \|H\|_{\mathrm{HS},\tau} \coloneqq \int_0^1 \sup_{(x, \xi) \in Y} |H_s(x, \xi)/\tau(x, \xi)| \, ds.
\]
On $\widetilde{\mathrm{Cont}}_0(Y, \xi_{\mathrm{std}})$, 
we also write
\[
    d_{\mathrm{HS},\tau}(\varphi, \psi) \coloneqq \inf_{H: \varphi = \varphi^H \circ \psi}\|H\|_{\mathrm{HS},\tau} = \inf_{H: \varphi = \varphi^H \circ \psi}\int_0^1 \sup_{(x, \xi) \in Y} |H_s(x, \xi)/\tau(x,\xi)| \, ds.
\]

The main result of the section is the following, which shows that the interleaving distance of sheaves is bounded by the Hofer--Shelukhin distance, generalizing the result of the first two authors \cite{AI20} and \cite[Theorem 5.2]{AI24}.  We remark that the only technical point to generalize is a lemma equivalent to the following one. The rest is a similar argument refining partitions of $[0,1]$.

\begin{lemma} \label{lem: HS-norm}
    Let $\tau \colon S^*M \to \bR$ be a contact Hamiltonian that defines a complete contact flow. 
    Let $H \colon S^*M \times I \to \bR$ be a compactly supported contact Hamiltonian such that $\varphi = \varphi^H_1$ such that $\supp(H) \subseteq \tau^{-1}(\bR_{>0})$. 
    Then, for $s_0,s_1 \in I$ with $s_0 \leq s_1$, we have 
    \[
        d_{\tau}(K^H_{s_0} \circ P_\tau, K^H_{s_1} \circ P_\tau) \leq 2 \max_{(x,\xi,s) \in S^* M \times [s_0,s_1]}  | H_s(x,\xi)/\tau(x, \xi) | |s_1-s_0|.
    \]
    Here $P_\tau$ is the projector defined in \Cref{not: proj}.
\end{lemma}
\begin{proof}
By the compactness assumption, we can set 
    \[
        -a \coloneqq \min_{(x,\xi, s) \in S^* M \times [s_0,s_1]} \left( {H_s(x,\xi)}/{\tau(x,\xi)} \right),  \quad 
        b \coloneqq  \max_{ (x,\xi, s) \in S^* M \times [s_0,s_1]} \left( {H_s(x,\xi)}/{\tau(x,\xi)} \right).  
    \]
    Here, the value of $H_s(x,\xi)/\tau(x,\xi)$ is set to $0$ outside $\supp(H)$.  This implies the inequalities
    \[ 
    -a\tau \leq H \leq b \tau
    \]
    on $[s_0,s_1]$.
    Note that we have $P_\tau \in \Sh_{\tau_2 \geq 0}(M \times M)$. Thus, since for any $c \in \bR$, $K^{c\tau}_s = K^\tau_{cs}$, by \Cref{prop:continuation}, we know that there exist continuation maps $K^\tau_{-a(s_1 - s_0)} \circ P_\tau \to K^H_{s_1-s_0} \circ P_\tau$ and $K^H_{s_1-s_0} \circ P_\tau \to K^\tau_{b(s_1-s_0)} \circ P_\tau$, so there exist continuation maps
    \begin{gather*} 
    K^{\tau}_{-a(s_1-s_0)} \circ K^H_{s_0} \circ P_\tau \to K^H_{s_1 - s_0} \circ K^H_{s_0} \circ P_\tau = K^H_{s_1} \circ P_\tau, \\
    K^H_{s_1} \circ P_\tau = K^H_{s_1 - s_0} \circ K^H_{s_0} \circ P_\tau \to K^{\tau}_{b(s_1-s_0)} \circ K^H_{s_0} \circ P_\tau.
    \end{gather*}
    Consider the two different continuation maps, one from the Hamiltonian $\tau$ and the other from the Hamiltonian $a\ol{\tau}\#H$ followed by the Hamiltonian $\ol{H}\#b\tau$ in \Cref{eq: inv-up-to-c}. We can construct a family of positive contact isotopies from the compositions of the above two positive isotopies to the positive isotopy $(b -a)\tau$ by linear interpolation and thus conclude, using \Cref{prop:continuation} and \Cref{rem:continuation}, that the two compositions of the continuation maps
    \begin{align*}
        K^{\tau}_{-a(s_1-s_0)} \circ K^H_{s_0} \circ P_\tau \rightarrow K^H_{s_1} \circ P_\tau \rightarrow K^{\tau}_{b(s_1-s_0)} \circ K^H_{s_0} \circ P_\tau, \\
        K^{\tau}_{-b(s_1-s_0)} \circ K^H_{s_1} \circ P_\tau \rightarrow K^H_{s_0} \circ P_\tau \rightarrow K^{\tau}_{a(s_1-s_0)} \circ K^H_{s_1} \circ P_\tau
    \end{align*}
    are equivalent to continuation maps. 
    Thus, we conclude
    \[
        d_{\tau}(K^H_{s_0} \circ P_\tau, K^H_{s_1} \circ P_\tau) \leq (a + b)(s_1 - s_0) \leq 2 \max_{(x,\xi,s) \in S^* M \times [s_0,s_1]}  | H_s(x,\xi)/\tau(x,\xi) | |s_1 - s_0|. \qedhere
    \]
\end{proof}

\begin{theorem}[Theorem~\ref{thm:main-Hofer-distance}] \label{thm:hofer-norm}
    Let $\tau \colon S^*M \to \bR$ be a contact Hamiltonian that defines a complete contact flow. 
    Let $H \colon S^*M \times I \to \bR$ be a compactly supported contact Hamiltonian such that $\supp(H) \subseteq \tau^{-1}(\bR_{>0})$. 
    Then, we have 
    \[
        d_{\tau}(P_\tau, K^H_1 \circ P_\tau) \leq 2\|H\|_{\mathrm{HS},\tau}.
    \]
    In particular, for any $F \in \Sh_{\tau\geq 0}(M)$, we have
    \[
        d_\tau(F, K^H_1 \circ F) \leq 2\|H\|_{\mathrm{HS},\tau}.
    \]
    Similarly, let $\varphi \in \widetilde{\operatorname{Cont}}_0(Y, \xi_{\mathrm{std}})$ be a homotopy class of contact isotopies. 
    Then we have 
    \[
        d_{\tau}(P_\tau, K^{\varphi} \circ P_\tau) \le 2\,{d}_{\mathrm{HS},\tau}(\id, \varphi).
    \]
    In particular, for $F \in \Sh_{\tau \geq 0}(M)$, we have
    \[
        d_\tau(F, K^{\varphi} \circ F) \leq 2\,{d}_{\mathrm{HS},\tau}(\id, \varphi),
    \]
\end{theorem}
\begin{proof}
    We divide the time interval into small pieces $0 = s_0 < s_1 < \dots < s_n = 1$.  
    The previous \Cref{lem: HS-norm} implies that  
    \[
        d_\tau(K^H_{s_{i-1}} \circ P_\tau, K^{H}_{s_i} \circ P_\tau) \leq 2 \max_{(x,\xi,s) \in S^* M \times [s_{i-1},s_i]} | H_s(x,\xi)/\tau(x,\xi) | |s_i - s_{i-1}|.
    \]
    Using the triangle inequality of the interleaving distance $d_\tau$ and refining the partition so that $\sup_{1\leq i\leq n}|s_i - s_{i-1}| \to 0$, we obtain the result that
    \[
        d_\tau(P_\tau, K^H_1 \circ P_\tau) \leq 2 \int_0^1 \sup_{(x,\xi) \in S^*M}|H_s(x, \xi)/\tau(x,\xi)| \, ds \eqqcolon 2\|H\|_{\mathrm{HS},\tau}.
    \]
    The inequality on $\Sh_{\tau \geq 0}(M)$ follows from the property of being a projector: for a sheaf $F \in \Sh(M)$,  $F$ belongs to the subcategory $\Sh_{\tau \geq 0}(M)$ if and only if $F = P_\tau \circ F$.
    
    When $\varphi \in \widetilde{\operatorname{Cont}}_0(Y, \xi_{\mathrm{std}})$ is a homotopy class of contact isotopies, then by \Cref{rem:GKS} we know that the sheaf quantization $K^\varphi = K^H_1$ is well-defined. Therefore, 
    \[
        d_{\tau}(P_\tau, K^\varphi \circ P_\tau) \leq 2\inf_{H : \varphi = [\varphi^H_t]}\|H\|_{\mathrm{HS},\tau}.
    \]
    This then completes the proof.
\end{proof}

\begin{remark}\label{rem:hofer-persist}
    Consider the sheaf quantization $K^\tau \in \Sh(M \times M \times \bR)$. 
    Using functoriality of the interleaving distance \cite{AI20}, one can show that 
    \[
    d_\tau(\pi_{\bR*}\mathscr{H}om(1_\Delta \boxtimes 1_\bR, K^\tau), \pi_{\bR*}\mathscr{H}om(K^\varphi \boxtimes 1_\bR, K^\tau)) \leq 2 \,{d}_{\mathrm{HS},\tau}(\id, \varphi).
    \]
    This is the sheaf-theoretic analogue of a recent result of Cant \cite{Cant23Shelukhin} and Djordjevi\'c--Uljarevi\'c--Zhang \cite{DUZ23} on the interleaving distance of the symplectic cohomology. Alternatively, (over a discrete ring) assuming that $M$ is spin, this also follows from their results \cite{Cant23Shelukhin,DUZ23} plus the isomorphism of the Hamiltonian Floer cohomology and the sheaf invariants by Guillermou--Viterbo \cite{GV24} and Kuo--Shende--Zhang \cite{KSZ23}. 
\end{remark}

\begin{remark}
As a consequence, we can show that the sheaf-theoretic spectral norm of contact Hamiltonians are also bounded by the Hofer--Shelukhin norm.
\end{remark}

Let $\alpha$ and $\alpha'$ be contact forms on $(Y, \xi)$ such that $\alpha' = e^h\alpha$. Then the conformal norm or Banach--Mazur norm of the contact form \cite{SZ21,RosenZhang21BM} is
\[
    d_{\mathrm{BM}}(\alpha, \alpha') = \sup_{(x,\xi) \in Y} |h(x, \xi)|.
\]
Let $\varphi \in \operatorname{Cont}_0(Y, \xi)$ such that $\varphi^*\alpha' = \alpha$. Then the Banach--Mazur norm of the contactomorphism \cite{SZ21,RosenZhang21BM} is
\[
    \|\varphi\|_{\mathrm{BM}} = d_{\mathrm{BM}}(\alpha, \alpha') = \sup_{(x,\xi) \in Y} |h(x, \xi)|.
\]
The Banach--Mazur norm defines a pseudo-distance $d_{\mathrm{BM}}$ on the space of contact forms, since any two contact forms are related by a conformal factor.
Let $\alpha$ and $\alpha'$ be contact forms on the contact manifold $(Y, \xi)$ and $R_\alpha$ and $R_{\alpha'}$ be the Reeb vector fields associated to the contact forms. Suppose $\alpha = e^h\alpha'$. Then under the contact form $\alpha$, the Reeb vector field $R_{\alpha'}$ is defined by the contact Hamiltonian $e^h$.

\begin{lemma}\label{lem:conformal-reeb-commute} 
    Let $\alpha$ and $\alpha'$ be contact forms on the contact manifold $(Y, \xi)$ and $R$ and $R'$ be the Reeb vector fields associated to the contact forms. Suppose $\alpha =  \varphi^*\alpha'$ for some $\varphi \in \Cont(Y,\xi)$. Then 
    \[
        \varphi_*R = R' \circ \varphi, \quad \varphi \circ \varphi^{R}_s = \varphi^{R'}_s \circ \varphi.
    \]
\end{lemma}
\begin{proof}
It is sufficient to prove the first equality since the second differentiates to the first. 
It suffices to check $\varphi_*R_\alpha$ satisfy the same Hamiltonian equation for $\alpha'$ up to pre-composing with $\varphi$. We first compute that
\[
    \alpha'_{\varphi(y)}( \varphi_{*y} R_y ) = (\varphi^* \alpha')_y (R_y) = \alpha_y (R_y ) = 1.
\]
Secondly, we compute that 
\[
    d \alpha'_{\varphi(y)}(\varphi_{*y}R_y, \varphi_{*y}(-)) = (\varphi^* d \alpha')_y(R_y, -) = (d \alpha )_y(R_y, -) = 0.
\]
Uniqueness of Reeb vector field thus implies that $\varphi_{*y}R_y = R'_{\varphi(y)}$ as desired. 
\end{proof}

Similar to \Cref{thm:hofer-norm}, we can also show the log interleaving distance of the quantization of Reeb flows are bounded by the Banach--Mazur distance.

\begin{proposition}\label{prop:banach-mazur}
    Let $\tau \colon S^*M \to \bR$ and $\tau' \colon S^*M \to \bR$ be two positive Hamiltonians that define the Reeb flows for the contact forms $\alpha$ and $\alpha'$ on $(S^*M, \xi_{\mathrm{std}})$. Then
    \[
        d_{\ln\tau}(K^\tau_1 \circ F, K^{\tau'}_1 \circ F) \leq 2d_{\mathrm{BM}}(\alpha, \alpha'), \quad d_{\ln\tau}(K^\tau_1, K^{\tau'}_1) \leq 2d_{\mathrm{BM}}(\alpha, \alpha').
    \]
\end{proposition}
\begin{proof}
    By \Cref{prop:continuation}, if $G \leq H$, the positive contact isotopy $H \#(-G)$ gives a canonical continuation map $K^G \to K^H$.
    We know that the defining Hamiltonian satisfies the relation $\tau' = e^h\tau$. When $h$ is unbounded, there is nothing to show. When $-a \leq h \leq b$, we have continuation maps  
    \begin{gather*}
        K^{\tau}_{e^{-a}} \circ F \to K^h_1 \circ F \to K^{\tau}_{e^{b}} \circ F,\\
        K^{\tau}_{e^{-b}} \circ K^h_1 \circ F \to F \to K^{\tau}_{e^{a}} \circ K^h_1 \circ F.
    \end{gather*}
    By \Cref{prop:continuation} and \Cref{rem:continuation}, we know that the compositions of the above continuation maps equal the natural continuation map of the Reeb flow. Therefore, $(F, K^h_1 \circ F)$ are log $(a, b)$-interleaved. This proves the result.
\end{proof}

\begin{remark}\label{rem:distance-of-hom}
    Consider the sheaf quantization $K^\tau \in \Sh(M \times M \times \bR)$. Using functoriality of the interleaving distance \cite{AI20}, one can show that 
    \[
        d_{\ln\tau}(\pi_{\bR *}\mathscr{H}om(1_\Delta \boxtimes 1_\bR, K^\tau), \pi_{\bR *}\mathscr{H}om(1_\Delta \boxtimes 1_\bR, K^{\tau'})) \leq 2\, d_{\mathrm{BM}}(\alpha, \alpha').
    \]
    This is the sheaf theory analogue of the result of Stojisavljevi\'c--Zhang \cite{SZ21}. Alternatively, (over a discrete ring) assuming that $M$ is spin, this also follows from their Floer-theoretic results \cite{SZ21} plus the isomorphism result by Guillermou--Viterbo \cite{GV24} and Kuo--Shende--Zhang \cite{KSZ23}.
\end{remark}

\begin{remark}
As a consequence, we can show that the log sheaf-theoretic spectral norm of contact Hamiltonians are bounded by the Banach--Mazur norm.
\end{remark}

\subsection{Non-degeneracy of the (Chekanov--)Hofer--Shelukhin distance}
Using the above results and \Cref{thm:complete}, we are able to define sheaf quantizations of limits of contact isotopies under the Hofer--Shelukhin distance as in the symplectic situation \cite[Section 5.3]{AI24}.  We will use that to deduce \Cref{thm:Hofer-nondeg,thm:chekanov-hofer-shelukhin} (\Cref{thm:main-Hofer-nondeg}).

However, unlike the symplectic situation, due to the issue that the Hofer--Shelukhin distance is not conjugation invariant, we need more work to show that the limit sheaf quantizations are invertible. 

Let $\varphi \colon S^*M \to S^*M$ be a contactomorphism such that $\varphi^*\alpha = e^h\alpha$. We recall that from \Cref{not:graph}, the graph of the contactomorphism is the Legendrian submanifold
\[
    \Gamma^\varphi \coloneqq \{(x, -\xi, y, \eta) \mid (y, \eta/|\eta|) = \varphi(x, \xi/|\xi|), |\xi|/|\eta| = e^{h(x,\xi)}\} \subseteq S^*(M \times M).
\]
For the contact Hamiltonian $H$ that defines the contact isotopy $\varphi_t \colon S^*M \to S^*M$ such that $(\varphi_t)^*\alpha = e^{h_t}\alpha$, one can compute that the conformal factor is equal to (where $R$ is the Reeb vector field of the contact form $\alpha$)
\[
    h_t = \int_0^t dH_s(R) \circ \varphi_s\, ds.
\]

\begin{theorem}\label{thm:Hofer-quantize}
    Let $\tau \colon S^*M \to \bR$ be a positive Hamiltonian. Let $\varphi_n \in \operatorname{Cont}_0(S^*M, \xi_{\mathrm{std}})$ be the time-$1$ maps of the contact Hamiltonians flows defined by $H_n$ that uniformly converges to a continuous function $H_\infty$. Then there exists a sheaf quantization
    $K^{H_\infty}_1 \in \Sh(M \times M)$ with 
    \[
        \cSS^\infty(K^{H_\infty}_1) \subseteq \bigcap\nolimits_{n\geq 1}\overline{\bigcup\nolimits_{k\geq n}\Gamma^{\varphi_k}}.
    \]
    Moreover, $d_{\tau}(1_\Delta, K^{H_\infty}_1) \leq 2\|H_\infty\|_{\mathrm{HS},\tau}$. 
\end{theorem}
\begin{proof}
    Since $\varphi_n$ and $H_n$ form a Cauchy sequence under the Hofer--Shelukhin norm, by \Cref{thm:hofer-norm}, we know that the sheaf quantizations $K^{H_n}_1$ form a Cauchy sequence under the interleaving distance. By \Cref{thm:complete}, there exists a sheaf kernel $K^{H_\infty}_1 = \lim_{n\to\infty}K^{H_n}_1$. \Cref{thm:ss-limit} implies that 
    \[
        \cSS^\infty(K^{H_\infty}_1) \subseteq \bigcap\nolimits_{n\geq 1}\overline{\bigcup\nolimits_{k\geq n}\Gamma^{\varphi_k}}.
    \]
    Moreover, by \Cref{thm:hofer-norm}, we know that $d_{\tau}(1_{\Delta}, K^{H_\infty}_1) \leq \|H_\infty\|_{\mathrm{HS},\tau}$.
\end{proof}

\begin{theorem}\label{thm:Hofer-quantize-invert}
    Let $\tau \colon S^*M \to \bR$ be a positive Hamiltonian. Let $\varphi_n \in \operatorname{Cont}_0(S^*M, \xi_{\mathrm{std}})$ be the time-$1$ maps of the contact Hamiltonians flows defined by $H_n$, each of which has bounded conformal factor, that uniformly converges to $H_\infty$, and $\varphi_n$ uniformly converges under the $C^0$-topology so that $\ol{H}_n$ uniformly converges to $\ol{H}_\infty$. Then there exists a sheaf quantization $K^{H_\infty}_1$ and $K^{\ol{H}_\infty} \in \Sh(M \times M)$ such that 
    \[K^{\ol{H}_\infty}_1 \circ K^{H_\infty}_1 \simeq K^{H_\infty}_1 \circ K^{\ol{H}_\infty}_1 \simeq 1_\Delta.\]
\end{theorem}
\begin{proof}
    Since $\varphi_n$ converges to a continuous map, $\ol{H}_n(p, t) = -H_n(\varphi_n(p), t)$ also converges to a continuous function $\ol{H}_\infty$, and there exists a sheaf kernel $K^{\ol{H}_\infty}_1$ such that $K^{\ol{H}_n}_1 \to K^{\ol{H}_\infty}_1$. We have $K^{\ol{H}_n}_1 \circ K^{H_\infty}_1 \to K^{\ol{H}_\infty}_1 \circ K^{H_\infty}_1$ since the interleaving distance on the product is right invariant. On the other hand, we can show that $K^{\ol{H}_n}_1 \circ K^{H_m}_1 \to K^{\ol{H}_n}_1 \circ K^{H_\infty}_1$ as follows. When $(K^{H_m}_1, K^{H_{m'}}_1)$ are $(a, b)$-interleaved, we know that there are morphisms such that the compositions
    \begin{gather*}
    K^{\ol{H}_n}_1 \circ K^\tau_{-a} \circ K^{H_m}_1 \to K^{\ol{H}_n}_1 \circ K^{H_{m'}}_1 \to K^{\ol{H}_n}_1 \circ K^\tau_{b} \circ K^{H_m}_1,  \\
    K^{\ol{H}_n}_1 \circ K^\tau_{-b} \circ K^{H_{m'}}_1 \to K^{\ol{H}_n}_1 \circ K^{H_{m}}_1 \to K^{\ol{H}_n}_1 \circ K^\tau_{a} \circ K^{H_{m'}}_1
    \end{gather*}
    are the continuation morphisms. Then by applying \Cref{lem:conformal-reeb-commute}, we have $\varphi^{\tau_n}_t \circ \varphi_{n}^{-1} = \varphi_{n}^{-1} \circ \varphi^\tau_t$, where $\tau_n$ is the Reeb vector field associated to $\varphi_n^*\alpha$. Therefore, by \Cref{prop:banach-mazur}, when the conformal factor of $\varphi_n$ is bounded by $h_n$, we get
    \begin{gather*}
    K^\tau_{-e^{h_n}a}\circ K^{\ol{H}_n}_1 \circ K^{H_m}_1 \to K^{\ol{H}_n}_1 \circ K^{H_{m'}}_1 \to K^\tau_{e^{h_n}b} \circ K^{\ol{H}_n}_1 \circ K^{H_m}_1,  \\
    K^\tau_{-e^{h_n}b} \circ K^{\ol{H}_n}_1 \circ K^{H_{m'}}_1 \to K^{\ol{H}_n}_1 \circ K^{H_{m}}_1 \to K^\tau_{e^{h_n}a} \circ K^{\ol{H}_n}_1 \circ K^{H_{m'}}_1
    \end{gather*}
    whose compositions are continuation morphisms. Thus, given $n \in \bN$, for any $H_m$ and $H_{m'}$, 
    \[d_{\tau}(K^{\ol{H}_n}_1 \circ K^{H_m}_1, K^{\ol{H}_n}_1 \circ K^{H_{m'}}_1) \leq e^{h_n}\,d_{\tau}(K^{H_m}_1, K^{H_{m'}}_1).\]
    Hence, we know that for any given $n \in \bN$, $K^{\ol{H}_n}_1 \circ K^{H_m}_1 \to K^{\ol{H}_n}_1 \circ K^{H_\infty}_1$, and 
    \[d_{\tau}(1_\Delta, K^{\ol{H}_n}_1 \circ K^{H_\infty}_1) \leq 2\|\ol{H}_n \# H_\infty\|_{\mathrm{HS},\tau}.\]
    Since $\ol{H}_n \# H_\infty \to 0$ as $n \to \infty$, we know $K^{\ol{H}_n}_1 \circ K^{H_\infty}_1 \to 1_\Delta$. Therefore, $d_{\tau}(1_\Delta, K^{\ol{H}_\infty}_1 \circ K^{H_\infty}_1) = 0$. By \Cref{prop: ss-is-same-kernel,thm:non-deg-0}, we can conclude that $K^{\ol{H}_\infty}_1 \circ K^{H_\infty}_1 = 1_\Delta$.
\end{proof}

\begin{remark} 
    One can, instead of considering the interleaving distance defined by $\tau_2$ throughout the proof, apply the interleaving distance defined by $\tau_2$ to show the convergence $K^{H_n}_1 \to K^{H_\infty}_1$ and the distance defined by $\tau_1$ to show the convergence $K^{\ol{H}_n}_1 \to K^{\ol{H}_\infty}_1$. Then, define the bi-interleaving distance such that $(F, G)$ are $(a, b)$-bi-interleaved if there are morphisms
    \[
    K^{\tau}_{-a} \circ F \circ K^{\tau}_{-a} \to G \to K^{\tau}_{b} \circ F \circ K^{\tau}_{b}, \quad K^{\tau}_{-b} \circ G \circ K^{\tau}_{-b} \to F \to K^{\tau}_{a} \circ G \circ K^{\tau}_{a}.
    \]
    If we can show that $d(1_\Delta, K^{H_\infty}_1 \circ K^{\ol{H}_\infty}_1) = 0$ under the bi-interleaving distance, then maybe it is possible to prove $K^{H_\infty}_1 \circ K^{\ol{H}_\infty}_1 = 1_\Delta$ without any bound on the conformal factors. However, while completeness of the bi-interleaving distance still holds for formal reason, this strategy does not work because 
    \Cref{cor: ss-is-same,thm:non-deg-0} do not work for the bi-interleaving distance. Indeed, when considering convolutions on both sides $K^{\tau} \circ_I F \circ_I K^\tau$, the singular support may no longer be $I$-non-characteristic as the estimation \Cref{eq: non-characteristic} may fail.
\end{remark}

\begin{remark}
    The above theorem in particular constructs sheaf quantization for topological contact dynamical systems in the sense of M\"uller--Spaeth, which are contact homeomorphisms such that the both contact Hamiltonians and the conformal factors uniformly converge \cite{MullerSpaeth}. However, note that our theorem does not even require the contactomorphisms to converge to a homeomorphism, but rather a continuous map (and we only require the conformal factors to be uniformly bounded). 
\end{remark}

We can also construct sheaf quantizations for Reeb flows of limits of contact forms where the conformal factor uniformly converges.

\begin{proposition}
    Let $M$ be a closed manifold and $\tau \colon S^*M \to \bR$ and $\tau'_n \colon S^*M \to \bR$ be positive Hamiltonians that define the Reeb flows for the contact forms $\alpha$ and $\alpha'_n$. Suppose $\tau_n' \to \tau_\infty'$ uniformly. Then there exists a sheaf quantization $K^{\tau_\infty'}_1 \in \Sh(M \times M)$ with
    \[
    \cSS^\infty(K^{\tau_\infty'}_1) \subseteq \bigcap\nolimits_{n\geq 1}\overline{\bigcup\nolimits_{k\geq n}\Gamma^{\tau_k'}_1}.
    \]
    When $\varphi^{\tau_n}_1$ uniformly converges under the $C^0$-topology, $K^{-\tau'_\infty}_1 \circ K^{\tau'_\infty}_1 = 1_\Delta$. 
\end{proposition}

Now, we can easily prove the non-degeneracy of the Hofer--Shelukhin norm for cosphere bundles $S^*M$. This is proved for general closed contact manifolds by Shelukhin \cite{Shelukhin17}.

\begin{theorem}[Theorem~\ref{thm:main-Hofer-nondeg}~(1)]\label{thm:Hofer-nondeg}
    Let $M$ be a closed manifold. Then the Hofer--Shelukhin metric $d_{\mathrm{HS},\alpha}$ is non-degenerate on $\mathrm{Cont}_0(S^*M, \xi_{\mathrm{std}})$.
\end{theorem}
\begin{proof} 
    Let $\varphi, \varphi'$ be contactomorphisms in the identity component such that $d_{\HS,\tau}(\varphi, \varphi') = 0$. Since $d_{\HS,\tau}(\varphi,\varphi') = d_{\HS,\tau}(\varphi \circ (\varphi')^{-1}, \id)$ \cite[Remark 6]{Shelukhin17}, we may assume $\varphi' = \id$. Then there exists a sequence of contact isotopies $\Phi_n$ induced by the Hamiltonians $H_n$ such that the time-1 maps $\varphi_n = \varphi$ and $\|H_n\|_{\mathrm{HS},\tau} \to 0$.
   By \Cref{thm:Hofer-quantize}, there exists a sheaf quantization $K^{H_\infty}_1 \in \Sh(M \times M)$ such that $d_\tau(1_\Delta, K^{H_\infty}_1) = 0$. 
    By \Cref{prop: ss-is-same-kernel}, this implies that $\cSS(1_\Delta) = \cSS(K^{H_\infty}_1)$. 
    By the singular support estimation in \Cref{thm:ss-limit,thm:Hofer-quantize}, this means that
     \[
        \Gamma^{\id} \subseteq \bigcap\nolimits_{n\geq 1}\overline{\bigcup\nolimits_{k\geq n}\Gamma^{\varphi_k}} = \Gamma^\varphi.
    \]
     Therefore, we can conclude that $\varphi = \id$.
\end{proof}

Let $\Lambda \subseteq S^*M$ be a Legendrian and $\operatorname{Leg}_0(\Lambda)$ be the space of Legendrians that are Legendrian isotopic to $\Lambda$. We recall from Rosen--Zhang \cite{RosenZhang20Dichotomy} that the Chekanov--Hofer--Shelukhin norm on $\operatorname{Leg}_0(\Lambda)$ is
\[
    d_{\mathrm{CHS},\alpha}(\Lambda_0, \Lambda_1) = \inf_{H: \varphi^H_1(\Lambda_0) = \Lambda_1}\|H\|_{\mathrm{CHS},\Lambda,\tau} = \inf_{H: \varphi^H_1(\Lambda_0) = \Lambda_1} \int_0^1 \sup_{(x, \xi) \in \varphi_s^H(\Lambda)}|H_s(x, \xi)/\tau(x, \xi)| \, ds.
\] 
They proved that the distance is either non-degenerate or zero. Later, Cant \cite{Cant23Legendrian} and Nakamura \cite{Nakamura23Legendrian} gave examples of Legendrians in certain open contact manifolds (coming from the product with the symplectization of overtwisted contact manifolds) where the metric is zero. Usher \cite{Usher21Conformal} proved non-degeneracy of the distance for hypertight Legendrians in hypertight contact manifolds, Hedicke \cite{Hedicke24Lorentzian} proved non-degeneracy of the distance for Legendrians that do not admit positive loops, and finally Dimitroglou Rizell--Sullivan \cite{DRG24Rabinowitz} proved the case for Legendrians in closed contact manifolds. 

Here, we can prove the non-degeneracy of the Chekanov--Hofer--Shelukhin norm for a special class of closed Legendrians $\Lambda \subseteq S^*M$ which admits nontrivial sheaves.

\begin{theorem}\label{thm:chekanov-hofer-shelukhin}
    Let $\Lambda \subseteq S^*M$ be a closed connected Legendrian and suppose that $\Sh_\Lambda(M) \supsetneq \Loc(M)$ contains sheaves that are not locally constant. Then the Chekanov--Hofer--Shlukhin metric $d_{\mathrm{CHS},\alpha}$ is non-degenerate on $\mathrm{Leg}_0(\Lambda)$.  
\end{theorem}
\begin{proof}
    Let $\Lambda'$ be a Legendrian that is isotopic to $\Lambda$ such that $d(\Lambda,\Lambda') = 0$. Then there is a sequence of contact isotopies $\Phi_n$ induced by the contact Hamiltonian $H_n \colon S^*M \times I \to \bR$ such that $\|H_n\|_{\mathrm{CHS},\tau} \to 0$ and the time-$1$ maps satisfy $\varphi_n(\Lambda) = \Lambda'$. 
    Then by the isotopy extension theorem \cite[Proposition 2.41]{Geiges-intro-contact}, we can replace the contact Hamiltonian function by $H_n \colon S^*M \times [0, 1] \to \bR$ such that
    \[
        \varphi_n(\Lambda) = \Lambda', \quad \|H_n\|_{\mathrm{HS},\tau} = \|H_n\|_{\mathrm{CHS},\Lambda,\tau} \to 0.
    \]
    Consider $F \in \Sh_\Lambda(M)$ that is not local constant. Since $\Lambda$ is a connected Legendrian, by \cite[Theorem 7.2.1]{KS90} the microstalk is locally constant, and thus $\cSS^\infty(F) \subseteq \Lambda$ implies that $\cSS^\infty(F) = \Lambda$. Since $H_n \rightarrow H_\infty = 0$, by \Cref{thm:Hofer-quantize}, there exists a sheaf quantization $K^{H_\infty}_1 \in \Sh(M \times M)$ such that $d_\tau(F, K^{H_\infty}_1 \circ F) = 0$. 
    However, by \Cref{cor: ss-is-same}, this implies that $\cSS(F) = \cSS(K^{H_\infty}_1 \circ F)$. By the singular support estimation in \Cref{thm:ss-limit,thm:Hofer-quantize}, this means that
    \[
    \Lambda \subseteq \bigcap\nolimits_{n\geq 1}\overline{\bigcup\nolimits_{k\geq n}\varphi_{k}(\Lambda)} = \Lambda'.
    \]
    Therefore, we can conclude that $\Lambda = \Lambda'$.
\end{proof}

\begin{remark}
    We expect that one could show the non-degeneracy of the Chekanov--Hofer--Shelukhin for general closed Legendrians by considering the doubling construction \Cref{thm:doubling}, whose endomorphism should be analogous to the subcomplex within small action windows in \cite{DRG24Rabinowitz}. However, since it is hard to control the effect of the contact isotopy on the Reeb push-off $\Lambda^\tau_\epsilon$, some additional work needs to be done, and thus in this paper we do not proceed in this direction.
\end{remark}

\section{\texorpdfstring{$C^0$}{C0}-Distances and Sheaves}\label{sec:C0distance}

In this section, we will discuss the relation between the interleaving distance of sheaves and the $C^0$-topology of the contactomorphism group. We will prove a new sheaf quantization theorem for any $C^0$-small contactomorphism \Cref{thm:main-C0-distance} (\Cref{thm:quantization-C0-small}) and use it to deduce \Cref{thm:main-rigidity,thm:main-Maslov,thm:main-sheaf-invariance} (\Cref{thm:sheaf-invariance-summary,thm:rigidity,thm:maslov-J,thm:coiso-rigidity}). 

\subsection{Sheaf quantization of contact homeomorphisms in jet bundles}
In this section, we first prove a simple case of the result on sheaf quantizations for the $C^0$-limits of compactly supported contactomorphisms in the 1-jet bundles. We hope to provide the reader with some intuition on the sheaf quantization results before going into more complicated constructions. The argument is independent of the rest of the subsections.

We use the standard embedding $J^1M \hookrightarrow S^*(M \times \bR)$ by $(x, \xi, z) \mapsto (x, z; \xi, 1)$ to translate the problem into a problem about certain contactomorphisms in $S^*(M \times \bR)$. Moreover, in this subsection, we will only consider the standard Reeb flow defined by $T_t \colon (x, \xi, z) \mapsto (x, \xi, z+t)$.

Our key observation is that, in 1-jet bundles, after suitable homotopies, the $C^0$-norm of the contact isotopy in the $z$-direction can be controlled by the $C^0$-norm of the time-1 map. Similar observation is also used in for example \cite[Section 4.1]{BHS2022}.

\begin{lemma}\label{lem:contact_rescale_2}
    Let $M$ be a complete Riemannian manifold and $\bR$ be the standard Riemannian manifold with the Euclidean metric. For any compactly supported contact isotopy $\varphi_t \in \operatorname{Cont}_{0,c}(J^1M, \xi_\mathrm{std})$, the contact isotopy is homotopic to $\varphi_t' \in \operatorname{Cont}_{0,c}(J^1M, \xi_\mathrm{std})$ such that $\operatorname{supp}(\varphi_t') \subseteq T^*M \times (-r, r)$ and $d_{C^0,z}(\id, \varphi'_t) \leq d_{C^0,z}(\id, \varphi)$.
\end{lemma} 
\begin{proof}
    Since the contact isotopy is compactly supported, we may assume that $d_{C^0,z}(\id, \varphi_t) \leq R_t$ where $R_t$ is a smooth function and $R_1 = r$. Consider the contactomorphism $\rho_{r,R} \colon T^*M \times \bR \to T^*M \times \bR, (x, \xi, z) \mapsto (x, r\xi/R, rz/R)$. We can define the 2-parametric family of contact isotopies
    \[
    \varphi_{s,t} = \rho_{sR_t +(1-s)r,R_t} \circ \varphi_t \circ \rho_{sR_t+(1-s)r,R_t}^{-1}.
    \]
    Then since $d_{C^0,z}(\id, \varphi_{0,t}) \leq R_t$, we can conclude that $d_{C^0,z}(\id, \varphi_{1,t}) \leq r$. In fact, $\varphi_t(x, R_t\xi/r, R_tz/r) \in T^*M \times (R_tz/r -R_t, R_tz/r + R_t)$, so therefore
    \[
    \rho_{r,R_t} \circ \varphi_t \circ \rho_{r,R_t}^{-1}(x, \xi, z) \in T^*M \times (z - r, z + r).
    \]
    This then completes the proof.
\end{proof}

Using the above lemma, we can prove the $C^0$-continuity of the interleaving distance between sheaf kernels of the contactomorphisms in the identity component. 

For a compactly supported contact isotopy $\Phi \colon J^1M \times I \to J^1M$, it extends to a contact isotopy $\Phi \colon S^*(M \times \bR) \times I \to S^*(M \times \bR)$, and thus there exists a canonical sheaf quantization by \Cref{thm: GKS}. The Reeb flow $T_t \colon J^1M \times I \to J^1M$ is not compactly supported and does not extend to $S^*(M \times \bR)$. Nevertheless, one can set the sheaf quantization to be
\[
T_t \coloneqq 1_{\Delta_M \times \{(z_1, z_2) \mid 0 < z_2 - z_1 < t\}}[1].
\]
For any $F \in \Sh_{\zeta > 0}(M \times \bR_z)$, we know that $T_t \circ \cSS^\infty(F) = T_t(\cSS^\infty(F))$.\footnote{The sheaf kernels are not uniquely determined by the above condition, but they are unique in the localization. For instance, by composing with the microlocal kernel \Cref{not: proj} $P_\zeta \colon \Sh(M \times \bR_z) \to \Sh_{\zeta \geq 0}(M \times \bR_z)$, the sheaf kernel $1_{\Delta_M \times \{(z_1, z_2) \mid 0 < z_2 - z_1 < t\}}$ can be identified with the sheaf kernel $1_{\Delta_M \times \{(z_1, z_2) \mid z_2 - z_1 < t\}}$.}

\begin{proposition}\label{prop:continuity-kernel}
    Let $M$ be a complete Riemannian manifold and $\bR$ be the standard Riemannian manifold with the Euclidean metric. Then for any compactly supported $\varphi \in \operatorname{Cont}_{0,c}(J^1M, \xi_\mathrm{std})$, the sheaf quantization $K^\varphi$ satisfies the relation that $d_\tau(1_\Delta, K^\varphi) \leq 2\,d_{C^0}(\id, \varphi)$.
\end{proposition}
\begin{proof}
    Consider a contactomorphism $\varphi$ such that $d_{C^0}(\id, \varphi) \leq \epsilon$. We show that there are canonical morphisms whose compositions
    \[
    T_{-\epsilon} \to K^\varphi \to T_\epsilon, \quad T_{-\epsilon} \circ K^\varphi \to 1_\Delta \to T_\epsilon \circ K^\varphi
    \]
    are equivalent to the continuation morphisms. In fact, by \Cref{lem:contact_rescale_2}, we may assume that $\varphi$ is induced by a contact isotopy $\varphi_t$ such that $d_{C^0,z}(\id, \varphi) \leq \epsilon$. Then we know that $\Gamma^{\Phi}_t \cap \Gamma^T_{\pm \epsilon} = \varnothing$, in other words $\cSS^\infty(K^{\Phi}_t) \cap \cSS^\infty(T_{\pm \epsilon}) \cap S^*_{\zeta>0}(M \times \bR) = \varnothing$, and $\cSS^\infty(K^{\Phi}_t) \cap S^*_{\zeta \leq 0}(M \times \bR)$ is a constant family. Thus, $\cSS^\infty(K^{\Phi}_t) \cap \cSS^\infty(T_{\pm \epsilon}) $ defines a compactly supported Legendrian isotopy. It follows from \Cref{thm: GKS} that
    \begin{gather*}
    \Hom(K^\varphi, T_\epsilon) = \Hom(1_\Delta, T_\epsilon) = \Gamma(M, 1_M),\\
    \Hom(T_{-\epsilon}, K^\varphi) = \Hom(T_{-\epsilon}, 1_\Delta) = \Gamma(M, 1_M).
    \end{gather*}
    This ensures that there exist morphisms such that the compositions are equivalent to continuation morphisms, and thus $d_\tau(1_\Delta, K^\varphi) \leq 2 \epsilon$.
\end{proof}

\begin{proposition}
    Let $\varphi_n \in \operatorname{Cont}_{0,c}(J^1M, \xi_\mathrm{std})$ be compactly supported contactomorphisms such that $\varphi_n \to \varphi_\infty$ in the $C^0$-topology where $\varphi_\infty$ is a homeomorphism. Then there is a functor 
    \[K^{\varphi_\infty} \colon \Sh_{\Lambda}(M \times \bR) \to \Sh_{\varphi_\infty(\Lambda)}(M \times \bR)\]
    such that $K^{\varphi_n}$ converges to $K^{\varphi_\infty}$.
\end{proposition}
\begin{proof}
    Without loss of generality, we may assume that $d_{C^0}(\varphi_n, \varphi_{n+1}) \leq \epsilon_n$ and $\sum_{n=1}^\infty \epsilon_n < +\infty$. Since $\varphi_n$ is a diffeomorphism, we know $d_{C^0}(\varphi_n \circ \varphi_{n+1}^{-1}, \id) = d_{C^0}(\varphi_n, \varphi_{n+1}) \leq \epsilon_n$. By \Cref{prop:continuity-kernel}, we know that there exist morphisms
    \begin{gather*}
    K^\tau_{-\epsilon_n} \circ K^{\varphi_{n}^{-1} \circ \varphi_{n+1}} \to 1_\Delta \to K^\tau_{\epsilon_n} \circ K^{\varphi_{n}^{-1} \circ \varphi_{n+1}},\quad
    K^\tau_{-\epsilon_n} \to K^{\varphi_{n}^{-1} \circ \varphi_{n+1}} \to K^\tau_{\epsilon_n}
    \end{gather*} 
    whose compositions are the natural continuation morphisms. Then, we know that there exist morphisms
    \begin{gather*}
    K^{\tau_n}_{-\epsilon_n} \circ K^{\varphi_{n+1}} \to K^{\varphi_n} \to K^{\tau_n}_{\epsilon_n} \circ K^{\varphi_{n+1}},\\
    K^{\tau_n}_{-\epsilon_n} \circ K^{\varphi_n} \to K^{\varphi_{n+1}} \to K^{\tau_n}_{\epsilon_n} \circ K^{\varphi_n}
    \end{gather*}
    whose compositions are the natural continuation morphisms. Then, we know that $K^{\varphi_n}$ form a Cauchy sequence and by \Cref{thm:complete}, $K^{\varphi_n}$ converges to a sheaf $K^{\varphi_\infty}$ under the interleaving distance. Finally, by \Cref{thm:ss-limit}, we have
    \[
    \cSS^\infty(K^{\varphi_\infty} \circ F) \subseteq \bigcap\nolimits_{n\geq 0}\overline{\bigcup\nolimits_{k\geq n} \cSS^\infty(K^{\varphi_k} \circ F)} = \bigcap\nolimits_{n\geq 0}\overline{\bigcup\nolimits_{k\geq n} \varphi_k(\cSS^\infty(F))} = \varphi_\infty(\cSS^\infty(F)).
    \]
    This completes the proof.
\end{proof}

\subsection{Sheaf quantization of nearby Legendrians with no chords}

We generalize the sheaf quantization result of Guillermou \cite{Guillermou23} for closed Legendrians in the 1-jet bundle with no Reeb chords (equivalently, closed exact Lagrangians in the cotangent bundles) in \Cref{thm:guillermou-pre,thm:guillermou}. In this section, we follow \Cref{not:legendrian_movie}.

We will also need a generalization to certain non-compact Legendrians, following \cite{Li23Cobordism2,AsanoIkeLi25}. Therefore, we will consider a complete Riemannian metric on the manifold $M$, which induces a complete Riemannian metric on $S^*M$. We will take the standard contact form $\alpha_{\mathrm{std}}$ on $S^*M$ induced by the Riemannian metric and all the contact Hamiltonian functions are the Hamiltonian with respect to this standard contact form $\alpha_{\mathrm{std}}$.

\begin{definition}
    Let $M$ be a complete Riemannian manifold and $\Lambda \subseteq S^*_{\zeta>0}(M \times \bR)$ be a properly embedded Legendrian. Then for a bounded positive contact Hamiltonian $\tau$ in \Cref{def: bound-reeb}, the associated contact flow on $S^*_{\zeta>0}(M \times \bR)$ is called a \emph{separating Reeb flow} for $\Lambda$ if there is $s > 0$ such that $\Lambda$ and $\Lambda^\tau_s$ are separated by a hypersurface $T^*M \times a$, and $\bigcup_{\epsilon<t<s}\Lambda^\tau_t$ is disjoint from some tubular neighborhood of $\Lambda$ with positive radius, for some $\epsilon > 0$.
\end{definition}

\begin{lemma}\label{lem:separate-reeb-flow}
    Let $M$ be a complete Riemannian manifold, $\Lambda \subseteq S^*_{\zeta>0}(M \times \bR)$ be a properly embedded Legendrian and $\tau$ be a $C^1$-bounded positive Hamiltonian defining the separating Reeb flow for $\Lambda$. Then there exists a cut-off function $\rho$ on a tubular neighborhood of $\Lambda$ such that $(1-\rho)\tau$ defines a complete contact flow and sends $\Lambda \cup \Lambda^\tau_\epsilon$ to $\Lambda \cup \Lambda^\tau_s$.
\end{lemma}
\begin{proof}
    Consider the tubular neighborhood of $\Lambda$ of positive radius with respect to the Riemannian metric that is disjoint from $\bigcup_{\epsilon<t<s}\Lambda^\tau_t$. We can choose a cut-off function $\rho$ supported in the tubular neighborhood such that $|d\rho|$ is uniformly bounded. Then since $|(1-\rho)\tau| \leq |\tau|$ and $|d((1-\rho)\tau)| \leq |\tau| |d\rho| + |d\tau|$ are both uniformly bounded, the norm of the contact vector field $X_{(1-\rho)\tau}$ is uniformly bounded and therefore defines a complete contact flow.
\end{proof}

We can now state the following theorem, which generalizes of Guillermou's results to the non-compact setting when there exists a separating Reeb flow.

\begin{theorem}[Guillermou {\cite[Corollaries~12.3.2 \& 12.3.3, Theorems~12.4.3 \& 12.4.4]{Guillermou23}}]\label{thm:guillermou-pre}
    Let $M$ be a complete Riemannian manifold and $\Lambda \subseteq S^*_{\zeta>0}(M \times (-a, b))$ be a properly embedded Legendrian with no Reeb chords with respect to some separating Reeb flow. Then there exists a fully faithful embedding where the second functor is give by restriction at $+\infty$
    \[
        \Psi_{\Lambda,\infty} \colon \msh_\Lambda(\Lambda) \xhookrightarrow{\Psi_\Lambda} \operatorname{Sh}_\Lambda(M \times \bR)_0 \xhookrightarrow{i_\infty^*} \operatorname{Loc}(M),
    \]
    where $\operatorname{Sh}_\Lambda(M \times \bR)_0$ is the subcategory of sheaves with trivial stalks at $-\infty$.
\end{theorem}
\begin{proof}
    First, by \Cref{thm:doubling}, we know that for sufficiently small $\epsilon > 0$, there exists a fully faithful embedding
    \[
        w_\Lambda \colon \msh_\Lambda(\Lambda) \hookrightarrow \Sh_{\Lambda \cup \Lambda^\tau_\epsilon}(M \times \bR)_0.
    \]
    Consider the separating Reeb flow for $\Lambda$ defined by $\tau$. Using \Cref{lem:separate-reeb-flow}, we know there is a cut-off function $\rho$ supported in the tubular neighborhood and $(1-\rho)\tau$ defines a complete contact flow that sends $\Lambda \cup \Lambda^\tau_\epsilon$ to $\Lambda \cup \Lambda^\tau_s$. By the sheaf quantization of Guillermou--Kashiwara--Schapira \Cref{thm: GKS}, we have
    \[
        w_{\Lambda,s} \colon \msh_\Lambda(\Lambda) \hookrightarrow \Sh_{\Lambda \cup \Lambda^\tau_\epsilon}(M \times \bR)_0 \xrightarrow{\sim} \Sh_{\Lambda \cup \Lambda^\tau_s}(M \times \bR)_0.
    \]
    Since $\Lambda$ and $\Lambda^\tau_s$ are separated by a hypersurface $T^*M \times a$, we can restrict the sheaf to $M \times (-\infty, a)$, push-forward to $M \times (-\infty, +\infty)$ via a diffeomorphism and then restrict to $M \times \infty$.  
    This gives the functor
    \[
        \Sh_{\Lambda \cup \Lambda^\tau_s}(M \times \bR)_0 \rightarrow \operatorname{Sh}_\Lambda(M \times \bR)_0 \rightarrow \operatorname{Loc}(M).
    \]
    
    Full faithfulness of the functor will follow from perturbation of $\Lambda$ by the separating Reeb flow and microlocal Morse lemma. First, consider sheaves $w_{\Lambda,\epsilon}F, w_{\Lambda,\epsilon}G \in \Sh_{\Lambda \cup \Lambda^\tau_\epsilon}(M \times \bR)_0$. Denote their images in $\Sh_\Lambda(M \times \bR)_0$ by $F$ and $G$. We can use the Reeb flow $\varphi_{(1-\rho)\tau}^s$ to perturb $w_{\Lambda,\epsilon}G$ by \Cref{lem: pert} and apply microlocal Morse lemma. This implies
    \begin{align*}
        \Hom(w_{\Lambda,\epsilon}F, w_{\Lambda,\epsilon}G) &= \Hom(w_{\Lambda,\epsilon}F, w_{\Lambda,s}G) \simeq \Hom(w_{\Lambda,\epsilon}F|_{(-\infty,a)}, w_{\Lambda,s}G|_{(-\infty,a)}) \\
        &= \Hom(w_{\Lambda,\epsilon}F|_{(-\infty,a)}, \Psi_{\Lambda}G|_{(-\infty,a)}) = \Hom(w_{\Lambda,\epsilon}F, \Psi_\Lambda G).
    \end{align*}
    Then we can apply the Reeb flow $\varphi_{(1-\rho)\tau}^s$ to perturb $w_{\Lambda,\epsilon}F$ and $G$ by \Cref{thm: GKS} and apply microlocal Morse lemma again. This implies
    \begin{align*}
        \Hom(w_{\Lambda,\epsilon}F, \Psi_\Lambda G) &= \Hom(w_{\Lambda,s}F, \Psi_\Lambda G) \simeq \Hom(w_{\Lambda,s}F|_{(-\infty,a)}, \Psi_\Lambda G|_{(-\infty,a)}) \\
        &= \Hom(\Psi_{\Lambda}F|_{(-\infty,a)}, \Psi_\Lambda G|_{(-\infty,a)}) = \Hom(\Psi_\Lambda F, \Psi_\Lambda G).
    \end{align*}
    This implies full faithfulness of $\Sh_{\Lambda \cup \Lambda^\tau_s}(M \times \bR)_0 \rightarrow \operatorname{Sh}_\Lambda(M \times \bR)_0$. Then, we consider the Reeb flow $\varphi^\tau_s$ by \Cref{lem: pert} and apply the microlocal Morse lemma to $\Psi_\Lambda F, \Psi_\Lambda G \in \Sh_\Lambda(M \times \bR)_0$. Denote their images in $\Loc(M)$ by $\Psi_{\Lambda,\infty}F$ and $\Psi_{\Lambda,\infty}G$. This implies
    \begin{align*}
        \Hom(\Psi_\Lambda F, \Psi_\Lambda G) &= \Hom(\Psi_\Lambda F, K^{\tau}_{s} \circ \Psi_\Lambda G) \simeq \Hom(\Psi_\Lambda F|_{(a,+\infty)}, K^{\tau}_{s} \circ \Psi_\Lambda G|_{(a,+\infty)}) \\
        &= \Hom(\Psi_{\Lambda,\infty}F \boxtimes 1_\bR, \Psi_\Lambda G) \simeq \Hom(\Psi_{\Lambda,\infty}F \boxtimes 1_\bR|_{(a,+\infty)}, \Psi_\Lambda G|_{(a,+\infty)}) \\
        &= \Hom(\Psi_{\Lambda,\infty}F, \Psi_{\Lambda,\infty}G).
    \end{align*}
    This implies the full faithfulness of $\operatorname{Sh}_\Lambda(M \times \bR)_0 \rightarrow \operatorname{Loc}(M)$. Finally, we show that it preserves microlstalks using \cite[Theorem 13.5.1]{Guillermou23} and \cite[Corollary 1.3]{Jin20Jhomomorphism}.
\end{proof}

    Given the above theorem, the following corollary immediately follows from exactly the same argument as in Guillermou \cite{Guillermou23} and Jin \cite{Jin20Jhomomorphism}:

\begin{theorem}[Guillermou {\cite[Theorem 13.5.1, Proposition 13.5.2 \& Corollary 13.5.3]{Guillermou23}, Jin \cite[Corollary 1.3]{Jin20Jhomomorphism}}]\label{thm:guillermou}
    Let $M$ be a complete Riemannian manifold and $\Lambda \subseteq S^*_{\zeta>0}(M \times (-a, b))$ be a properly embedded Legendrian with no Reeb chords for some separating Reeb flow. Then the Maslov data of $\Lambda$ is trivial, and there is an equivalence that preserves stalks
    \[
        \Psi_{\Lambda,\infty} \colon \operatorname{Loc}(\Lambda) \xrightarrow{\sim} \operatorname{Sh}_\Lambda(M \times \bR)_0 \xrightarrow{\sim} \operatorname{Loc}(M).
    \]
    In particular, the projection $\Lambda \to M$ is a homotopy equivalence and there exists a unique sheaf $F \in \operatorname{Sh}_\Lambda(M \times \bR)$ such that
    \begin{enumerate}
        \item $\supp(F) \subseteq M \times (-a,+\infty)$, $\cSS^\infty(F) = \Lambda$;
        \item $m_\Lambda(F) = 1_\Lambda$ and $i_{M\times b}^*F = 1_M$ for $b \gg 0$.\footnote{We fix the Maslov data on $\Lambda$ that is the pull-back of the Maslov data on the diagonal $M$ via the natural projection map (which is known to be a homotopy equivalence).}
    \end{enumerate}
\end{theorem}
\begin{remark}
    We remind the reader how the argument works. First, we show that $\Lambda \to M$ is $\pi_1$-injective by \cite[Proposition 13.1.1]{Guillermou23}. Second, we show that the Maslov class is zero by \cite[Theorem 13.2.1]{Guillermou23}. Third, we show that for $\cC = \Mod(\Bbbk)$, the functor $\msh_\Lambda(\Lambda) \to \Loc(M)$ sends simple local systems to simple local systems and thus $\pi\colon \Lambda \to M$ induces isomorphisms on cohomology by \cite[Proposition 13.3.1]{Guillermou23}. Next, we can show that the second Stiefel--Whitney class is zero \cite[Proposition 13.4.2]{Guillermou23}. This implies that when $\cC = \Mod(\Bbbk)$, the functor $\Loc(\Lambda) \to \Loc(M)$ is an equivalence \cite[Proposition 13.5.2]{Guillermou23}, and thus $\Lambda \to M$ is a homotopy equivalence. Then, since $\Lambda \to M$ is a homotopy equivalence, we can apply the argument in \cite[Corollary 1.3]{Jin20Jhomomorphism} and conclude that the Maslov data is trivial and $\Loc(\Lambda) \to \Loc(M)$ is an equivalence over any coefficients.
\end{remark}

\begin{remark}\label{rem: noncompact guillermou}
    Unlike the case when $M$ is compact, it is not true that any embedded Legendrian $L \subseteq S^*_{\zeta>0}(M \times \bR)$ with no Reeb chords admits a separating Reeb flow. For instance, most non-compact embedded Legendrians considered in \cite[Section 2 \& 3]{Li23Cobordism2} with no Reeb chords have no separating Reeb flows. (While one can construct a separating Legendrian isotopy for those Legendrians with no Reeb chords, the Legendrian isotopy may not extend to a contact isotopy due to non-compactness, and even when it extends to a contact isotopy, it is in general not separating \cite[Section 2 \& 3]{Li23Cobordism2}.) 
\end{remark}

The following corollary of Guillermou's result is also well known to experts. Here, we use the isomorphism $\pi_* \colon \operatorname{Loc}(\Lambda) \to \operatorname{Loc}(M)$ given by the natural projection $\pi \colon \Lambda \to M$, which is a homotopy equivalence by Abouzaid--Kragh \cite{Kragh13Parametrized} and Guillermou \cite{Guillermou23} for compact $M$, and is extended in the above result for non-compact $M$ when there exists a separating Reeb flow.

\begin{corollary}\label{cor:guillermou-quant}
    Let $\Lambda \subseteq S^*_{\zeta>0}(M \times (-a, b))$ be a properly embedded Legendrian with no Reeb chords for some separating Reeb flow defined by $\tau$. Then there exists $C_\tau$ such that for any local system $L \in \operatorname{Loc}(\Lambda) \simeq \operatorname{Loc}(M)$, the sheaf quantizations $(\Psi_\Lambda L, \Psi_{0_M}L)$ are $(C_\tau a, C_\tau b)$-interleaved, and thus
    \[
        d_\tau(\Psi_\Lambda L, \Psi_{0_M}L) \leq C_\tau(a + b).
    \]
    Moreover, we can choose the continuation morphisms which under $i_\infty^*$ restrict to the identity morphism $L \to L$.
\end{corollary}
\begin{proof}
    First, using \Cref{thm:guillermou} (1) \& (2), we know $\Psi_\Lambda L|_{M \times [b,+\infty)} = L \boxtimes 1_{[b,+\infty)}$, $\Psi_{\Lambda} L|_{M \times (-\infty,-a]} = 0$. Let $T_t\colon S^*_{\zeta>0}(M \times \bR) \to S^*_{\zeta>0}(M \times \bR)$ be the Reeb flow given by the vertical push-off. We will first estimate the distance in terms of the Reeb flow $T_t$ and then deduce the estimation in terms of our separating Reeb flow $\varphi^\tau_t$. Since $\Lambda$ and $0_M \subseteq S^*_{\zeta>0}(M \times \mathbb{R})$, we can apply microlocal Morse lemma and show that 
    \[
        \Hom(T_{-a} \circ \Psi_{0_M}L, \Psi_\Lambda L) = \Hom(\Psi_\Lambda L, T_{b} \circ \Psi_{0_M} L) = \Hom(L, L).
    \]
    Since the full faithfulness in \Cref{thm:guillermou} implies that $\Hom(\Psi_\Lambda L, \Psi_\Lambda L) = \Hom(L, L)$, we can conclude that there exist morphisms
    \[
        T_{-a} \circ \Psi_{0_M}L \to \Psi_\Lambda L \to T_{b} \circ \Psi_{0_M}L, \quad T_{-b} \circ \Psi_\Lambda L \to \Psi_{0_M} L \to T_{a} \circ\Psi_\Lambda L
    \]
    whose compositions are the natural continuation maps, and the restrictions under the functor $i_\infty^*$ are the identity morphisms $L \to L$.
    
    Let $C_\tau > 0$ be the lower bound of the bounded contact Hamiltonian $\tau$ as in \Cref{def: bound-reeb}. Then we have natural continuation morphisms $K^{\tau}_{-C_\tau s} \to T_{-s}$ and $T_s \to K^{\tau}_{C_\tau s}$ for any $s > 0$. Hence there exist morphisms
    \[
        K^{\tau}_{-C_\tau a} \circ \Psi_{0_M}L \to \Psi_\Lambda L \to K^{\tau}_{C_\tau b} \circ \Psi_{0_M}L, \quad K^{\tau}_{-C_\tau b} \circ \Psi_\Lambda L \to \Psi_{0_M} L \to K^{\tau}_{C_\tau a} \circ\Psi_\Lambda L
    \]
    whose compositions are the continuation morphisms, and the restrictions under the functor $i_\infty^*$ are the identity morphisms $L \to L$. This completes the proof.
\end{proof}

\subsection{Sheaf quantization of \texorpdfstring{$C^0$}{C0}-small contactomorphisms}
Using the sheaf quantization theorem on nearby Legendrians \Cref{thm:guillermou-pre,thm:guillermou}, we prove the sheaf quantization theorem for any $C^0$-small contactomorphisms and deduce the interleaving distance estimation in \Cref{thm:quantization-C0} (\Cref{thm:main-C0-distance}).

Let $\varphi \colon S^*M \to S^*M$ be a contactomorphism such that $\varphi^*\alpha = e^h\alpha$. Recall it follows from \Cref{not:graph} that the graph of the contactomorphism to be the Legendrian submanifold
\[
\Gamma^\varphi = \{(x, \xi, y, \eta) \mid (y, \eta/|\eta|) = \varphi(x, \xi/|\xi|), |\xi|/|\eta| = e^{h(x,\xi/|\xi|)}\} \subseteq S^*(M \times M).
\]
The natural projection $\pi \colon S^*(M \times M) \to M \times M$ sends $(x, \xi, y, \eta)$ to $(x, y)$, so it factors through the projection $S^*M \times S^*M \to M \times M$. In particular, when the $C^0$-distance between two contactomorphisms are small, the $C^0$-distance of the projections of their graphs are also small.

Following \Cref{def: bound-reeb}, we will fix a bounded Reeb flow defined by the positive Hamiltonian $\tau \colon S^*M \to \bR$. Given \Cref{lem:disjoint-graph}, we will fix a bounded Reeb flow defined by the contact Hamiltonian $\tau$ and the cut-off function $\rho \colon S^*(M \times M) \to \bR$ so that the Hamiltonian $(1-\rho)\tau_2 \colon S^*(M \times M) \to \bR$ defines a complete contact flow.

\begin{lemma}\label{lemma:no-chord} 
    Let $M$ be a complete Riemannian manifold, $\alpha$ be a contact form on $S^*M$ and $\tau \colon S^*M \to \bR$ be a positive Hamiltonian that defines the bounded Reeb flow for $\alpha$ such that there are no closed Reeb orbits of length less than $\epsilon$. 
    Let $\varphi \colon S^*M \to S^*M$ be any injective map. Then the graph $\Gamma^\varphi \subseteq S^*(M \times M)$ has no chords of length less than $\epsilon$ with respect to the contact flow of the non-negative Hamiltonian $\tau_2$.
\end{lemma}
\begin{proof}
    We know that $\Gamma^\varphi$ is disjoint from $0_M \times S^*M$. Consider the chords on $\Gamma^\varphi$ with respect to the contact Hamiltonian $\tau_2 = (1-\rho)\tau_2$ on the complement of the tubular neighborhood of $S^*M \times 0_M$. They are determined by the equations
    \[
        (y_i, \eta_i/|\eta_i|) = \varphi(x_i, \xi_i/|\xi_i|), \quad (x_2, \xi_2/|\xi_2|) = (x_1, \xi_1/|\xi_1|), \quad (y_2, \eta_2/|\eta_2|) = \varphi^\tau_s(y_1, \eta_1/|\eta_1|), . 
    \]
    Then we get the equality $(y_1, \eta_1/|\eta_1|) = \varphi^\tau_{s}(y_1, \eta_1/|\eta_1|)$. We know that $\varphi^\tau_s$ has no Reeb orbits of length less than $\epsilon$ by assumption, so $\Gamma^\varphi$ has no chords of length less than $\epsilon$ with respect to $\varphi^{\tau_2}_s$, defined as a contact flow on $S^*(M \times M) \setminus 0_M \times S^*M$. 
\end{proof}

Given the above lemma that the graph of a contactomorphism have no short Reeb chords, we can prove the following theorem, which is the main technical result in this paper.

\begin{theorem}[\Cref{thm:main-C0-distance}]\label{thm:quantization-C0-small}
    Let $M$ be a complete Riemannian manifold, $\alpha$ be a contact form on $S^*M$ and $\tau \colon S^*M \to \bR$ be a positive Hamiltonian that defines a bounded Reeb flow for $\alpha$. Then for any $\epsilon > 0$ sufficiently small and any $\varphi \in \mathrm{Cont}(S^*M, \xi_{\mathrm{std}})$ with $d_{C^0}(\id, \varphi) \leq \epsilon$, there exists a unique sheaf quantization $K^\varphi$ of $\varphi$ such that for some tubular neighborhood $U_{2\epsilon}$ of $\Delta$ with radius $2\epsilon$,
    \begin{enumerate}
        \item $\supp(K^\varphi) \subseteq U_{2\epsilon}$, $\cSS^\infty(K^\varphi) \subseteq \Gamma^\varphi$;
        \item $m_{\Gamma_\varphi}(K^\varphi) = 1_{\Gamma^\varphi} \in \operatorname{Loc}(\Gamma^\varphi)$.
    \end{enumerate}
    Moreover, there is some constant $C_\tau$ such that $d_\tau(1_\Delta, K^\varphi) \leq 2 C_\tau\, d_{C^0}(\id, \varphi)$.
\end{theorem}

\begin{figure}
    \centering
    \includegraphics[width=0.8\linewidth]{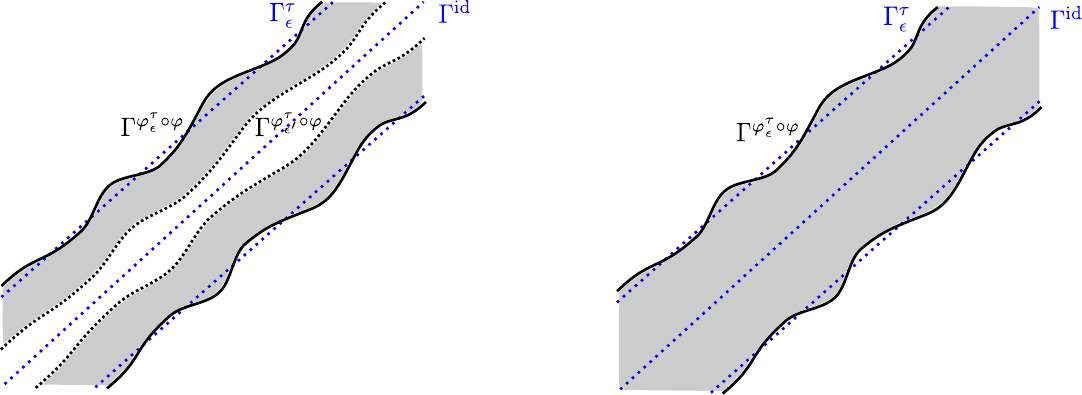}
    \caption{Construction of the sheaf quantization for a small contactomorphism $\varphi$ in \Cref{thm:quantization-C0-small} when $M = \bR$. Here, the blue curves are the projections of the graphs $\Gamma^\id$ and $\Gamma^\tau_\epsilon$ while the black curves is the projection of the graph $\Gamma^{\varphi^\tau_\epsilon \circ \varphi}$.}
\end{figure}

\begin{proof}
    Let $T^*_\Delta(M \times M)$ be the conormal bundle of the diagonal and $T^*_{\Delta,<r}(M \times M)$ be the subset of the conormal consisting of conormal vectors of norm less than $r$. Since the Reeb flow is bounded, without loss of generality, we can assume that for sufficiently small $\epsilon > 0$, the Reeb push-off of the conormal of the diagonal 
    \[
    \Phi_{\tau_2} \colon T^*_{\Delta,<2\epsilon}(M \times M) \hookrightarrow M \times M, \quad (x, -\xi, x, \xi) \mapsto (x, \pi \circ \varphi^{\tau}_{1}(x, \xi))
    \]
    defines an embedding. Write $U_{2\epsilon} = \Phi_{\tau_2}(T^*_{\Delta,<2\epsilon}(M \times M))$, $U_{2\epsilon}^\circ = \Phi_{\tau_2}(T^*_{\Delta,<2\epsilon}(M \times M) \setminus \Delta)$ and $S_\epsilon = \Phi_{\tau_2}(T^*_{\Delta,\epsilon}(M \times M))$. We have a diffeomorphism $U_{2\epsilon}^\circ \cong S_\epsilon \times (-\epsilon, \epsilon)$. 
 
    Let $C_1 > 0$ be the lower bound on the norm of the vector field $\partial_z$ along the coordinate $z\in (-\epsilon, \epsilon)$. Consider $\delta > 0$ such that $\delta/\epsilon = C_1/2$. Let $\varphi$ be a contactomorphism such that $d_{C^0}(\id, \varphi) \leq \delta$. Consider the contactomorphism $\varphi_{\epsilon}^\tau \circ \varphi$. We have $d_{C^0}(\varphi_{\epsilon}^\tau, \varphi_{\epsilon}^\tau \circ \varphi) \leq C_1 \epsilon$. Since the projections $\Gamma^{\varphi_{\epsilon}^\tau} \to M \times M$ and $\Gamma^{\varphi_{\epsilon}^\tau \circ \varphi} \to M \times M$ factor through $S^*M \times S^*M$, the distance of the projections of the graphs in $M \times M$ is bounded by their distance in $S^*M \times S^*M$ which is at most $C_1 \epsilon$. Therefore, we can assume that
    \[
    \Gamma^{\varphi_{\epsilon}^\tau \circ \varphi} \subseteq S^*(S_\epsilon \times (-\epsilon, \epsilon)) \cong S^*U_{2\epsilon}^\circ.
    \]
    Since $d_{C^0}(\varphi_{\epsilon}^\tau, \varphi_{\epsilon}^\tau \circ \varphi) < \epsilon$, and for any $(x, \xi, y, \eta) \in \Gamma_{\epsilon}^\tau$, $\left<\eta, \partial_z\right>_y > 0$, we may assume that for any $(x, \xi, y, \eta) \in \Gamma^{\varphi_{\epsilon}^\tau \circ \varphi}$, $\left< (\xi, \eta), \partial_z\right>_{(x,y)} = \left<\eta, \partial_z\right>_y > 0$. Thus,
    \[
    \Gamma^{\varphi_{\epsilon}^\tau \circ \varphi} \subseteq S^*_{\zeta>0}(S_\epsilon \times (-\epsilon, \epsilon)_z) \cong S^*_{\zeta>0}U_{2\epsilon}^\circ.
    \]
    By \Cref{lemma:no-chord}, we know that $\Gamma^{\varphi_{\epsilon}^\tau \circ \varphi}$ has no chords of length less than $2\epsilon$ with respect to the contact flow $\varphi^{\tau_2}_t$. For $0< \epsilon' < \epsilon$, since the projection of $\Gamma^{\varphi_{\epsilon}^\tau \circ \varphi}$ and $\Gamma^{\varphi_{\epsilon'}^\tau \circ \varphi}$ are contained in small neighborhoods of $S_\epsilon$ and $S_{\epsilon'}$, we can assume that the graphs $\Gamma^{\varphi_{\epsilon}^\tau \circ \varphi}$ and $\Gamma^{\varphi_{\epsilon'}^\tau \circ \varphi}$ are separated by a hypersurface $S_{\epsilon''}$. Then \Cref{thm:guillermou} implies that there exists a unique sheaf 
    $K^{\varphi_{\epsilon}^\tau \circ \varphi}_{U_{2\epsilon}^\circ} \in \operatorname{Sh}(U_{2\epsilon}^\circ)$ such that for some sufficiently small $\epsilon' < \epsilon$,
    \begin{enumerate}
        \item $\supp(K^{\varphi_{\epsilon}^\tau \circ \varphi}_{U_{2\epsilon}^\circ}) \subseteq U_{2\epsilon-\epsilon'}^\circ$, $\cSS^\infty(K^{\varphi_{\epsilon}^\tau \circ \varphi}_{U_{2\epsilon}^\circ}) \subseteq \Gamma^{\varphi_{\epsilon}^\tau \circ \varphi}$;
        \item $m_{\Gamma_{\varphi_{\epsilon}^\tau \circ \varphi}}(K^{\varphi_{\epsilon}^\tau \circ \varphi}_{U_{2\epsilon}^\circ}) = 1_{\Gamma_{\varphi_{\epsilon}^\tau \circ \varphi}}$ and $i_{S_{\epsilon'}}^*K^{\varphi_{\epsilon}^\tau \circ \varphi}_{U_{2\epsilon}^\circ} = 1_{S_{\epsilon'}}$.
    \end{enumerate} 
    Then since the restriction functor $\Loc(U_{\epsilon'}) \to \Loc(S_{\epsilon'})$ is conservative, there exists a unique extension $K^{\varphi_{\epsilon}^\tau \circ \varphi} \in \operatorname{Sh}(M \times M)$ such that 
    \[
    i_{U_{2\epsilon}^\circ}^*K^{\varphi_{\epsilon}^\tau \circ \varphi} = K^{\varphi^{\epsilon}_\tau \circ \varphi}_{U_{2\epsilon}^\circ}, \quad i_{U_{\epsilon'}}^*K^{\varphi_{\epsilon}^\tau \circ \varphi} = 1_{U_{\epsilon'}}.
    \]
    We define $K^\varphi \coloneqq K^{\tau}_{-\epsilon} \circ K^{\varphi_{\epsilon}^\tau \circ \varphi}$. We can conclude by \Cref{lem: microsupport-convolution-composition} that this is the unique sheaf in $\Sh(M \times M)$ such that 
    \[
    \supp(K^\varphi) \subseteq U_{2\epsilon}, \quad \cSS^\infty(K^\varphi) \subseteq \Gamma^\varphi, \quad m_{\Gamma^\varphi}(K^\varphi) = 1_{\Gamma^\varphi} \in \operatorname{Loc}(\Gamma^\varphi).
    \]
    
    Finally, we estimate the interleaving distance. We know that $\Phi_{\tau_2} \colon T^*_{\Delta,<2\epsilon}(M \times M) \hookrightarrow M \times M$ defines an embedding, so the Reeb vector field defined by the Riemannian metric is bounded by the Reeb vector field $\partial_z$ defined by the contact form on the 1-jet bundle $S^*_{\zeta > 0}U_{2\epsilon}^\circ$. Then the Reeb flow defined by $\tau$ is bounded with respect to the standard contact form on the 1-jet bundle $S^*_{\zeta > 0}U_{2\epsilon}^\circ$. Let $C_2 > 0$ be the lower bound of the bounded contact Hamiltonian $\tau$ with respect to the standard contact form on $S^*_{\zeta > 0}U_{2\epsilon}^\circ$. First, by \Cref{cor:guillermou-quant}, we know that there exists $C_2 > 0$ 
    such that the sheaf quantizations $(1_{U_\epsilon^\circ}, K^{\varphi_{\epsilon}^\tau \circ \varphi}_{U_{2\epsilon}^\circ})$ are $(C_2a, C_2b)$-interleaved for some $a, b < \epsilon$. Moreover, by \Cref{cor:guillermou-quant}, the restrictions of the continuation morphisms to $S_{\epsilon-C_2b}$ are the identity. 
    Then, by the construction of the sheaf quantization $K^{\varphi_{\epsilon}^\tau \circ \varphi}$, there exist morphisms
    \[
        K^\tau_{\epsilon - C_2a} \to K^{\varphi_{\epsilon}^\tau \circ \varphi} \to K^{\tau}_{\epsilon+C_2b}, \quad K^{\tau}_{-\epsilon-C_2b} \circ K^{\varphi_{\epsilon}^\tau \circ \varphi} \to K^{\tau}_{-\epsilon} \to K^\tau_{-\epsilon + C_2a} \circ K^{\varphi_{\epsilon}^\tau \circ \varphi}
    \]
    whose compositions are the continuation morphisms. 
    Define $K^\varphi \coloneqq K^{\tau}_{-\epsilon} \circ K^{\varphi^\tau_\epsilon \circ \varphi}$, then we have morphisms
    \[
        K^{\tau}_{-C_2a} \to K^\varphi \to K^{\tau}_{C_2b}, \quad K^{\tau}_{-C_2b} \circ K^\varphi \to 1_\Delta \to K^{\tau}_{C_2a} \circ K^\varphi
    \]
    whose the compositions are the continuation morphisms. 
    Therefore, since $a, b < \epsilon$ and $\delta/\epsilon = C_1/2$, we can set $C_\tau = 2C_2/C_1$ and thus
    $d_{\tau}(1_\Delta, K^\varphi) \leq 2C_\tau\,d_{C^0}(\id, \varphi)$.
    This concludes the proof. 
\end{proof}

Now, we will show that such sheaf quantizations for $C^0$-small contactomorphisms are functorial with respect to compositions:

\begin{proposition}\label{prop:quantization-functorial}
    Let $M$ be a complete Riemannian manifold and $\alpha$ be a contact form on $S^*M$. Then for the $\epsilon > 0$ specified in \Cref{thm:quantization-C0-small} and any $\varphi, \varphi' \in \mathrm{Cont}(S^*M, \xi_{\mathrm{std}})$, each of which has bounded conformal factor, such that $d_{C^0}(\id, \varphi) + d_{C^0}(\id, \varphi') \leq \epsilon$, the sheaf quantizations in \Cref{thm:quantization-C0-small} satisfy the relation
    \[
        K^\varphi \circ K^{\varphi'} = K^{\varphi \circ \varphi'}.
    \]
    In particular, $K^\varphi$ is an invertible sheaf kernel with inverse $K^{\varphi^{-1}}$ if $d_{C^0}(\id,\varphi) \leq {\epsilon}/{2}$.
\end{proposition}
\begin{proof}
    First, by the singular support estimation \Cref{lem: microsupport-convolution-composition}, we know that
    \[
    \supp(K^\varphi \circ K^{\varphi'}) \subseteq U_{2\epsilon}, \quad \cSS^\infty(K^\varphi \circ K^{\varphi'}) \subseteq \Gamma^{\varphi \circ \varphi'}.
    \]
    By the invariance of microstalks under contact transformation \cite[Theorem 7.5.11]{KS90}, we can compute that $m_{\Gamma^{\varphi \circ \varphi'}}(K^\varphi \circ K^{\varphi'})$ is a rank $1$ local system. Then, by \Cref{cor:guillermou-quant}, for sufficiently small $\epsilon' > 0$, the restriction 
    $K^\tau_\epsilon \circ K^\varphi \circ K^{\varphi'}|_{S_{\epsilon'}}$
    is also a rank $1$ local system. However, since the assumption $d_{C^0}(\id, \varphi) + d_{C^0}(\id, \varphi') \leq \epsilon$, it follows from \Cref{thm:quantization-C0-small} that
    \[
    d_\tau(1_\Delta, K^\varphi \circ K^{\varphi'}) \leq d_\tau(1_\Delta, K^{\varphi'}) + d_\tau(K^{\varphi'}, K^\varphi \circ K^{\varphi'}) = d_\tau(1_\Delta, K^{\varphi'}) + d_\tau(1_\Delta, K^\varphi) \leq 2C_\tau \epsilon.
    \]
    In particular, for sufficiently small $\epsilon' > 0$, $\Hom(K^\tau_{\epsilon'}, K^\tau_\epsilon \circ K^\varphi \circ K^{\varphi'}) \neq 0$. This forces the restriction $K^\tau_\epsilon \circ K^\varphi \circ K^{\varphi'}|_{U_{\epsilon'}}$ to be the constant local system $1_{U_{\epsilon'}}$. Thus, by \Cref{cor:guillermou-quant}, it follows that
    \[
    m_{\Gamma^{\varphi \circ \varphi'}}(K^\varphi \circ K^{\varphi'}) = 1_{\Gamma^{\varphi \circ \varphi'}}.
    \]
    Then, the uniqueness statement in \Cref{thm:quantization-C0-small} implies that $K^{\varphi \circ \varphi'} = K^\varphi \circ K^{\varphi'}$.
\end{proof}

For a given contactomorphism $\varphi \in \mathrm{\Cont}_0(S^*M,\xi_{\mathrm{std}})$ in the identity component, we can take a contact isotopy $\varphi_t \in \mathrm{\Cont}_0(S^*M,\xi_{\mathrm{std}})$ such that $\varphi_0 = \id$ and $\varphi_1 = \varphi$ and then apply \cref{thm: GKS} to get a sheaf quantization $K^\varphi$ of $\varphi$.
We remark that given a contactomorphism $\varphi \in \mathrm{Cont}_0(S^*M, \xi_{\mathrm{std}})$, it is not true that there is a unique contact isotopy $\varphi_t \in \mathrm{Cont}_0(S^*M, \xi_{\mathrm{std}})$ connecting $\id$ and $\varphi$. Therefore, the sheaf quantization of contact isotopies by Guillermou--Kashiwara--Schapira does not give a unique sheaf quantization of the contactomorphism $\varphi \in \mathrm{Cont}_0(S^*M, \xi_{\mathrm{std}})$ (actually, the sheaf quantization of Guillermou--Kashiwara--Schapira is defined on the universal cover of $\mathrm{Cont}_0(S^*M, \xi_{\mathrm{std}})$). However, we are able to prove the weaker result:

\begin{proposition}\label{prop:compare-gks}
    Let $\Phi$ be a contact isotopy such that $\varphi_0 = \id$, $\varphi_1 = \varphi$ and $d_{C^0}(\id, \varphi_t) \leq \epsilon$ for any $t \in I$. For the sheaf quantization $K^{\Phi} \in \Sh(M \times M \times I)$ induced by Guillermou--Kashiwara--Schapira \Cref{thm: GKS}, $K^\varphi = K^{\Phi}_1$ is isomorphic to the sheaf quantization of $\varphi$ in \Cref{thm:quantization-C0-small}.
\end{proposition}
\begin{proof}
    Consider the sheaf quantization $K^{\Phi} \in \Sh(M \times M \times I)$ in \Cref{thm: GKS}. We know by \Cref{lem: microsupport-convolution-composition} that $\cSS^\infty(K^{\Phi}) \subseteq \Gamma^\Phi$. First, this implies that $\cSS^\infty(K^{\Phi}_1) \subseteq \Gamma^\varphi.$
    Next, since $d_{C^0}(\id, \Phi) \leq \epsilon$, we know that $\cSS^\infty(K^{\Phi}) \subseteq S^*(U_\epsilon \times I)$. 
    This means that $K^{\Phi}|_{(M \times M \setminus U_\epsilon) \times I}$ is a constant sheaf. Since $K^{\Phi}|_{(M \times M \setminus U_\epsilon) \times 0} = 0$, we have $K^{\Phi}|_{(M \times M \setminus U_\epsilon) \times 1} = 0$. 
    Hence we get 
    \[
        \supp(K^{\Phi}_1) \subseteq U_\epsilon.
    \]
    Finally, consider the microlocalization $m_{\Gamma^\Phi}(K^{\Phi}) \in \Loc(\Gamma^\Phi)$. Since the restriction at $0$ is the constant sheaf $1_{S^*\Delta} \in \Loc(S^*\Delta)$, the restriction at $1$ is also the constant sheaf:
    \[
        m_{\Gamma^\varphi}(K^{\Phi}_1) = 1_{\Gamma^\varphi}.
    \]
    Thus, we can simply apply the uniqueness statement \Cref{thm:quantization-C0-small} to deduce the result.  
\end{proof}

\begin{example}
    Let $M = S^1$ and consider the contactomorphism $\varphi_t \colon S^*M \to S^*M, \, (x, \xi) \mapsto (x+t, \xi)$. Then $\varphi_{n}$ is the identity contactomorphism for any $n \in \bN$. However, the sheaf quantization
    by \Cref{thm: GKS} gives different sheaf quantizations
    $K^{\varphi_{n}} = \pi_!1_{\overline{U}_{n}(\Delta)}$
    where $\pi \colon \bR \times \bR \to S^1 \times S^1$ is the projection and $\overline{U}_{n}(\Delta) = \{(x, y) \mid |x - y| \leq n\} \subseteq \bR^2.$ Indeed, we need $\epsilon < {1}/{2}$ for this case.
\end{example}

\subsection{Sheaf quantization of contact homeomorphisms}
We now prove the results on the invariance of categories of sheaves under contact homeomorphisms, which are limits of contactomorphisms under the $C^0$-topology in \Cref{thm:quantization-C0,prop:quantization-micro} (\Cref{thm:main-sheaf-invariance}).

For a contact homeomorphism, the graph is not usually well defined due to the issue that the conformal factors do not converge. However, we can define the pseudo-graph to be a closed subset in $S^*(M \times M)$ as follows:

\begin{definition}
    Let $\varphi \colon S^*M \to S^*M$ be a homeomorphism. Then we define the pseudo-graph of $\varphi$ as the following subset
    \[
        \Gamma^{\varphi}_{\mathrm{pseudo}} = \{(x, \xi, y, \eta) \mid (y, \eta/|\eta|) = \varphi(x, \xi/|\xi|)\} \subseteq S^*(M \times M).
    \]
    In other words, it is the image of $\Gamma^\varphi \times \bR \subseteq S^*M \times S^*M \times \bR$ under the open embedding 
    \begin{align*}
    S^*M \times S^*M \times \bR &\hookrightarrow S^*(M \times M) \\
    ((x,\xi),(y,\eta),t) &\mapsto (x,y,e^t\xi,e^{-t}\eta),
    \end{align*}
    whose image is $S^*(M \times M) \setminus (0_M \times S^*M \cup S^*M \times 0_M)$. 
\end{definition}  
\begin{warning}
    Although the name suggests, the pseudo-graph does not become a graph even when the conformal factors converges. Without a well-defined conformal factor, a contact homeomorphism does not lift to a homogeneous map $\dot T^*M \to \dot T^*M$. Therefore, it does not make sense to define a half-dimensional graph in $\dot T^*(M \times M)$. Indeed, the pseudo-graph is one dimension higher, and fits into a natural fibration $\Gamma^\varphi_{\mathrm{pseudo}} \to \Gamma^\varphi$ whose fiber is $\bR$.
\end{warning}

We now state the sheaf quantization theorem for contact homeomorphisms. We will take the standard contact form $\alpha_{\mathrm{std}}$ on $S^*M$ induced by the complete Riemannian metric and all the contact Hamiltonian functions are the Hamiltonian with respect to this standard contact form $\alpha_{\mathrm{std}}$.

\begin{theorem}
\label{thm:quantization-C0} 
    Let $\varphi_n \in \mathrm{Cont}(S^*M, \xi_{\mathrm{std}})$ be contactomorphisms
    such that $d_{C^0}(\id,\varphi_n) < \epsilon/2$, where $\epsilon>0$ is sufficiently small so that \cref{thm:quantization-C0-small} holds.
    Then for a bounded Reeb flow defined by $\tau$ under the standard contact form, the sheaf quantizations $K^{\varphi_n}$ converge under the interleaving distance 
    to a sheaf quantization $K^{\varphi_\infty}$ of $\varphi_\infty$ such that 
    \[\cSS^\infty(K^{\varphi_\infty}) \subseteq \overline{\Gamma^{\varphi_\infty}_{\mathrm{pseudo}}}\]
    and for any $F \in \Sh(M)$, $\cSS^\infty(K^{\varphi_\infty} \circ F) \subseteq \varphi_\infty(\cSS^\infty(F))$. If the conformal factors $h_n$ are uniformly bounded, then 
    \[\cSS^\infty(K^{\varphi_\infty}) \subseteq \Gamma^{\varphi_\infty}_{\mathrm{pseudo}}.\]
    Moreover, there exists $C_\tau > 0$ such that $d_\tau(1_\Delta, K^{\varphi_\infty}) \leq 2C_\tau\, d_{C^0}(\id, \varphi_\infty)$.
\end{theorem}
\begin{proof}
    Without loss of generality, we assume that $d_{C^0}(\varphi_n, \varphi_{n+1}) \leq \epsilon_n$ and $\sum_{n=1}^\infty \epsilon_n < \infty$. Since $\varphi_n$ is a diffeomorphism, we know $d_{C^0}(\varphi_n \circ \varphi_{n+1}^{-1}, \id) = d_{C^0}(\varphi_n, \varphi_{n+1}) \leq \epsilon_n$. 
    Using \Cref{thm:quantization-C0-small} and \Cref{prop:quantization-functorial}, 
    we know that there exist morphisms 
    \[
        K^{\tau}_{-C_\tau\epsilon_n} \circ K^{\varphi_{n+1}} \to K^{\varphi_n} \to K^{\tau}_{C_\tau\epsilon_n} \circ K^{\varphi_{n+1}}, \quad K^{\tau}_{-C_\tau\epsilon_n} \circ K^{\varphi_{n}} \to K^{\varphi_{n+1}} \to K^{\tau}_{C_\tau\epsilon_n} \circ K^{\varphi_{n}}
    \]
    whose compositions are the natural continuation morphisms. Then, we know that $(K^{\varphi_n})_{n \ge N}$ form a Cauchy sequence, and by \Cref{thm:complete} it converges under the interleaving distance to a sheaf kernel $K^{\varphi_\infty}$. By \Cref{thm:ss-limit}, we have
    \begin{align*}
    \cSS^\infty(K^{\varphi_\infty}) \subseteq \bigcap\nolimits_{n\geq 0}\overline{\bigcup\nolimits_{k \geq n} \cSS^\infty(K^{\varphi_k})} = \bigcap\nolimits_{n\geq 0}\overline{\bigcup\nolimits_{k \geq n}\Gamma^{\varphi_k}} \subseteq \overline{\Gamma^{\varphi_\infty}_{\mathrm{pseudo}}}.
    \end{align*}
    The result $\cSS^\infty(K^{\varphi_\infty} \circ F) \subseteq \varphi_\infty(\cSS^\infty(F))$ follows from the estimation \Cref{thm:ss-limit} and the fact that $\varphi_n \to \varphi_\infty$.
    If the conformal factors are pointwise uniformly bounded, say ${-h} \leq {h_n(x, \xi)} \leq h$ for all $n \in \bN$ and $(x, \xi) \in S^*M$, we have
    \begin{align*}
    \cSS^\infty(K^{\varphi_\infty}) &\subseteq \{(x, \xi, y, \eta) \mid (y, \eta/|\eta|) = \varphi_\infty(x, \xi/|\xi|), |\xi|/|\eta| \in [e^{-h}, e^{h}]\} \subseteq \Gamma^{\varphi_\infty}_{\mathrm{pseudo}}.
    \end{align*} 
    This completes the proof of the singular support estimation. Finally, for the distance estimation, we apply \Cref{thm:quantization-C0-small} and get $d_\tau(1_\Delta, K^{\varphi_n}) \leq 2C_\tau\,d_{C^0}(\id, \varphi_n)$. Thus, since $K^{\varphi_n} \to K^{\varphi_\infty}$, we can conclude that $d_\tau(1_\Delta, K^{\varphi_\infty}) \leq 2C_\tau\,d_{C^0}(\id, \varphi_\infty)$.
\end{proof}

We explain in the following example why considering the pseudo-graph is necessary for sheaf quantizations of contact homeomorphisms.

\begin{example}
    Let $M = \bR$. Consider a sequence of smooth increasing functions $f_n \colon \bR \to \bR$ that converges to $f_\infty(x) = x^2 \mathrm{sign}(x)$. Then the sequence of contactomorphisms $\varphi_n(x, \xi) = (f_n(x), \xi)$ converges to the homeomorphism 
    \[\varphi_\infty(x, \xi) = (x^2 \mathrm{sign}(x), \xi).\]
    One can show that conformal factors of $\varphi_n$ are unbounded near $x = 0$. Moreover, since $\varphi_n$ and $\varphi_\infty$ are induced by the homeomorphisms $f_n$ and $f_\infty(x) = x^2 \mathrm{sign}(x)$, the sheaf quantization of $\varphi_\infty$ is
    \[K^{\varphi_n} = 1_{\Gamma^{f_n}}, \quad K^{\varphi_\infty} = 1_{\Gamma^{f_\infty}}.\]
    In particular, the limit of the graphs of $\varphi_n$ does not define a section of the fibration $\Gamma^{\varphi_\infty}_{\mathrm{pseudo}} \to \Gamma^{\varphi_\infty}$, and indeed 
    \[\cSS^\infty(K^{\varphi_\infty}) \cap S^*_{(0,0)}\bR^2 = \overline{\Gamma^{\varphi_\infty}_{\mathrm{pseudo}}} \cap S^*_{(0,0)}\bR^2 = \{(0, 0, \xi, \eta) \mid \mathrm{sign}(\xi) = \mathrm{sign}(\eta)\}.\]
    More generally, for $M = N \times \bR$, one can consider the example of Usher \cite[Section 5.2]{Usher21Conformal} of the contact homeomorphism $\varphi_\infty \colon S^*(N \times \bR) \to S^*(N \times \bR)$ where it is defined by the time-1 flow of the smooth Hamiltonian on $S^*(N \times \bR) \setminus S^*_{N \times 0}(N \times \bR)$, and extends only continuously to $S^*_{N \times 0}(N \times \bR)$, and see that limit of the graphs of $\varphi_n$ does not define a section of the fibration $\Gamma^{\varphi_\infty}_{\mathrm{pseudo}} \to \Gamma^{\varphi_\infty}$.
\end{example}
 
Showing that the sheaf quantization is invertible turns out to be very tricky. Since the singular support estimation above is more complicated when we do not have control on the conformal factors, one cannot deduce that $\cSS(K^{\varphi_\infty^{-1}} \circ K^{\varphi_\infty}) = \cSS(K^{\varphi_\infty^{-1}} \circ K^{\varphi_\infty}) = T^*_\Delta(M \times M)$ just by \Cref{lem: microsupport-convolution-composition}. Nevertheless, we can prove the following result:

\begin{theorem}\label{thm:quantization-invertible}
    Let $M$ be a complete Riemannian manifold and $\varphi_n \in \mathrm{Cont}(S^*M, \xi_{\mathrm{std}})$ be contactomorphisms
    such that $d_{C^0}(\id,\varphi_n) < \epsilon/2$, where $\epsilon>0$ is sufficiently small so that \cref{thm:quantization-C0-small} holds. Suppose that each $\varphi_n$ has bounded conformal factor, $\varphi_n \to \varphi_\infty$ in the $C^0$-topology and $\varphi_\infty$ is a homeomorphism. Then the sheaf quantization $K^{\varphi_\infty}$ is invertible with inverse $K^{\varphi_\infty^{-1}}$.
\end{theorem}
\begin{proof}
    We only show that $1_\Delta = K^{\varphi_\infty^{-1}} \circ K^{\varphi_\infty}$. We have $K^{\varphi_n^{-1}} \circ K^{\varphi_\infty}_1 \to K^{\varphi^{-1}_\infty} \circ K^{\varphi_\infty}$ since the interleaving distance on the product is right invariant. 
    On the other hand, we can show that $K^{\varphi^{-1}_n} \circ K^{\varphi_m} \to K^{\varphi^{-1}_n} \circ K^{\varphi_\infty}$ by the same proof as in \Cref{thm:Hofer-quantize-invert}. When $(K^{\varphi_m}, K^{\varphi_{m'}})$ are $(a, b)$-interleaved, we know that there are morphisms such that the compositions
    \begin{gather*}
    K^{\varphi^{-1}_n} \circ K^\tau_{-a} \circ K^{\varphi_m} \to K^{\varphi^{-1}_n}_1 \circ K^{\varphi_{m'}} \to K^{\varphi^{-1}_n} \circ K^\tau_{b} \circ K^{\varphi_m},  \\
    K^{\varphi^{-1}_n} \circ K^\tau_{-b} \circ K^{\varphi_{m'}} \to K^{\varphi^{-1}_n} \circ K^{\varphi_{m}} \to K^{\varphi^{-1}_n} \circ K^\tau_{a} \circ K^{\varphi_{m'}}
    \end{gather*}
    are the continuation morphisms. Then by applying \Cref{lem:conformal-reeb-commute} and by \Cref{prop:banach-mazur}, when the conformal factor of $\varphi_n$ is bounded by $h_n$, we get
    \begin{gather*}
    K^\tau_{-e^{h_n}a}\circ K^{\varphi^{-1}_n} \circ K^{\varphi_m} \to K^{\varphi^{-1}_n} \circ K^{\varphi_{m'}} \to K^\tau_{e^{h_n}b} \circ K^{\varphi^{-1}_n} \circ K^{\varphi_m},  \\
    K^\tau_{-e^{h_n}b} \circ K^{\varphi^{-1}_n} \circ K^{\varphi_{m'}} \to K^{\varphi^{-1}_n} \circ K^{\varphi_{m}} \to K^\tau_{e^{h_n}a} \circ K^{\varphi^{-1}_n} \circ K^{\varphi_{m'}}
    \end{gather*}
    whose compositions are continuation morphisms. Thus, given $n \in \bN$, for any $H_m$ and $H_{m'}$, 
    we know that 
    \[d_\tau(K^{\varphi^{-1}_n} \circ K^{\varphi_m}, K^{\varphi^{-1}_n} \circ K^{\varphi_{m'}}) \leq e^{h_n}\, d_\tau(K^{\varphi_m}, K^{\varphi_{m'}}).\]
    Hence, for any given $n \in \bN$, $K^{\varphi^{-1}_n} \circ K^{\varphi_m} \to K^{\varphi^{-1}_n} \circ K^{\varphi_\infty}$, and since $d_\tau(1_\Delta, K^{\varphi^{-1}_n} \circ K^{\varphi_m}) \leq 2 C_\tau\,d_{C^0}(\id, \varphi_n^{-1} \circ \varphi_m)$ by \Cref{thm:quantization-C0-small}, we know that 
    \[d_\tau(1_\Delta, K^{\varphi^{-1}_n} \circ K^{\varphi_\infty}) \leq 2 C_\tau\,d_{C^0}(\id, \varphi_n^{-1} \circ \varphi_\infty).\]
    Since $\varphi^{-1}_n \circ \varphi_\infty \to \id$ as $n \to \infty$, we know $K^{\varphi^{-1}_n} \circ K^{\varphi_\infty} \to 1_\Delta$. Therefore, $d_\tau(1_\Delta, K^{\varphi^{-1}_\infty} \circ K^{\varphi_\infty}) = 0$. By \Cref{prop: ss-is-same-kernel,thm:non-deg_kernel} we know that $1_\Delta = K^{\varphi_\infty^{-1}} \circ K^{\varphi_\infty}$.
\end{proof}

Kashiwara--Schapira \cite{KS90} constructed quantization of local contactomorphisms on the level of microlocal sheaves. The microlocal quantization can also be deduced using the sheaf quantization of Guillermou--Kashiwara--Schapira \cite{GKS} as explained in \cite{Li23Cobordism}. 
    
\begin{theorem}\label{prop:quantization-micro}
    Let $\varphi_n \in \mathrm{Cont}(S^*M, \xi_{\mathrm{std}})$ be contactomorphisms
    such that $d_{C^0}(\id,\varphi_n) < \epsilon/2$, where $\epsilon>0$ is sufficiently small so that \cref{thm:quantization-C0-small} holds. Suppose that each $\varphi_n$ has bounded conformal factor, $\varphi_n \to \varphi_\infty$ in the $C^0$-topology and $\varphi_\infty$ is a homeomorphism. 
    Then there exists a equivalence between sheaves of categories
    \[
        K^{\varphi_\infty} \colon {\varphi_\infty}_*\msh \to \msh,
    \]
    whose inverse is $K^{\varphi_\infty^{-1}}$. In particular, for any open set $\Omega \subseteq S^*M$, let $\Lambda \subseteq \Omega$ be any closed coisotropic subset. Then there exists a equivalence between sheaves of categories over $\varphi_\infty(\Omega)$
    \[
        K^{\varphi_\infty} \colon {\varphi_\infty}_*\msh_\Lambda \to \msh_{\varphi_\infty(\Lambda)},
    \]
    whose inverse is $K^{\varphi_\infty^{-1}}$, and when $\Lambda$ and $\varphi_\infty(\Lambda)$ are smooth Legendrians, then a  microstalk at $p \in \Lambda$ is sent to a microstalk at $\varphi_\infty(p) \in \varphi_\infty(\Lambda)$ (thus their corresponding corepresentatives are matched).
\end{theorem}
\begin{proof}
    Consider the sheaf quantization $K^{\varphi_\infty}$ of $\varphi_\infty$. For an open subset $\Omega \subseteq S^*M$, by \Cref{thm:quantization-C0}, we can restrict the convolution functor to the subcategories, which induces a morphism of presheaves of categories
    \[
    K^{\varphi_\infty} \circ - \colon \Sh_{\Lambda \cup \Omega^c}(M) \to \Sh_{\varphi_\infty(\Lambda \cup \Omega^c)}(M), \quad \Sh_{\Omega^c}(M) \simeq \Sh_{\varphi_\infty(\Omega^c)}(M).
    \]
    Then, by sheafification, we get a morphism of the associated sheaves of categories. The result follows from \Cref{thm:quantization-invertible}.

    Finally, recall that an object $F \in \msh_\Lambda(\Lambda)$ is simple if and only if $\mu hom(F, F)_p = 1$. Hence simpleness is preserved under equivalences of the sheaves of categories. Then the microstalk (up to autoequivalences) is preserved because the microstalk of $G \in \msh_\Lambda(\Lambda)$ is (non-canonically) computed by
    \[
        m_\Lambda(G)_p = \mu hom(F, G)_p.
    \]
    This therefore finishes the proof.
\end{proof}

\begin{remark}
    Since $K^{\varphi_\infty}$ induces an equivalence of sheaves of categories $K^{\varphi_\infty} \colon {\varphi_\infty}_*\msh \xrightarrow{\sim} \msh$, by \cite[Corollary 10.1.5]{Guillermou23}, we know that for any open subset $\Omega \subseteq S^*M$ and $F, G \in \msh(\Omega)$,
    \[
        \varphi_{\infty*}\mu hom(F, G) \simeq \mu hom(K^{\varphi_\infty} \circ F, K^{\varphi_\infty} \circ G).
    \]
\end{remark}

In general, we can also define the sheaf quantization for given a sequence of contactomorphisms $\varphi_n \in \mathrm{Cont}_0(S^*M, \xi_{\mathrm{std}})$ that converges to a homeomorphism $\varphi_\infty$ in the $C^0$-topology, we can also define a sheaf quantization $K^{\varphi_\infty}$ of $\varphi_\infty$ as follows, though in a non-canonical way.

\begin{theorem}[\Cref{thm:main-sheaf-invariance}]\label{thm:sheaf-invariance-summary}
    Let $(S^*M, \xi_{\mathrm{std}})$ be the cosphere bundle with the standard contact structure and $\Lambda \subseteq S^*M$ be a Legendrian embedding. Consider contactomorphisms $\varphi_n \in \operatorname{Cont}_0(Y, \xi)$ each of which has bounded conformal factor $h_n$. Suppose $\varphi_n \to \varphi_\infty$ in the $C^0$-topology and $\varphi_\infty$ is a homeomorphism. Then there exists a (non-canonical) sheaf quantization $K^{\varphi_\infty} \in \Sh(M \times M)$ with inverse $K^{\varphi_\infty^{-1}} \in \Sh(M \times M)$, and the convolutions preserve microstalks.
\end{theorem}
\begin{proof}
    Let $\epsilon>0$ be sufficiently small such that \Cref{thm:quantization-C0-small} holds and suppose $d_{C^0}(\varphi_n, \varphi_\infty) < \epsilon/4$ for $n \geq N$. 
    We define $K^{\varphi_N}$ by \Cref{thm: GKS} and  $K^{\varphi_\infty \circ \varphi_N^{-1}}$ by \Cref{thm:quantization-C0}. 
    Now we can define $K^{\varphi_\infty} \coloneqq K^{\varphi_\infty \circ \varphi_N^{-1}} \circ K^{\varphi_N}$. Similarly, we define $K^{\varphi_N^{-1}}$ by \Cref{thm: GKS} and $K^{\varphi_N \circ \varphi_\infty^{-1}}$ by \Cref{thm:quantization-C0}. Let $K^{\varphi_\infty^{-1}} \coloneqq K^{\varphi_N^{-1}} \circ K^{\varphi_N \circ \varphi_\infty^{-1}}$. Then the result follows from \Cref{thm:quantization-invertible,prop:quantization-micro}.
\end{proof}

\begin{warning}
    The sequence $(K^{\varphi_n})_{n\in\bN}$ is not canonical, and thus neither is the limit $K^{\varphi_\infty}$. In general, it is an interesting question when such sequence (and its limit) is canonical.
\end{warning}

The sheaf quantization result allows us to recover and strengthen the theorem of Dimitroglou Rizell--Sullivan \cite{DRG24C0Legendrian}. For a sequence of compactly supported contactomorphisms $\varphi_n$ that converges to a homeomorphism $\varphi_\infty$, we can show that for any (not necessarily properly embedded) Legendrian $\Lambda \subseteq S^*M$, if $\varphi_\infty(\Lambda)$ is smooth, then it has to be a smooth Legendrian.

\begin{theorem}[\Cref{thm:main-Maslov} Part 1]\label{thm:rigidity}
    Let $M$ be a complete Riemannian manifold and $\Lambda \subseteq S^*M$ be a smooth Legendrian. Let $\varphi_n \in \operatorname{Cont}(S^*M, \xi_{\mathrm{std}})$ be contactomorphisms, each of which has bounded conformal factor $h_n$. Suppose $\varphi_n \to \varphi_\infty$ in the $C^0$-topology, $\varphi_\infty$ is a homeomorphism and $\varphi_\infty(\Lambda)$ is smooth. Then $\varphi_\infty(\Lambda)$ is a smooth Legendrian.
    Moreover, if $\varphi_n \in \operatorname{Cont}_0(S^*M, \xi_{\mathrm{std}})$ for any $n$, $\Lambda$ and $\varphi_\infty(\Lambda)$ have the same Maslov class: $\mu(\Lambda) = \varphi_\infty^*\mu(\varphi_\infty(\Lambda))$.
\end{theorem}
\begin{proof}
    Let $\epsilon>0$ be sufficiently small so that \cref{thm:quantization-C0-small} holds. Without loss of generality, we may assume that there exists $N$ such that $d_{C^0}(\varphi_n,\varphi_\infty)<\epsilon/4$ for $n \ge N$. Then we know that there exist sheaf quantizations $K^{\varphi_n \circ \varphi_N^{-1}}$ that converges to $K^{\varphi_\infty \circ \varphi_N^{-1}}$, and by \Cref{prop:quantization-micro}, that $\varphi_\infty \circ \varphi_N^{-1}$ induces an equivalence between sheaves of categories
    \[
        (\varphi_{\infty} \circ \varphi_N^{-1})_* \msh_{\varphi_N(\Lambda)}
        \simeq 
        \msh_{\varphi_\infty(\Lambda)}.
    \]
    
    Let the coefficients be $\Mod(\bZ/2\bZ)_{/[1]}$. 
    Then $\msh_{\varphi_N(\Lambda)}(\varphi_N(\Lambda))$ admits a global rank 1 object by \Cref{thm:musheaf}, and thus $\msh_{\varphi_\infty(\Lambda)}(\varphi_\infty(\Lambda))$ also admits a global rank 1 object.
    Then, by the coisotropicity theorem of singular supports \cite[Theorem 6.5.4]{KS90}, we know that the support of the rank 1 object $\varphi_\infty(\Lambda)$ is (cone) coisotropic in the sense of \cite[Definition 6.5.1]{KS90}. Since it is a smooth submanifold, we can conclude that it is a Legendrian submanifold.

    For the assertion for the Maslov class, we apply the construction in \Cref{thm:sheaf-invariance-summary}, which implies that when $\varphi_n \in \mathrm{Cont}_0(S^*M, \xi_\text{std})$, we have an equivalence of sheaves of categories
    \[
        \msh_{\varphi_\infty(\Lambda)} \simeq \varphi_{\infty*}\msh_{\Lambda}.
    \]
    We note that Guillermou's result \Cref{thm:musheaf} \cite[Section 10.3 \& 10.6]{Guillermou23} for $\Mod(\bZ/2\bZ)$ shows that $\msh_{\Lambda}$ is a locally constant sheaf of categories twisted by the Maslov class $\mu(\Lambda) \in H^1(\Lambda; \bZ)$ via degree shifting.
    This implies that $\msh_{\varphi_\infty(\Lambda)} \simeq \varphi_{\infty*} \msh_{\Lambda}$ is a locally constant sheaf of categories twisted by the Maslov class $\varphi_{\infty*}\mu(\Lambda) \in H^1(\varphi_\infty(\Lambda); \bZ)$. 
    Hence we get $\mu(\varphi_{\infty}(\Lambda)) = \varphi_{\infty*}\mu(\Lambda)$.
\end{proof}

\begin{remark}
    In the situation of \cref{thm:rigidity}, if $\Lambda$ is closed, then we can prove that $\Lambda$ and $\varphi_\infty(\Lambda)$ have the same relative Stiefel--Whitney class as follows.
    Since $\varphi_\infty|_{\Lambda} \colon \Lambda \to \varphi_\infty(\Lambda)$ is a homeomorphism, each Wu class $v_k(\Lambda)$ is preserved by the homeomorphism $\varphi_\infty$. Then by the formula $w_2(\Lambda)=v_2(\Lambda)+v_1(\Lambda)^2$, we obtain the desired equality. 
\end{remark}

We are now able to prove the invariance of Maslov data of any embedded Legendrian under contact homeomorphisms in cosphere bundles \Cref{thm:maslov-J} (\Cref{thm:main-Maslov}). For readers that prefer to work with the classical dg categories (or stable $\infty$-categories over a discrete ring), we note that the following theorem is the only part in the main body of the paper that depends on Jin's result \Cref{thm:musheaf} over the sphere spectrum. \footnote{\cref{thm:Hausdorff-cgcat,thm:Hasusdorff-ringspec} in the appendix also depend on Jin's result.}

\begin{theorem}[\Cref{thm:main-Maslov} Part 2]\label{thm:maslov-J}
    Let $M$ be a complete Riemannian manifold and $\Lambda \subseteq S^*M$ be a smooth Legendrian. Let $\varphi_n \in \operatorname{Cont}_0(S^*M, \xi_{\mathrm{std}})$ be contactomorphisms, each of which has bounded conformal factor $h_n$. Suppose $\varphi_n \to \varphi_\infty$ in the $C^0$-topology, $\varphi_\infty$ is a homeomorphism and $\varphi_\infty(\Lambda)$ is smooth. Then when $\varphi_\infty(\Lambda)$ is smooth (and hence Legendrian), the compositions of the Lagrangian Gauss map and the $J$-homomorphism are homotopically commutative via $\varphi_\infty$:
    \[\begin{tikzcd}[row sep=scriptsize]
    \Lambda \ar[dr,dashed] \ar[drr] \ar[dd, "\varphi_\infty"] & & \\
    & U/O \ar[r,dashed] & B\mathrm{Pic}(\mathbb{S}) \\
    \varphi_\infty(\Lambda). \ar[ur,dashed] \ar[urr] & &
    \end{tikzcd}\]
\end{theorem}
\begin{proof}
    Since $\varphi_n \to \varphi_\infty$ in the uniform topology, by \Cref{thm:quantization-C0}, we know that there exist sheaf quantizations $K^{\varphi_n}$ that converges to $K^{\varphi_\infty}$, and by \Cref{prop:quantization-micro}, we know that induces an equivalence between sheaves of categories
    \[
        \msh_{\varphi_\infty(\Lambda)}(\varphi_\infty(\Lambda)) \simeq \msh_{\Lambda}(\Lambda).
    \]
    Then, by Jin's result \Cref{thm:musheaf} \cite[Theorem 1.1]{Jin20Jhomomorphism}, we know that for the coefficient $\Mod(\bS)$ the sheaf of categories $\msh_{\varphi_\infty(\Lambda)}$ is classified by the composition of the Lagrangian Gauss map and the $J$-homomorphism:
    \[
        \Lambda \to U/O \to B\mathrm{Pic}(\mathbb{S}).
    \]
    This implies that the map $\varphi_\infty(\Lambda) \to U/O \to B\mathrm{Pic}(\mathbb{S})$ is homotopic to $\Lambda \to U/O \to B\mathrm{Pic}(\mathbb{S})$.
\end{proof}

Finally, combining our sheaf quantization theorem with the  Guillermou--Viterbo $\gamma$-coisotropicity theorem \cite[Theorem 1.2]{GV24}, we conclude the $C^0$-rigidity of coisotropic submanifolds. We remark that while the last part relies on Floer-theoretic spectral invariants, it is also possible to reprove the results using sheaves.

We follow the definition of coisotropic submanifolds in contact manifolds by Huang \cite{Huang15}, which is also used in \cite{RosenZhang20Dichotomy,Usher21Conformal}. In particular, by \cite[Proposition 4.1]{RosenZhang20Dichotomy}, for a contact manifold $(Y, \xi)$, $C \subseteq Y$ is coisotropic if and only if in the symplectization, $C \times \bR_{>0} \subseteq Y \times \bR_{>0}$ is symplectic coisotropic. 

The main criterion of (symplectic) coisotropicity we will use is as follows, based on the $\gamma$-coisotropicity introduced by Viterbo \cite{Viterbo22} and closely related to the local rigidity of Usher \cite{Usher22}:

\begin{lemma}[Usher {\cite[Theorem 2.1]{Usher22}}]\label{lem: coisotropicity}
    Let $(X, \omega)$ be a symplectic manifold and $C \subseteq X$ be a smooth submanifold. Then $C$ is coisotropic if there is an open dense subset $C_0 \subseteq C$ such that for every $x \in C_0$, there is a $\gamma$-coisotropic subset $L_x \subseteq C_0$ with $x \in L_x$.
\end{lemma}
\begin{proof}
    We argue by contrapositive. If $C \subseteq X$ is not coisotropic at $x \in C$, then there is an open neighborhood $V \subseteq C$ that is nowhere coisotropic. By \cite[Theorem 2.1]{Usher22}, we know that $V \subseteq C$ is nowhere locally rigid in the sense of \cite[Definition 1.1]{Usher22}, that is, the infimum of the Hofer norm of Hamiltonians to displace a neighborhood of $p$ in $C$ from a small open ball around $p$ is zero. Then, since the $\gamma$-metric is bounded by the Hofer metric \cite[Proposition 7.5~(4)]{Viterbo22}, we can conclude that $V \subseteq C$ is nowhere $\gamma$-coisotropic, that is, the infimum of the spectral norm of Hamiltonians to displace a neighborhood of $p$ in $C$ from a small open ball around $p$ is zero.
\end{proof}

\begin{theorem}[\Cref{thm:main-rigidity}]\label{thm:coiso-rigidity}
    Let $M$ be a complete Riemannian manifold and $C \subseteq S^*M$ be a (locally closed) smooth coisotropic submanifold. Let $\varphi_n \in \operatorname{Cont}(S^*M, \xi_{\mathrm{std}})$ be contactomorphisms, each of which has bounded conformal factor $h_n$. Suppose $\varphi_n \to \varphi_\infty$ in the $C^0$-topology, $\varphi_\infty$ is a homeomorphism and $\varphi_\infty(C)$ is smooth. Then $\varphi_\infty(C)$ is also coisotropic. 
\end{theorem}

\begin{proof}
    First, by \cite[Corollary 4.10]{Usher21Conformal}, we know there is an open dense subset $C_0 \subseteq C$ such that for each $x \in C_0$, locally there exists a smooth Legendrian $\Lambda_x \subseteq C$ with $x \in \Lambda$. Then by \Cref{prop:quantization-micro}, 
    for a sufficiently large $N$, there is an equivalence between sheaves of categories
    \[
        K^{\varphi_\infty \circ \varphi_N^{-1}} \colon (\varphi_\infty \circ \varphi_N^{-1})_* \msh_{\varphi_N(\Lambda_x)} \to \msh_{\varphi_\infty(\Lambda_x)}.
    \]
    In particular, since $\msh_{\varphi_N(\Lambda_x)}$ is a locally constant sheaf of categories, the stalks of $\msh_{\varphi_N(\Lambda_x)}$ are nowhere zero,
    and thus the stalks of $\msh_{\varphi_\infty(\Lambda_x)}$ are nowhere zero. Then, by the $\gamma$-coisotropicity theorem \cite[Theorem 1.2]{GV24}, we know that the conic subset $\varphi_\infty(\Lambda_x) \times \bR_{>0}$ is $\gamma$-coisotropic in $T^*M$. Then by \Cref{lem: coisotropicity} we know that $\varphi_\infty(C) \times \bR_{>0}$ is a coisotropic submanifold in $T^*M$. By \cite[Proposition 4.1]{RosenZhang20Dichotomy}, this implies $\varphi_\infty(C)$ is a coisotropic submanifold in $T^*M$.
\end{proof}

\begin{remark}
    The above result does not say anything about convex hypersurfaces \cite{Giroux}. As noted in \cite[Example 3.4]{RosenZhang20Dichotomy}, any hypersurface in a contact manifold is a coisotropic submanifold, while not every hypersurface is a convex hypersurface (in fact, $C^0$-flexibility of convex surfaces is recently shown in \cite{SerrailleStokic25}). Neither do we say anything about rigidity of coisotropic submanifolds where $TC \cap \xi$ is of constant rank ($C^0$-rigidity of such type of coisotropic surfaces is recently shown in \cite{SerrailleStokic25}, where it is called regular coisotropic).
\end{remark}

\appendix

\section{Local \texorpdfstring{$C^0$}{C0}-/Hausdorff-Rigidity without Interleaving Distance}\label{sec:appendix}

In this section, we give a simpler proof of the local $C^0$-rigidity of Legendrians, without using interleaving distances. We also show a result on rigidity of certain Hausdorff limits of Legendrians. 

\begin{proposition}\label{prop:local-rigid}
    Let $M$ be a smooth manifold, and $\varphi_n \in \mathrm{Cont}_0(S^*M, \xi_{\mathrm{std}})$ be a sequence of contactomorphism that are equal to the identity on a fixed open subset $\Omega \subseteq S^*M$. Suppose $\varphi_n \rightarrow \varphi_\infty$. Then for any connected smooth properly embedded Legendrian $\Lambda \subseteq S^*M$ such that $\Lambda \cap \Omega \neq \varnothing$, there does not exist any open subset $\Lambda_0 \subseteq \Lambda$ such that $\varphi_\infty(\Lambda_0)$ is smooth but nowhere Legendrian.
\end{proposition}
\begin{proof}
    For some open subset $\Lambda_0 \subseteq \Lambda$, assume that $\varphi_\infty(\Lambda_0)$ is smooth but nowhere Legendrian. Let the coefficients be $\Mod(\bZ/2\bZ)_{/[1]}$. We begin with considering the doubling construction \Cref{thm:doubling}. The choice of coefficients ensures $\msh_{\Lambda}(\Lambda) = \Loc(\Lambda)$ so there is a microlocal rank 1 sheaf $F \in \Sh_{\Lambda \cup \Lambda^\tau_\epsilon}(M \times \bR)$. Let $K^{\varphi_n}\in \Sh(M \times M)$ be the sheaf quantizations of $\varphi_n$ by \Cref{thm: GKS}. Consider the sheaf
    \[F_\infty \coloneqq \mathrm{Fib}\Big(\bigoplus_{n\in \bN}K^{\varphi_n} \circ F \to \prod_{n\in \bN}K^{\varphi_n} \circ F \Big).\]
    Since $\cSS^\infty(F) = \Lambda \cup \Lambda^\tau_\epsilon$, by the singular support estimation of direct sums and products \cite[Proposition 3.4]{GV24}, we know that $\cSS^\infty(F_\infty) \subseteq \varphi_\infty(\Lambda \cup \Lambda^\tau_\epsilon)$. Since $\varphi_\infty(\Lambda_0)$ is smooth but nowhere Legendrian, by the coisotropicity theorem \cite[Theorem 6.5.4]{KS90}, we know that in fact $\cSS^\infty(F_\infty) \subseteq \varphi_\infty(\Lambda \cup \Lambda^\tau_\epsilon \setminus \Lambda_0)$. Moreover, since the contact isotopies are all supported away from $\Omega$, by \cite[Theorem 7.5.11]{KS90} we know that the microstalk of $F_\infty$ in $\Lambda \cap \Omega$ is just $\mathrm{Fib}(\bigoplus_{n\in \bN}1 \to \prod_{n\in \bN}1).$
    
    Given that $\varphi_n \to \varphi_\infty$ in the $C^0$-topology, we know that $\varphi_n^{-1} \to \varphi_\infty^{-1}$ in the $C^0$-topology. Then consider the sheaf
    \[G_\infty = \mathrm{Fib}\Big(\bigoplus_{n\in \bN}K^{\varphi_n^{-1}} \circ F_\infty \to \prod_{n\in \bN}K^{\varphi_n^{-1}} \circ F_\infty \Big).\]
    Since $\cSS^\infty(F_\infty) \subseteq \varphi_\infty(\Lambda \cup \Lambda^\tau_\epsilon \setminus \Lambda_0)$, by the singular support estimation of direct sums and products \cite[Proposition 3.4]{GV24}, we know that $\cSS^\infty(G_\infty) \subseteq \Lambda \cup \Lambda^\tau_\epsilon \setminus \Lambda_0$. This means that the microstalk of $G_\infty$ along $\Lambda_0$ is zero. Moreover, since the contact isotopies are all supported away from $\Omega$, by \cite[Theorem 7.5.11]{KS90}, the microstalk of $G_\infty$ in $\Lambda \cap \Omega$ is $\mathrm{Fib}(\bigoplus_{n\in \bN}V \to \prod_{n\in \bN}V)$ where $V = \mathrm{Fib}(\bigoplus_{n\in \bN}1 \to \prod_{n\in \bN}1).$ 
    
    Since $\Lambda$ is a connected smooth Legendrian, by isotopy extension theorem, there exists a contact isotopy that sends a point in $\Lambda$ to any other point in $\Lambda$. By the invariance of microlocalization under contactomorphisms \cite[Theorem 7.2.1]{KS90}, we know that the microstalk along $\Lambda$ must be locally constant. This leads to a contradiction.
\end{proof}

We can strengthen the above result to show the local rigidity of $C^0$-limits of Legendrians:

\begin{theorem}
    Let $M$ be a smooth manifold, and $\varphi_n \in \mathrm{Cont}_0(S^*M, \xi_{\mathrm{std}})$ be a sequence of contactomorphism that are equal to the identity on a fixed open subset $\Omega \subseteq S^*M$. Suppose $\varphi_n \rightarrow \varphi_\infty$. Then for any connected smooth properly embedded Legendrian $\Lambda \subseteq S^*M$ such that $\Lambda \cap \Omega \neq \varnothing$, $\varphi_\infty(\Lambda)$ is Legendrian if it is smooth. Moreover, if the Maslov class vanishes $\mu(\Lambda) = 0$, then we also have $\mu(\varphi_\infty(\Lambda)) = 0$.
\end{theorem}
\begin{proof}
    Consider the same sheaf $F_\infty = \mathrm{Fib}(\bigoplus_{n\in \bN}K^{\varphi_n} \circ F \to \prod_{n\in \bN}K^{\varphi_n} \circ F )$. By \Cref{prop:local-rigid}, we know that there does not exist any open subset $\Lambda_0 \subseteq \Lambda$, such that 
    $\cSSif(F_\infty) \cap \varphi_\infty(\Lambda_0) = \varnothing$.
    Hence, $\cSSif(F_\infty) \subseteq \varphi_\infty(\Lambda)$ is dense. Since the singular support is a closed subset, we can conclude that $\cSSif(F_\infty) = \varphi_\infty(\Lambda)$. By the coisotropicity theorem \cite[Theorem 6.5.4]{KS90}, we know that $\varphi_\infty(\Lambda)$ is cone coisotropic. Since it is smooth, we can conclude that it is Legendrian.
    
    Finally, suppose the Maslov class $\mu(\Lambda) = 0$. Let the coefficients be $\Mod(\bZ/2\bZ)$. The choice of coefficients ensures $\msh_{\Lambda}(\Lambda) = \Loc(\Lambda)$ so there is a microlocal rank 1 sheaf $F \in \Sh_{\Lambda \cup \Lambda^\tau_\epsilon}(M \times \bR)$. Since $\cSSif(F_\infty) = \varphi_\infty(\Lambda)$, the microstalk of $F_\infty$ is $\mathrm{Fib}(\bigoplus_{n\in \bN}1 \to \prod_{n\in \bN}1)$, which is a bounded complex. Then we can conclude that $\mu(\varphi_\infty(\Lambda)) = 0$ by \Cref{thm:musheaf}.
\end{proof}

More importantly, we can apply the above argument to local Hausdorff limits of Legendrians, where the result in the main body of the paper does not apply. 
The main observation is the following property of cone coisotropicity \cite[Definition 6.5.1]{KS90}.

\begin{lemma}\label{lem:coisotropic-boundary}
    Let $\Lambda \subseteq S^*M$ be a connected properly embedded smooth submanifold with $\dim \Lambda = \dim M - 1$. Suppose $\Lambda_0 \subseteq \Lambda$ is a closed cone coisotropic subset and $\Lambda_0$ contains a non-empty open subset of $\Lambda$. Then $\Lambda_0 = \Lambda$ and thus $\Lambda$ is a properly embedded Legendrian.
\end{lemma}
\begin{proof}
    Consider a smooth path $\gamma$ from $p$ to $q$ in the submanifold $\Lambda$. Suppose $p \in \Lambda_0$. We show that $q \in \Lambda_0$ as well. We follow the notations in \cite[Section 8]{GV24}. First, note that the contingent cone and paratingent cone $C_p^-(\Lambda_0) \subseteq C_p^+(\Lambda_0) \subseteq T_p\Lambda$. Hence for any hypersurface $H$ with $C_p^-(\Lambda_0) \subseteq T_pH$, we know that $T_p\Lambda \subseteq T_pH$. Otherwise, $C_p^-(\Lambda_0)$ is contained in an isotropic subspace $T_pH \cap T_p\Lambda$, which is a contradiction. Now, consider the hypersurface $H$ with $T\Lambda \subseteq TH$ whose characteristic curve passing through $p$ is $\gamma$. Then \cite[Lemma 6.5.3]{KS90} implies that $\gamma \subseteq \Lambda_0$. This completes the proof.
\end{proof}

By the $h$-principle for (stabilized or) loose Legendrians \cite{CE12,MurphyLoose}, for any submanifold $N \subseteq S^*M$ of $\dim N = \dim M-1$, there exists a sequence of isotopic (stabilized or) loose Legendrians $\Lambda_n \subseteq S^*M$ such that the Hausdorff limit of $(\Lambda_n)_{n\in \bN}$ is $N$. However, when there exist sheaves with singular support in $\Lambda_n \subseteq S^*M$, we can still show a rigidity result (this also aligns with the result in \cite[Theorem 1.7]{DRS20Persistence} that Legendrians equipped with augmentations cannot be squeezed into tubular neighborhoods of stabilized or loose Legendrians).

\begin{theorem}
    Let $M$ be a smooth manifold and $\Lambda_n \subseteq S^*M$ be a sequence of connected smooth properly embedded Legendrians that converges to a smooth manifold $\Lambda_\infty \subseteq S^*M$ in the sense of Hausdorff convergence. Assume that 
    \begin{enumerate}
    \item there exists an open subset $\Omega \subseteq S^*M$ such that $\Lambda_n \cap \Omega$ is independent of $n$ and non-empty,
    \item there exists a microlocal rank 1 sheaf $F_n \in \Sh_{\Lambda_n}(M;\bZ/2\bZ)$ for each $n \in \bN$. 
    \end{enumerate}
    Then there exists a sheaf $F_\infty \in \Sh_{\Lambda_\infty}(M)$ with non-trivial bounded microstalk. In particular, $\Lambda_\infty$ is a Legendrian and the Maslov class $\mu(\Lambda_\infty) = 0$.
\end{theorem}
\begin{proof}
    Consider the sheaf
    $F_\infty = \mathrm{Fib}(\bigoplus_{n\in \bN}F_n \to \prod_{n\in \bN}F_n)$.
    Then by the singular support estimation \cite[Proposition 3.4]{GV24}, we know that $\cSS^\infty(F_\infty) \subseteq \Lambda_\infty$. Since $\Lambda_n \cap \Omega$ are identical and non-empty, we know that the microstalk of $F_\infty$ along $\Lambda_\infty \cap \Omega$ is $\mathrm{Fib}(\bigoplus_{n\in \bN}1 \to \prod_{n\in \bN}1)$. Then by \Cref{lem:coisotropic-boundary}, we know that $\cSS^\infty(F) = \Lambda_\infty$ and thus $\Lambda_\infty$ is a smooth Legendrian. Then by Guillermou's result \Cref{thm:musheaf} \cite[Section 10.3 \& 10.6]{Guillermou23}, we can conclude that the Maslov class $\mu(\Lambda_\infty) = 0$.
\end{proof}

The sheaf quantization we considered above has microstalk of infinite rank. Thus, it cannot show that local $C^0$-limits or local Hausdorff limits preserve Maslov data over general coefficients. 
This can be fixed via the construction of ultraproducts.
Below we use the ultraproducts in categorical settings following \cite{BSS2020asympalg}. 
First, let us recall the definition of ultrafilter on non-negative integers $\bN$.
We write $\cP(\bN)$ for the power set of $\bN$.

\begin{definition}
    A non-empty subset $\cU$ of $\cP(\bN)$ is said to be an \emph{ultrafilter} on $\bN$ if it satisfies the following:
    \begin{enumerate}
        \item $\varnothing \notin \cU$;
        \item for $A, B\subseteq \bN$, if $A\in \cU$ and $A\subseteq B$, then $B\in \cU$;
        \item for $A, B\subseteq \bN$, if $A, B\in \cU$, then $A\cap B\in \cU$;
        \item for $A\subseteq \bN$, either $A\in\cU$ or $\bN\setminus A\in \cU$. 
    \end{enumerate}
\end{definition}

For example, for each $n\in \bN$, $\{A\subseteq \bN\mid n\in A\}$ is a ultrafilter. 
An ultrafilter of this form is said to be principal. Otherwise, it is said to be non-principal. 
Any non-principal ultrafilter contains every cofinite subset in $\bN$ (a subset whose complement if finite).

\begin{definition}
     Let $\cU \subseteq \cP(\bN)$ be an ultrafilter and $\cD$ be a category which admits products and filtered colimits. 
     For a sequence $(d_n)_{n\in \bN}$ of objects of $\cD$, the \emph{ultraproduct} of $(d_n)_{n\in \bN}$ is defined to be the object
     \[ {\prod}_{\cU}d_n\coloneqq \clmi{A\in \cU}\prod_{n\in A} d_n. \]
     For an object $d$ of $\cD$, the \emph{ultrapower} $d^\cU$ of $d$ is defined to be the ultraproduct ${\prod}_{\cU}d$ of the constant sequence $(d)_{n\in \bN}$.
\end{definition}

\begin{remark}
    The set $\beta\bN$ of ultrafilters on $\bN$ admits a topology with open basis consisting of the subsets of the form $[A]\coloneqq \{\cU\in \beta \bN \mid A\in \cU\}$ for every $A\subseteq \bN$. 
    The topological space $\beta\bN$ is compact Hausdorff and the inclusion $j \colon \bN\to \beta\bN, n\mapsto \{A\subseteq \bN\mid n\in A\}$ gives the Stone--\v{C}ech compactification of $\bN$. 

    Let $\cD$ be a category which admits limits and filtered colimits. 
    For an ultrafilter $\cU$, let $i_{\cU}\colon \mathrm{pt} \to \beta\bN$ be the inclusion of a point to the point $\cU\in \beta\bN$. 
    A sequence $(d_n)_{n\in \bN}$ of objects of $\cD$ can be regarded as an object of $\Sh (\bN; \cD)$. 
    Then the ultraproduct ${\prod}_{\cU}d_n$ is obtained by 
    \[{\prod}_{\cU}d_n \simeq i_{\cU}^*j_*(d_n)_{n\in \bN}.\] 
\end{remark}

The following lemma follows from point set topology and the definition of singular supports:

\begin{lemma}\label{lem: ss-ultraproduct}
    For a sequence $(G_n)_{n\in \bN}$ in $\Sh(M)$ and any ultrafilter $\cU \in \beta\bN$, one has
    \[
        \cSS\left({\prod}_\cU G_n\right)\subseteq \bigcap_{A\in \cU}\overline{\bigcup_{n\in A}\cSS(G_n)}.
    \]
\end{lemma}
\begin{proof}
    For any closed subset $X \subset S^*M$, the subcategory $\Sh_X(M)$ is closed under limits and colimits in $\Sh (X)$ by \cite[Proposition 3.4]{GV24}. For any $A\in \cU$, ${\prod}_\cU G_n \simeq \colim_{A\supseteq B\in \cU }\prod_{n\in B} G_n$, and hence $\cSS\left({\prod}_\cU G_n\right)\subseteq \overline{\bigcup_{n\in A}\cSS(G_n)}$.
\end{proof}

Now, we apply the ultraproduct of sheaf quantizations to local Hausdorff limits of Legendrians. This is also beyond the reach of the technique in the main body of the paper. 

\begin{theorem}\label{thm:Hasusdorff-ringspec}
    Let $M$ be a smooth manifold, and $(\Lambda_n \subseteq S^*M)_{n\in \bN}$ be a sequence of connected smooth properly embedded Legendrians that converges to a smooth manifold $\Lambda_\infty \subseteq S^*M$ in the sense of Hausdorff convergence. 
    Let $R$ be an $E_\infty$-ring spectrum. 
    Assume that 
    \begin{enumerate}
        \item there exists an open subset $\Omega \subseteq S^*M$ such that $\Lambda_n \cap \Omega$ is independent of $i$ and non-empty,
        \item there exists a microlocal rank 1 sheaf $F_n \in \Sh_{\Lambda_n}(M)$ over $R$ for each $n\in \bN$. 
    \end{enumerate}
    Then there exists a microlocal rank 1 sheaf $F_\infty \in \Sh_{\Lambda_\infty}(M)$ over $R^\cU$. 
    In particular, $\Lambda_\infty$ is a Legendrian. 
    Moreover if $\Lambda_\infty$ is homotopy equivalent to a finite CW-complex, then the composite $\Lambda_\infty\to U/O\to B\mathrm{Pic}(\bS)\to B\mathrm{Pic}(R)$ is null-homotopic. 
\end{theorem}
\begin{proof}
    Note that the constant sheaf $1_M$ is an $E_\infty$-algebra object in $\Sh (M)$ and any object of $\Sh (M)$ admits a unique $1_M$-module structure. 
    Then $1_M^{\cU}\in \Sh (M)$, the ultrapower of $1_M$, is an $E_\infty$-algebra object and $F_\infty\coloneqq {\prod}_{\cU} F_n$ is equipped with a $1_M^{\cU}$-module structure. 
    Since $1^{\cU}_M \simeq (1^\cU)_M$, $F_\infty$ can be regarded as an object of $\Sh (M; R^\cU)$. By the singular support estimation \Cref{lem: ss-ultraproduct},
    $\cSS^\infty(F_\infty) \subseteq \Lambda_\infty$. 
    Since $\Lambda_n \cap \Omega$ are identical and non-empty, and the microstalks of $F_n$ are $R$, we know that the microstalk of $F_\infty$ along $\Lambda_\infty \cap \Omega$ is $R^\cU$. Therefore, by \Cref{lem:coisotropic-boundary}, we know that $\cSS^\infty(F) = \Lambda$ and $\Lambda$ is a smooth Legendrian; moreover, 
    as an object of $\Sh_{\Lambda_\infty}(M; R^\cU)$, $F_\infty$ is microlocal rank 1. Hence, by \Cref{thm:musheaf}, the composition $\Lambda_\infty\to U/O \to B\mathrm{Pic}(\bS) \to B\mathrm{Pic}(R^\cU)$ is null-homotopic. 
    
    Consider the natural morphism $B\mathrm{Pic}(R^\cU)\to B\mathrm{Pic}(R)^\cU$. 
    Then the further composition 
    $\Lambda_\infty \to U/O \to B\mathrm{Pic}(\bS) \to B\mathrm{Pic}(R^\cU)\to B\mathrm{Pic}(R)^\cU$
    is also null-homotopic. 
    Assume that $\Lambda_\infty$ is homotopy equivalent to a finite CW-complex, and then it is a compact object in the category of the spaces. Therefore, 
    $\mathrm{Map}(\Lambda_\infty, B\mathrm{Pic}(R)^\cU)\simeq \mathrm{Map}(\Lambda_\infty, B\mathrm{Pic}(R))^\cU$.
    Since the mapping class in $\mathrm{Map}(\Lambda_\infty, B\mathrm{Pic}(R))^\cU$ that corresponds to the map $\Lambda_\infty \to B\mathrm{Pic}(R)^\cU$ factors through $U/O$ and $B\mathrm{Pic}(\bS)$, it is given by the constant sequence $(\Lambda_\infty \to U/O \to B\mathrm{Pic}(\bS) \to B\mathrm{Pic}(R))_{n\in \bN}$. 
    This element is trivial if and only if the map $\Lambda_\infty\to U/O \to B\mathrm{Pic}(\bS) \to B\mathrm{Pic}(R)$ is null-homotopic. 
\end{proof}

\begin{remark}
    There is a natural morphism $\mathrm{Pic}(R^\cU)\to \mathrm{Pic}(R)^\cU$ of $\infty$-groups, which is used in the proof above. 
    In \cite{BSS2020asympalg}, two (different) categories ${\prod}_{\cU}^{\flat\flat}\mathrm{Mod}(R)$ and ${\prod}_\cU^\omega \mathrm{Mod}(R) $ are defined. By \cite[Theorem 3.63]{BSS2020asympalg}, the first one is ${\prod}_{\cU}^{\flat\flat}\mathrm{Mod}(R)\simeq \mathrm{Mod}(R^\cU)$, and the latter will be defined and used below. 
    There exists a fully faithful symmetric monoidal functor ${\prod}_{\cU}^{\flat\flat}\mathrm{Mod}(R)\to {\prod}_\cU^\omega \mathrm{Mod}(R) $ and it induces a morphism between their Picard groupoids. 
    By \cite[Proposition 2.2.3]{MS2016pictmf} and the definition of ${\prod}_\cU^\omega$, $\mathrm{Pic}({\prod}_\cU^\omega \mathrm{Mod}(R)) \simeq \mathrm{Pic}(R)^\cU$. This gives the natural morphism we want.
    
    The morphism $\mathrm{Pic}(R^\cU)\to \mathrm{Pic}(R)^\cU$ is not an equivalence in general. 
    Let $R$ be a connective $E_\infty$-ring spectrum. Then $[(R[n])_{n\in \bN}]$ is a point of $ \mathrm{Pic}(R)^\cU$ that is not contained in the image of the morphism since ${\prod}_\cU R[n]\simeq 0\in \mathrm{Mod}(R)$. 
\end{remark}

For $(\cC_n)_n$ be a sequence of categories, the ultraproduct is defined by 
\[{\prod}_{\cU} \cC_n \coloneqq \clmi{A\in \cU} \prod_{n\in A} \cC_n,\]
where the products and the filtered colimit are taken in the category of categories. 

Let $(\cC_n)_n$ be a sequence of compactly generated categories. 
Define the \emph{compactly generated ultraproduct} of the sequence $(\cC_n)_n$ by
\[{\prod}^{\omega}_{\cU} \cC_n \coloneqq \clmi{A\in \cU}^{\omega} {\prod_{n\in A}}^{\omega} \cC_n,\]
where $\clmi{}^{\omega}$ and ${\prod}^{\omega}$ denote the filtered colimit and the product in the category $\mathrm{Pr}^L_\omega$ of compactly generated categories and functors which preserve colimits and compact objects. 
See \cite{BSS2020asympalg} for detailed arguments about the compactly generated ultraproduct.  
If each of $\cC_n$ is stable, then so is ${\prod}^{\omega}_{\cU} \cC_n$.
If each of $\cC_n$ is symmetric monoidal\footnote{In this paper, we require that the unit object is compact and that compact objects are closed under the monoidal operation. This will hold for compactly generated rigid symmetric monoidal categories.}
, then so is ${\prod}^{\omega}_{\cU} \cC_n$.

For a sequence $(\cC_n)_n$ of compactly generated categories. 
Take an arbitrary element $A\in \cU$. 
A functor $R \colon \prod_{n\in A} \cC_n\to \prod_{n\in A}^\omega \cC_n$ is defined as follows. 
For a category $\cC$, let $\cC^\omega$ be the full subcategory of $\cC$ consisting of compact objects. 
By the definition of ${\prod}^\omega_{n\in A}$, $\prod_{n\in A}^\omega \cC_n\simeq \mathrm{Ind}(\prod_{n\in A}\cC_n^\omega)$. 
Let us restrict the Yoneda embedding $ \prod_{n\in A} \cC_n \to \mathrm{Fun}((\prod_{n\in A} \cC_n)^{\mathrm{op}},\cS)$ along the inclusion $\prod_{n\in A} \cC_n^\omega \to \prod_{n\in A} \cC_n$, we obtain the functor $\prod_{n\in A} \cC_n \to \mathrm{Fun}((\prod_{n\in A} \cC_n^\omega)^{\mathrm{op}},\cS)$. 
The essential image of this functor is contained in $\mathrm{Ind}(\prod_{n\in A}\cC_n^\omega)$ and hence we obtain the functor $R \colon \prod_{n\in A} \cC_n\to \prod_{n\in A}^\omega \cC_n$. 
Note that $R$ is a morphism in $\mathrm{Pr}^R$ and hence it preserves limits. 
Let $L\colon \prod_{n\in A}^\omega \cC_n\to {\prod}_{\cU}^\omega \cC_n$ be the natural morphism in $\mathrm{Pr}^L_\omega$. 

\begin{lemma}
    Let $(\cC_n)_n$ be a sequence of compactly generated symmetric monoidal categories. 
    The functor $R$ and $L$ defined above are symmetric monoidal. 
\end{lemma}

Note that $\prod_{n\in A} \cC_n$ is presentable but not compactly generated in general. 
If a presentable category $\cD$ is not compactly generated, for objects in $\Sh (M;\cD)$, 
the original definition of the singular support is not well-behaved. 
As mentioned by Efimov~\cite{Efimov-K-theory} and Zhang~\cite{Zhang25remark}, the definition through $\Omega$-lenses behaves well and admits some other characterizations. 
In this paper, we also utilize the definition through $\Omega$-lenses for general coefficients. 

\begin{lemma}\label{lem:ultraproduct-R}
    Let $X$ be a smooth manifold and $R\colon \cD_0\to \cD_1$ be a morphism in $\mathrm{Pr}^R_\mathrm{st}$. 
    Then $R$ induces a functor $R_* \colon \Sh (X;\cD_0)\to \Sh (X;\cD_1)$ by $(R_*F)(U)\coloneqq R(F(U))$ for each $F\in \Sh (X;\cD_0)$ and open $U\subseteq M$. 
    For each $F\in \Sh (X;\cD_0)$, $\cSS(R_*F) \subseteq \cSS (F)$.
\end{lemma}
\begin{proof}
    We use the characterization of $\Omega$-lenses as in \Cref{lem: omega-lenses-test} for singular supports of sheaves over general coefficients \cite[Definition~2.2]{Zhang25remark}.
\end{proof}

\begin{lemma}\label{lem:ultraproduct-L}
    Let $X$ be a smooth manifold and $L\colon \cD_0\to \cD_1$ be a morphism in $\mathrm{Pr}^L_{\mathrm{st}}$. 
    Then $L$ induces a functor $L_* \colon \Sh (X;\cD_0)\to \Sh (X;\cD_1)$ by 
    $\Sh (X;\cD_0)\simeq \Sh (X)\otimes \cD_0\to \Sh (X)\otimes \cD_1\simeq \Sh (X;\cD_1)$. 
    For each $F\in \Sh (X;\cD_0)$, $\cSS(L_*F) \subseteq \cSS (F)$.
\end{lemma}
\begin{proof}
    We use the characterization of singular supports given in \cite[Remark.~4.23~2)]{Efimov-K-theory} or \cite[Proposition~2.6]{Zhang25remark}.
\end{proof}

\begin{theorem}\label{thm:Hausdorff-cgcat}
    Let $M$ be a smooth manifold, and $(\Lambda_n \subseteq S^*M)_{n\in \bN}$ be a sequence of connected smooth properly embedded Legendrians that converges to a smooth manifold $\Lambda_\infty \subseteq S^*M$ in the sense of Hausdorff convergence. 
    Let $(\cC_n)_{n\in \bN}$ be a sequence of compactly generated stable rigid symmetric monoidal categories. 
    Assume that 
    \begin{enumerate}
        \item there exists an open subset $\Omega \subseteq S^*M$ such that $\Lambda_n \cap \Omega$ is independent of $n$ and non-empty,
        \item there exists a microlocal rank 1 sheaf $F_n \in \Sh_{\Lambda_n}(M; \cC_n)$ for each $n \in \bN$. 
    \end{enumerate}
    Then there exists a microlocal rank 1 sheaf $F \in \Sh_{\Lambda_\infty}(M; {\prod}^{\omega}_{\cU} \cC_n)$. 
    In particular, $\Lambda_\infty$ is a Legendrian. 
    Moreover if $\Lambda_\infty$ is homotopy equivalent to a finite CW-complex, then there exists $N\in \bN$ such that the composite $\Lambda_\infty\to U/O\to B\mathrm{Pic}(\bS)\to B\mathrm{Pic}(\cC_n)$ is null-homotopic for each $n\geq N$. 
\end{theorem}
\begin{proof}
    For any $A \subseteq \bN$, we can define a sheaf $F_A\in \Sh (M;\prod_{n\in A} \cC_n)$ by $F_A(U) = (F_n(U))_{n\in A}\in \prod_{n\in A} \cC_n$. 
    Define $F_\infty \coloneqq (L_*\circ R_*)(F_A)$. 
    Note that $F_\infty $ is the sheafification of a presheaf $F_\infty^\mathrm{pre}$ that assigns  to each open $U\subseteq M$ the the functor $[(c_n)_{n\in A}]\mapsto {\prod}_\cU \Hom ( c_n , F_n(U))$ which is an object of $\mathrm{Ind}(\prod_{n\in A}\cC_n^\omega)$, where $c_n$ is a compact object of $\cC_n$. 
    By the singular support estimations for $L_*$ and $R_*$ in \Cref{lem:ultraproduct-R,lem:ultraproduct-L}, as we vary $A \in \cU$, we know that 
    \[
        \cSS^\infty (F_\infty)\subseteq \bigcap_{A\in \cU}\cSS^\infty(F_A) = \bigcap_{A\in \cU}\ol{\bigcup_{n\in A}\cSS^\infty(F_n)} = \Lambda_\infty.
    \]
    Since $\Lambda_n \cap \Omega$ is independent of $n \in \bN$, the sheaf $F_n$ is simple along $\Lambda_n \cap \Omega$, and $L$ and $R$ are symmetric monoidal functors, we can conclude that $F_\infty$ is simple along $\Lambda_\infty \cap \Omega$. 
    After a perturbation by an ambient contact isotopy (and taking $\Omega$ smaller), we may assume $\Lambda_\infty \cap \Omega\to M$ is embedding and $\Lambda_\infty \cap \pi^{-1}(\pi (\Omega))=\Lambda_\infty \cap \Omega$. Then the microstalk at $\Lambda_\infty$ is identified with a cone of a map between stalks. The stalks are preserved by a colimit preserving functor. Hence we obtain the statement for $L_*$.
    Then by \Cref{lem:coisotropic-boundary}, we know that $\cSS^\infty(F_\infty) = \Lambda_\infty$ and $\Lambda_\infty$ is a Legendrian; moreover, $F_\infty$ is simple along $\Lambda_\infty$. Hence, by \Cref{thm:musheaf}, $\Lambda_\infty \to U/O \to B\mathrm{Pic}({\prod}^{\omega}_{\cU} \cC_n)$ is null-homotopic.
    
    When $\Lambda_\infty$ is homotopy equivalent to a finite CW-complex, we show that there is some $N \in \bN$ such that $\Lambda_\infty \to U/O \to B\mathrm{Pic}(\cC_n)$ is null homotopic for $n \geq N$. Otherwise, the cardinality of the set $A_0\coloneqq \{n\in \bN \mid \text{$\Lambda_\infty \to B\mathrm{Pic}(\cC_n)$ is not null-homotopic} \}$ is infinite. 
    Choose a non-principal ultrafilter $\cU$ so that $A_0\in \cU$. 
    Note that $\mathrm{Pic}({\prod}^{\omega}_{\cU} \cC_n)\simeq {\prod}_{\cU}  \mathrm{Pic}(\cC_n)$ by \cite[Proposition 2.2.3]{MS2016pictmf}. Note also that $\mathrm{Pic}(\cC_n)\simeq \mathrm{Pic}(\cC_n^\omega)$ since every invertible object in $\cC_n$ is compact. Since $\Lambda_\infty$ is a compact object in the category of spaces, we have
    $\mathrm{Map}(\Lambda_\infty , B\mathrm{Pic}({\prod}^{\omega}_{\cU} \cC_n))\simeq \mathrm{Map}(\Lambda_\infty, {\prod}_{\cU} B\mathrm{Pic}(\cC_n))\simeq {\prod}_{\cU} \mathrm{Map}(\Lambda_\infty , B\mathrm{Pic}(\cC_n))$.
    We know the map $\Lambda_\infty \to B\mathrm{Pic}(\bS)\to B\mathrm{Pic}({\prod}^{\omega}_{\cU}\cC_n)$ that classifies the locally constant sheaf of microsheaves is trivial. However, by our assumption, for any $n \in A_0$, the composition $\Lambda_\infty \to B\mathrm{Pic}(\cC_n)$ is non-trivial. Since $A_0 \in \cU$, we know that the corresponding map in ${\prod}_{\cU} \mathrm{Map}(\Lambda_\infty , B\mathrm{Pic}(\cC_n))$ is also non-trivial. This leads to a contradiction.
\end{proof}

\printbibliography

\end{document}